\documentclass[a4paper, 11pt]{article}
\usepackage{amsmath, amssymb, amsthm, mathrsfs}
\usepackage{ascmac,bm}
\usepackage{multirow}
\usepackage{graphicx,color}
\usepackage[colorlinks=true]{hyperref}

\newtheorem{thm}{Theorem}[section]
\newtheorem{prop}[thm]{Proposition}
\newtheorem{lem}[thm]{Lemma}
\newtheorem{cor}[thm]{Corollary}
\newtheorem{rem}[thm]{Remark}
\newtheorem{defn}{Definition}[section]

\numberwithin{equation}{section}

\allowdisplaybreaks[4]

\title{Unconditional uniqueness for the derivative nonlinear Schr\"odinger equation by normal form approach}

\author{Nobu Kishimoto%
\footnote{Research Institute for Mathematical Sciences, Kyoto University, Kitashirakawa Oiwake-cho, Sakyo-ku, Kyoto 606-8502, Japan.\quad {\it E-mail address}. \texttt{nobu@kurims.kyoto-u.ac.jp}\\
2020 {\it Mathematics Subject Classification.} 35Q55, 35A02.\\
{\it Keywords and phrases.} Derivative nonlinear Schr\"odinger equation, Unconditional uniqueness, Normal form reduction.}%
}

\date{}

\begin{document}

\maketitle

\begin{abstract}
We prove uniqueness of solutions to the Cauchy problem for the derivative nonlinear Schr\"odinger equation in $L^\infty_tH^{1/2}_x$.
Our proof is based on the method of normal form reduction (NFR), which has been employed to obtain the uniqueness in $C_tH^s_x$, $s>1/2$.
To overcome logarithmic divergences at the $H^{1/2}$ regularity, we exploit the $B^{0+}_{\infty,1}$ control of solutions provided by a refined Strichartz estimate.
Our NFR argument consists of two stages: we first use NFR finitely many times to derive an intermediate equation in which the main cubic nonlinearity is restricted to a certain type of frequency interaction; we then apply the infinite NFR scheme to the intermediate equation.
Moreover, we modify the usual NFR argument relying on continuity in time of solutions so that the uniqueness in the class $L^\infty_tH^{1/2}_x$ can be obtained directly.
\end{abstract}

\tableofcontents


\section{Introduction}

We consider uniqueness of solutions to the Cauchy problem for the one-dimensional cubic derivative nonlinear Schr\"odinger equation (DNLS)
\begin{alignat}{2}
   i\partial_t u + \partial_x^2 u &= i\lambda \partial_x \big(|u|^2u\big),&\qquad &t\in (-T,T),\quad x\in \mathcal{M}, \label{dnls} \\
   u(0, x) &= u_0(x), &\quad &x \in \mathcal{M}, \label{ic}
\end{alignat}
where $\lambda \in \mathbb{R}$ is a non-zero real constant, the spatial domain $\mathcal{M}$ is either the real line $\mathbb{R}$ or the torus $\mathbb{T}=\mathbb{R}/2\pi \mathbb{Z}$, and the initial datum $u_0$ is given in $L^2$-based Sobolev space $H^s(\mathcal{M})$.

Well-posedness of the Cauchy problem \eqref{dnls}--\eqref{ic} has been extensively studied; we refer the reader to \cite{KNV23,HGKNV26} for a detailed description of previous works.
It has been proved to be well-posed in the scaling-critical space $L^2(\mathbb{R})$ in the non-periodic case (\cite{HGKNV26}) and in $H^{\frac16}(\mathbb{T})$ in the periodic case (\cite{KNV23}).
Currently, all the known well-posedness results below $H^{\frac12}$ rely on the complete integrability of \eqref{dnls}.
Note that the (local-in-time) well-posedness in $H^{\frac12}$ was obtained in \cite{T99,H06} by a gauge transformation and a fixed-point argument in Fourier restriction norm spaces (Bourgain spaces), without utilizing complete integrability. 

A standard definition of well-posedness in a space $X$ of initial data requires existence of a solution in $C_tX$, its uniqueness, and its continuous dependence on initial data in the topology of $X$.
Concerning the uniqueness, in many cases it is initially obtained in a space smaller than $C_tX$ or in some weaker sense (e.g., the unique limit of regular solutions starting from regularized initial data), and then the uniqueness in the entire class $C_tX$ is considered as a separate problem.
The latter property is called \emph{unconditional uniqueness}, and the Cauchy problem is said to be unconditionally well-posed if it is well-posed and the uniqueness holds unconditionally (cf. \cite{K95}).
Unconditional uniqueness is a concept that is independent of how solutions are constructed.

Regarding the well-posedness results on DNLS \eqref{dnls}, those below $H^{\frac32}$ were (originally) conditional ones.
For instance, the $H^{\frac12}$ results in \cite{T99,H06} provide uniqueness of solutions only in  the inverse gauge image of a suitable Bourgain space, which is smaller than $C_tH^{\frac12}$.
Later, the unconditional uniqueness in $H^1$ was proved in \cite{W08} and that in $H^s$ with $s>1/2$ was obtained in \cite{MY20,K-all} for both $\mathcal{M}=\mathbb{R}$ and $\mathbb{T}$.
On the other hand, one can argue for the uniqueness in $C_tH^s$ only if the nonlinearity is well-defined for a general function in $C_tH^s$.
This gives the lower bound of the Sobolev exponent $s\geq 1/6$ for unconditional uniqueness for \eqref{dnls} in view of the 1D Sobolev embedding $H^{\frac16}\subset L^3$.
In this article, we aim to fill the gap $1/2\geq s\geq 1/6$, and at the first onset we consider the endpoint regularity $s=1/2$, which is likely to be solved without utilizing complete integrability.

For $s\geq 1/6$, we can show (see Appendix~\ref{appendix:weak}) that any function $u\in L^\infty(-T,T;H^s(\mathcal{M}))$ solving \eqref{dnls} in the sense of distributions belongs to $C([-T,T];H^{s-\varepsilon}(\mathcal{M}))\cap C_w([-T,T];H^s(\mathcal{M}))$ for any $\varepsilon>0$.%
\footnote{More precisely, $u$ coincides for \textit{a.e.}~$t\in (-T,T)$ with a unique function which belongs to $C([-T,T];H^{s-\varepsilon}(\mathcal{M}))\cap C_w([-T,T];H^s(\mathcal{M}))$; we will always identify $u$ with this continuous representative.
Here, $C_w$ means the class of weakly continuous functions.}
In particular, the initial value $u(0)\in H^s(\mathcal{M})$ is well-defined.
Then, we say $u\in L^\infty(-T,T;H^s(\mathcal{M}))$ is a solution to the initial value problem \eqref{dnls}--\eqref{ic} with the initial datum $u_0\in H^s(\mathcal{M})$ if $u$ is a solution of \eqref{dnls} on $(-T,T)\times \mathcal{M}$ in the sense of distributions and $u$ satisfies the initial condition \eqref{ic} in the above sense.
Now, we state the main theorem.
\begin{thm} \label{thm:main}
For both $\mathcal{M}=\mathbb{R}$ and $\mathcal{M}=\mathbb{T}$, the solution to the Cauchy problem \eqref{dnls}--\eqref{ic} is unique in $L^\infty(-T,T;H^{\frac12}(\mathcal{M}))$.
\end{thm}

\begin{rem}\label{rem:inviscid}
(i) It may not seem to be very meaningful to focus on the particular regularity $H^{\frac12}$ for DNLS.
Indeed, we are motivated by the same problem for the modified Benjamin-Ono equation (mBO).
On one hand, many techniques can be applied to these two equations in a similar manner; e.g., the gauge transformation, multilinear estimates in Fourier restriction norms, the normal form reduction, and so on.
On the other hand, mBO is not known to be completely integrable.
Moreover, the space $H^{\frac12}$ has the special importance for mBO, as it is the energy space.
In particular, similarly to the case of the Benjamin-Ono equation, we expect that the unconditional uniqueness in the energy space could be used to justify the inviscid limit (i.e., considering the perturbed equation with the viscosity term $\varepsilon \partial_x^2u$ and taking the limit $\varepsilon \downarrow 0$; see \cite{M13}).
Since both the nonlinear interactions and the gauge transformation are much simpler for DNLS than for mBO, we intend to study DNLS first.

(ii) We aim to prove the uniqueness in a wider class $L^\infty_tH^{\frac12}$ rather than in $C_tH^{\frac12}$.
In fact, this stronger uniqueness assertion is required in application to the inviscid limit problem mentioned above.
Recall that the difference between $C_tH^s$ and $L^\infty_tH^s$ is not relevant when $s>1/2$, because a solution in $L^\infty_tH^s$ necessarily belongs to $C_tH^{s'}$ for any $s'\in (1/2,s)$ as seen above and we can apply the uniqueness in $C_tH^{s'}$.
However, at the limiting regularity the uniqueness in $C_tH^{\frac12}$ does not yield that in $L^\infty_tH^{\frac12}$ by the same argument, and we have to prove the latter directly.
\end{rem}

Here, we describe our strategy of the proof.
The unconditional uniqueness in $H^s$ for $s>1/2$ was proved (\cite{MY20,K-all}) by employing the gauge transformation used for local well-posedness in $H^{\frac12}$ (\cite{T99,H06}) and the \emph{normal form reduction} argument (NFR, for short).
The latter method is based on simple integration-by-parts operations and it has been frequently used for the unconditional uniqueness problem.
It was first used for the periodic KdV equation in \cite{BIT11}, and later for the periodic 1D cubic NLS in \cite{GKO13} where the infinite iteration scheme of NFR was first introduced, and then developed also for the non-periodic problem since \cite{KOY20}; see \cite{K-all} for a survey of the method.
In particular, some problems (such as the KdV equation, the modified KdV equation on $\mathbb{T}$ and the Benjamin-Ono equation) needed only finite-time applications of NFR, while others (such as NLS, DNLS, the modified KdV on $\mathbb{R}$ and the modified Benjamin-Ono equation) required infinite iteration scheme. 
To prove Theorem~\ref{thm:main}, we first apply the same gauge transformation as before and consider the transformed equation in which the nonlinearity has become essentially the cubic term $v^2\partial_x\bar{v}$.
For the case $s>1/2$ in \cite{MY20,K-all}, the infinite NFR scheme was successfully applied to this equation. 
However, it seems difficult to adapt the same argument to the case $s=1/2$ due to several reasons.

First of all, the nonlinear term $v^2\partial_x\bar{v}$ is not well-defined in the framework of distributions for a general function $v$ in $L^\infty_tH^{\frac12}$ (while it is well-defined for $v\in L^\infty_tH^s$, $s>1/2$, by the Sobolev inequalities).
This means that one cannot argue uniqueness of solutions in the entire class $L^\infty_tH^{\frac12}$ (or $C_tH^{\frac12}$) for the transformed equation.
The key ingredient to bypass this obstacle at $s=1/2$ is the so-called \emph{refined Strichartz estimate}, which was recently used in \cite{MP23} to improve the unconditional uniqueness result for the Benjamin-Ono equation via the NFR method.
We begin with an arbitrary solution $u\in L^\infty_tH^{\frac12}$ to \eqref{dnls}.
Then, $u$ satisfies the associated integral equation
\[ u(t)=e^{i(t-t_0)\partial_x^2}u(t_0)+\lambda \int _{t_0}^te^{i(t-t')\partial_x^2}\partial_x\big( |u(t')|^2u(t')\big) \,dt',\quad t,t_0\in [-T,T].\]
The refined Strichartz estimate (Proposition~\ref{prop:RS} below) is derived from this integral representation and the standard Strichartz estimate (for linear solutions) on short intervals.
It asserts that a solution $u$ has better $L^p$ property in $x$ (after averaging in $t$) than that predicted by the Sobolev embedding.
For instance, we will show $u\in L^4_tL^\infty$ for any solution $u\in L^\infty_tH^{\frac12}$, despite that $H^{\frac12}\not\subset L^\infty$.
Since the gauge transform $v$ of $u$ satisfies $|v|=|u|$, it also belongs to $L^4_tL^\infty$ in addition to $L^\infty_tH^{\frac12}$.
Hence, it suffices to prove the uniqueness for the transformed equation in the smaller class $L^\infty_tH^{\frac12}\cap L^4_tL^\infty$.
We can use this additional property and the Moser estimate
\[ \| fg\|_{H^{\frac12}}\lesssim \| f\|_{L^\infty}\| g\|_{H^{\frac12}}+\| f\|_{H^{\frac12}}\| g\|_{L^\infty} \]
to make sense of the nonlinearity $v^2\partial_x\bar{v}$ (see Lemma~\ref{lem:Sobolev-v} below).

Another (in fact, much more serious) difficulty is to overcome various logarithmic divergences arising in nonlinear estimates at the $H^{\frac12}$ level.
Now, the refined Strichartz estimate actually shows that $u\in L^4_tB^{\varepsilon}_{\infty,1}$ for any $\varepsilon\in(0,1/4)$.
We thus have an extra $\varepsilon$ regularity from the scaling viewpoint, and this is the key to control nonlinear terms in $H^{\frac12}$.
The cost of using $L^\infty_x$-type norms is that it is not allowed to freely take absolute values on the Fourier side.
Indeed, the $L^p_tL^q_x$-Strichartz estimates do not hold for $\mathcal{F}^{-1}\big[ \big| \mathcal{F}_x[e^{it\partial_x^2}u_0]\big| \big]$ unless $q=2$, and similarly the refined Strichartz estimate fails for $\mathcal{F}^{-1}[|\mathcal{F}_xu|]$.
Although this point was not so serious in the study of the Benjamin-Ono equation \cite{MP23}, but this is not the case for our problem, and we need to prepare Coifman-Meyer type multilinear estimates (see Subsection~\ref{subsec:Coifman-Meyer} for more details).

We have seen that whether the NFR procedure finishes in finite times or continues infinitely depends on the problem. 
In general, finiteness of the iteration allows us to make more precise evaluation and recover more derivatives at each step.
However, especially when the ``non-resonance'' property of the nonlinear interaction is not strong, a finite number of applications are not enough to recover sufficient derivatives and then an infinite iteration is required.
In this case, one will be forced to make somewhat rough evaluation at each step in order to deal with the increasingly complicated resonance structures (phase functions) in a unified manner.
This is also one of the reasons why the infinite NFR scheme does not work as well when $s=1/2$.
To address this issue, we use the finite and infinite iteration schemes for different parts of the nonlinearity $v^2\partial_x\bar{v}$.
More specifically, we use the infinite iteration only to the ``high$\times$high$\times$high$\to$high'' frequency interaction, for which the resonant function $2(\xi_1+\xi_3)(\xi_2+\xi_3)$ is possible to become much smaller than each frequency $|\xi_j|$, while for the other types of interactions we apply NFR up to twice and make a closer investigation of each interaction at each step. 
We also note that the $B^{\varepsilon}_{\infty,1}$ control of the solution will be used only in the finite iteration part, since logarithmic divergence does not occur in the ``high$\times$high$\times$high$\to$high'' interaction case.
Our hybrid NFR procedure will be described in more detail in Section~\ref{sec:NFR0}.

Finally, as mentioned in Remark~\ref{rem:inviscid} (ii), we need to prove the uniqueness in $L^\infty _tH^{\frac12}$ directly rather than in $C_tH^{\frac12}$.
The only issue is how to justify various formal computations for a general solution, and the existing results based on NFR are all assuming continuity in $t$ for this purpose.
We will confirm that the argument can be modified so that the computations are justified under the weaker assumption $u\in L^\infty_tH^{\frac12}$ (see Subsections~\ref{subsec:just}, \ref{subsec:just-infinite}).

This article is organized as follows.
In Section~\ref{sec:pre}, we transform the equation by the gauge transformation and prepare several lemmas (refined Strichartz estimates and Coifman-Meyer type estimates) used in our NFR argument.
In Section~\ref{sec:NFR0}, we classify frequency interactions of the main cubic nonlinear term and outline our two-stage NFR procedure.
The detailed analysis is provided in the following two sections.
Section~\ref{sec:NFR1} describes the first stage, where we apply the finite NFR iteration scheme to the most part of the nonlinearity.
As a conclusion of this stage, we obtain the intermediate equation (see \eqref{eq:varpi} and \eqref{eq:varpi-d}) in which the cubic interaction is restricted to the ``high$\times$high$\times$high$\to$high'' case and the other nonlinear parts have been transformed into the form that admits the $H^{\frac12}$ control.
The second stage consists of applying the infinite NFR scheme to the intermediate equation, and this is performed in Section~\ref{sec:NFR2}.
Finally, a proof of Theorem~\ref{thm:main} is given in Subsection~\ref{subsec:proof}.


\section{Preliminaries}
\label{sec:pre}


\subsection{Notation}

\begin{itemize}
\item We write $X\lesssim Y$ for the estimate $X\leq CY$ with a constant $C>0$.
$X\sim Y$ means $X\lesssim Y\lesssim X$.
By $X\gg Y$ we mean $X\geq CY$ with a suitably large constant $C>1$.
For $a\in \mathbb{R}$, we write $a+$, $a-$ to denote real numbers $a+\varepsilon$, $a-\varepsilon$ with a suitably small $\varepsilon>0$.
\item We will use the abbreviations $L^p_TX:=L^p_t(-T,T;X)$, $C_TX:=C_t([-T,T];X)$ for $T>0$ and a Banach space $X$ of functions in $x$.
Similarly, we will use $W_{T}^{1,p}X$, $C_{w,T}X$.
\item The Fourier and the inverse Fourier transforms (in the spatial variable) are denoted by $\mathcal{F}u$ or $\hat{u}$ and $\mathcal{F}^{-1}u$ or $\check{u}$, respectively.
We use the definition
\[ \hat{u}(\xi) :=\frac{1}{\sqrt{2\pi}}\int_{\mathcal{M}}u(x)e^{-i\xi x}\,dx,\qquad \xi\in \hat{\mathcal{M}}:=\left\{ \begin{alignedat}{2}
&\mathbb{R} &\quad &(\mathcal{M}=\mathbb{R}), \\ &\mathbb{Z} &\quad &(\mathcal{M}=\mathbb{T}),\end{alignedat} \right. \]
so that 
\[ u(x)=\mathcal{F}^{-1}[\hat{u}](x):=\frac{1}{\sqrt{2\pi}}\int _{\hat{\mathcal{M}}}\hat{u}(\xi) e^{i\xi x}\,d\xi :=\left\{ \begin{alignedat}{2}
&\frac{1}{\sqrt{2\pi}}\int _{\mathbb{R}}\hat{u}(\xi) e^{i\xi x}\,d\xi &\quad &(\mathcal{M}=\mathbb{R}), \\ 
&\frac{1}{\sqrt{2\pi}}\sum _{\xi \in \mathbb{Z}}\hat{u}(\xi) e^{i\xi x} &\quad &(\mathcal{M}=\mathbb{T}).\end{alignedat} \right. \]  
The linear Schr\"odinger propagator is denoted by 
\[ S(t):=e^{it\partial_x^2}=\mathcal{F}^{-1}e^{-it\xi^2}\mathcal{F}.\]
In the periodic setting, we also use the following projection operators:
\[ P_cu:=\frac{1}{\sqrt{2\pi}}\hat{u}(0)=\frac{1}{2\pi}\int_{\mathbb{T}}u(x)\,dx,\qquad P_{\neq c}u:=u-P_cu=\frac{1}{\sqrt{2\pi}}\sum _{\xi\in \mathbb{Z}\setminus \{ 0\}}\hat{u}(\xi) e^{i\xi x}.\]
\item The characteristic function of a set $\Omega$ is denoted by $\mathbf{1}_\Omega$. We often indicate only the conditions determining $\Omega$.
\item Related to the Littlewood-Paley decomposition, we denote dyadic integers $1,2,4,8,\dots $ by capital letters $N,L,K,\dots$.
We use only the inhomogeneous decomposition.
\item Let $\eta :\mathbb{R}\to \mathbb{R}$ be a smooth function satisfying $\mathbf{1}_{[-5/4,5/4]}\leq \eta \leq \mathbf{1}_{[-8/5,8/5]}$.%
\footnote{%
We impose this support condition rather than $\mathbf{1}_{[-1,1]}\leq \eta \leq \mathbf{1}_{[-2,2]}$ just in order to define the functions $\tilde{\psi}_N$ below.
The specific numbers $\frac54$, $\frac85$ do not play any role.%
}
Define the Littlewood-Paley decomposition $\{ P_N\}_{N\geq 1}$ by $P_N:=\mathcal{F}^{-1}\psi_N\mathcal{F}$, where
\[ \psi_N(\xi):=\left\{ \begin{alignedat}{2} &\eta (\frac{|\xi|}{N})-\eta (\frac{2|\xi|}{N}) &\quad &(N\geq 2),\\
&\eta (|\xi|) & &(N=1).\end{alignedat}\right. \]
We see that the function $\psi_N$ (smooth, even, defined on $\mathbb{R}$) is supported in $I_N$, where
\[ I_N:=\left\{ \begin{alignedat}{2} &\{ \xi :N/2\leq |\xi |\leq 2N\} &\quad &(N\geq 2),\\
&\{ \xi :|\xi |\leq 2\} & &(N=1).\end{alignedat}\right. \]
We sometimes use $\{ \tilde{\psi}_N\}_{N\geq 1}$, smooth functions satisfying $\textbf{1}_{\mathrm{supp}\,(\psi_N)}\leq \tilde{\psi}_N\leq \mathbf{1}_{I_N}$ and $\tilde{\psi}_N(\xi)=\tilde{\psi}_2(\xi/(N/2))$ ($N\geq 2$). 
\item The (inhomogeneous) Sobolev spaces $H^s(\mathcal{M})=B^s_{2,2}(\mathcal{M})$ and Besov spaces $B^s_{p,q}(\mathcal{M})$ are defined by
\[ \| u\|_{H^s}:=\| \langle \xi \rangle ^s\hat{u}(\xi )\|_{L^2_\xi}\sim \| u\|_{B^s_{2,2}},\quad \| u\|_{B^s_{p,q}}:=\big\| N^s\| P_Nu\|_{L^p_x}\big\|_{\ell^q_N},\]
where $\langle \xi \rangle :=(1+\xi^2)^{1/2}$.
We will estimate the solution in the space $L^\infty_TH^{\frac12}\cap L^4_TB^{0+}_{\infty,1}$.
For short, we write its norm as $\| \cdot \|_{S_T}$:
\[ \| w\|_{S_T}:=\| w\|_{L^\infty_TH^{\frac12}}+\| w\|_{L^4_TB^{0+}_{\infty,1}}.\]
\item $\hat{H}^s:=\langle \cdot \rangle ^{-s}L^2$ denotes the weighted $L^2$ space equipped with the norm 
\[ \| f\|_{\hat{H}^s}:=\| \langle \xi \rangle^s f(\xi)\|_{L^2_\xi}.\]
\item We will use the abbreviation $\xi_{ij\dots}$ for $\xi_i+\xi_j+\dots$.
For instance, $\xi_{13}:=\xi_1+\xi_3$ and $\eta_{123}:=\eta_1+\eta_2+\eta_3$.
\end{itemize}


\subsection{Refined Strichartz estimates}

Let $u\in L^\infty_TH^{\frac12}(\mathcal{M})$ satisfy the equation \eqref{dnls} in the sense of distributions.
We can verify (see Appendix~\ref{appendix:weak}) that $u\in C_TH^{\frac12-}(\mathcal{M})\cap C_{w,T}H^{\frac12}(\mathcal{M})$ and $u$ satisfies the integral equation 
\begin{equation}\label{ie}
\begin{aligned}
u(t)=S(t-t_0)u(t_0)+\lambda \int _{t_0}^tS(t-t')\partial_x\big( |u(t')|^2u(t')\big) \,dt'\qquad &\\
\text{in $H^{-\frac12-}(\mathcal{M})$},\qquad \forall t,t_0\in [-T,T].&
\end{aligned}
\end{equation}

The following a priori estimate, which we call a refined Strichartz estimate, holds equally for $\mathcal{M}=\mathbb{R}$ and $\mathbb{T}$.
\begin{prop}\label{prop:RS}
Let $T>0$, and suppose that $u\in L^\infty_TH^{\frac12}(\mathcal{M})$ is a solution of \eqref{dnls}.
Then, it holds that
\begin{gather}\label{est:RS1}
\| P_Nu\| _{L^4_TL^\infty}\lesssim \Big( N^{-\frac12+}+T^{\frac14}N^{-\frac14+}\Big) \Big( 1+\| u\|_{L^\infty_TH^{\frac12}}^2\Big) \| u\|_{L^\infty_TH^{\frac12}}\quad (N\geq 1).
\end{gather}
In particular, $u\in L^4_TB^s_{\infty,1}$ for $s<\frac14$.
Moreover, for $T\in (0,1]$ and $s\in [0,\frac14)$, 
\begin{gather}\label{est:RS2}
\| u\| _{L^4_TB^s_{\infty,1}}\lesssim \Big( T^{\frac14-s-}+T^{\frac12-s-}\| u\|_{L^\infty_TH^{\frac12}}^2\Big) \| u\|_{L^\infty_TH^{\frac12}}.
\end{gather}
\end{prop}

\begin{proof}
{\bf Proof of \eqref{est:RS1}:}
If $T\geq N^{-1}$, we divide the time interval $[-T,T]$ into $\sim TN$ sub-intervals $I_j:=[t_j,t_{j+1}]$ of size $\sim N^{-1}$.
In the periodic setting, it is known that the same Strichartz estimates as those on $\mathbb{R}$ hold for short time length $\lesssim N^{-1}$ if the frequency is restricted to $\{ |\xi |\leq N\}$; see \cite{BGT04}.
In particular, it holds that
\begin{align*}
\| S(t-t_j)P_Nu(t_j)\|_{L^4(I_j;L^\infty (\mathcal{M}))}&\lesssim \| P_Nu(t_j)\|_{L^2(\mathcal{M})},\\
\Big\| \int _{t_j}^tS(t-t')P_NF(t')\,dt'\Big\|_{L^4(I_j;L^\infty (\mathcal{M}))}&\lesssim \| P_NF\|_{L^1(I_j;L^2(\mathcal{M}))}
\end{align*}
for $\mathcal{M}=\mathbb{R}$, $\mathbb{T}$.
Using \eqref{ie}, we have
\begin{align*}
\| P_Nu\|_{L^4_TL^\infty}^4&\lesssim \sum _{j}\Big( \| P_Nu(t_j)\|_{L^2}^4+\| P_N\partial_x(|u|^2u)\|_{L^1(I_j;L^2)}^4\Big) \\
&\lesssim TN\big( N^{-\frac12}\| u\|_{L^\infty _TH^{\frac12}}\big) ^4+\sum _{j}\Big( \big( N^{-1}\big) ^{\frac34}\| P_N\partial_x(|u|^2u)\|_{L^4(I_j;L^2)}\Big) ^4 \\
&\lesssim TN^{-1}\| u\|_{L^\infty _TH^{\frac12}}^4+N^{-3}\| P_N\partial_x(|u|^2u)\|_{L^4_TL^2}^4\\
&\lesssim TN^{-1}\| u\|_{L^\infty _TH^{\frac12}}^4+TN^{-1+}\| \partial_x(|u|^2u)\|_{L^\infty_TH^{-\frac12-}}^4.
\end{align*}
By the Sobolev inequalities, we have
\[ \| \partial_x(|u|^2u)\|_{H^{-\frac12-}}\lesssim \| u\|_{H^{\frac12}}^3.\]
Combining these estimates, we have
\[ \| P_Nu\| _{L^4_TL^\infty}\lesssim T^{\frac14}N^{-\frac14+}\big( 1+\| u\|_{L^\infty_TH^{\frac12}}^2\big) \| u\|_{L^\infty_TH^{\frac12}}.\]
If $T\leq N^{-1}$, we apply the above short-time Strichartz estimates on the whole interval $[-T,T]$ and Bernstein's inequalities to obtain
\begin{align*}
\| P_Nu\|_{L^4_TL^\infty}&\lesssim \| P_Nu(0)\|_{L^2}+\| P_N\partial_x(|u|^2u)\|_{L^1_TL^2}\\
&\lesssim N^{-\frac12}\| u(0)\|_{H^{\frac12}}+TN^{\frac12+}\| \partial_x(|u|^2u)\| _{L^\infty_TH^{-\frac12-}}\\
&\lesssim N^{-\frac12+}\big( 1+\| u\|_{L^\infty_TH^{\frac12}}^2\big) \| u\|_{L^\infty_TH^{\frac12}}.
\end{align*}
We have thus proved \eqref{est:RS1}. 

{\bf Proof of \eqref{est:RS2}:}
From \eqref{est:RS1}, we have 
\[ \| P_{>M}u\| _{L^4_TL^\infty}\lesssim \Big( M^{-\frac12+}+T^{\frac14}M^{-\frac14+}\Big) \Big( 1+\| u\|_{L^\infty_TH^{\frac12}}^2\Big) \| u\|_{L^\infty_TH^{\frac12}}.\]
On the other hand, from H\"older and Schwarz, we have
\begin{align*}
\| P_{\leq M}u\|_{L^4_TL^\infty}&\lesssim T^{\frac14}\langle \log M\rangle^{\frac12}\| u\|_{L^\infty_TH^{\frac12}}.
\end{align*}
For $0<T\leq 1$, we take $M\sim T^{-1}$ and obtain \eqref{est:RS2} with $s=0$.
The case $s\in (0,\frac14)$ is proved in a similar way.
\end{proof}


\subsection{Gauge transformation}

Next, we recall the gauge transformation used in \cite{T99,H06}.
\begin{defn}[Gauge transformation]\label{defn:gauge}
The spatial antiderivative $\partial_x^{-1}$ is defined by
\[ \partial_x^{-1}f(x):=\left\{ \begin{aligned} &\int_{-\infty}^x f(y)\,dy \qquad \text{(for $f\in L^1(\mathbb{R})$, if $\mathcal{M}=\mathbb{R}$)}, \\
&\frac{1}{2\pi}\int_{\mathbb{T}}\Big( \int_z^xf(y)\,dy\Big) \,dz=\mathcal{F}^{-1}\big[ (i\xi)^{-1}\hat{f}(\xi)\big] (x)\\
&\qquad\qquad\qquad\quad \text{(for $f$ with $P_cf=0$, if $\mathcal{M}=\mathbb{T}$)}.\end{aligned}\right. \]
For $u\in L^\infty_TL^2(\mathcal{M})$, we define
\[ v(t):=e^{-i\lambda J(u(t))}u(t),\qquad J(u):=\left\{ \begin{alignedat}{2} &\partial_x^{-1}(|u|^2) &\quad &(\mathcal{M}=\mathbb{R}), \\
&\partial_x^{-1}P_{\neq c}(|u|^2) &\quad &(\mathcal{M}=\mathbb{T}),
\end{alignedat}\right. \]
and 
\[ w(t,x):=v\Big( t, x-2\lambda \int_0^tP_c(|v(t')|^2)\,dt'\Big) \qquad (\mathcal{M}=\mathbb{T}). \]
\end{defn}
Since $|v(t,x)|=|u(t,x)|$ and $P_c(|w(t)|^2)=P_c(|v(t)|^2)$, the mappings $u\mapsto v$, $v\mapsto w$ are invertible and the inverse mappings are given by
\begin{align*}
v(t,x)\quad &\mapsto \quad u(t,x):=e^{i\lambda J(v(t))}v(t,x),\\
w(t,x)\quad &\mapsto \quad v(t,x):=w\Big( t,x+2\lambda \int_0^tP_c(|w(t')|^2)\,dt'\Big) .
\end{align*}
In the proof of unconditional uniqueness, we use the bijectivity on $L^\infty_TH^{\frac12}$ of the gauge transformation: $u\mapsto v$ for $\mathcal{M}=\mathbb{R}$ and $u\mapsto w$ for $\mathcal{M}=\mathbb{T}$.
Note that the mapping $u\mapsto v$ is locally bi-Lipschitz on $L^\infty_TH^s$ for each $s\geq 0$.
(For the periodic case, see \cite[Lemma~2.3]{H06}.
For the non-periodic case, a similar argument with the fractional Leibniz rule estimates would work.
A proof for $0\leq s\leq 1/2$ will be essentially given in the proof of Lemma~\ref{lem:conv-v} (i) and Lemma~\ref{lem:conv-vR} (i) below.)

If $u$ is a solution to \eqref{dnls}, the equation for its gauge transform is \emph{formally} given as follows.
In the non-periodic case, $v$ solves
\begin{equation}\label{eq:vR}
\partial_tv=i\partial_x^2v-\lambda v^2\partial_x\bar{v}+i\frac{\lambda^2}{2}|v|^4v.
\end{equation}
In the periodic case, $v$ and $w$ solve
\begin{align}
&\,\begin{aligned}
\partial_tv&=i\partial_x^2v -\lambda \Big( v^2\partial_x\bar{v}-2P_c(v\partial_x\bar{v})v\Big) +2\lambda P_c(|v|^2)\partial_xv\\
&\quad +i\lambda^2\Big( \frac{1}{2}P_{\neq c}(|v|^4)v-P_c(|v|^2)P_{\neq c}(|v|^2)v\Big) ,
\end{aligned}\label{eq:v} \\
&\begin{aligned}
\partial_tw&=i\partial_x^2w -\lambda \Big( w^2\partial_x\bar{w}-2P_c(w\partial_x\bar{w})w\Big) \\
&\quad +i\lambda ^2\Big( \frac{1}{2}P_{\neq c}(|w|^4)w-P_c(|w|^2)P_{\neq c}(|w|^2)w\Big) .
\end{aligned}\label{eq:w}
\end{align}
An important step in the problem of unconditional uniqueness is to verify these equations for a general (distributional) solution $u\in L^\infty_TH^{\frac12}$ of \eqref{dnls}, which we address in the next proposition.

\begin{prop}\label{prop:justify-eq}
Let $u\in L^\infty_TH^{\frac12}$ be a distributional solution of \eqref{dnls}, and let $v,w$ be defined as in Definition~\ref{defn:gauge}.
Then, $u,v,w\in L^\infty_TH^{\frac12}\cap L^4_TL^\infty$, and $v^2\partial_x\bar{v}$, $w^2\partial_x\bar{w}$ are well-defined as distributions in $L^4_TH^{-\frac12-}$.
Moreover, $v$ and $w$ are solutions of \eqref{eq:vR} or \eqref{eq:v}, and respectively \eqref{eq:w}, on $(-T,T)\times\mathcal{M}$ in the sense of distributions.
\end{prop}

\begin{proof}
By Proposition~\ref{prop:RS} we have $u\in L^\infty_TH^{\frac12}\cap L^4_TL^\infty$, and hence, by definition, $v$, $w$ belong to $L^4_TL^\infty$.
The fact that $v,w\in L^\infty_TH^{\frac12}$ will be shown later.
Then, the nonlinear terms $v^2\partial_x\bar{v}$, $w^2\partial_x\bar{w}$ can be defined as distributions in $L^4_TH^{-\frac12-}$, using the next lemma:
\begin{lem}\label{lem:Sobolev-v}
We have 
\begin{align*}
\| f_1f_2\partial_x\bar{f}_3\|_{H^{-\frac12-}}\lesssim \Big( &\| f_1\|_{L^\infty}\| f_2\| _{H^{\frac12}}+\| f_1\|_{H^{\frac12}}\| f_2\| _{L^\infty}\Big) \| f_3\|_{H^{\frac12}}.
\end{align*}
In particular, if $f_1,f_2\in L^\infty _TH^{\frac12}\cap L^4_TL^\infty$ and $f_3\in L^\infty _TH^{\frac12}$, then $f_1f_2\partial_x\bar{f}_3\in L^4_TH^{-\frac12-}$.
\end{lem}

\begin{proof}
This follows from the Sobolev inequality
\[ \| fg\|_{H^{-\frac12-}}\lesssim \| f\|_{H^{\frac12}}\| g\| _{H^{-\frac12}} \]
and the Moser estimate
\[ \| fg\|_{H^{\frac12}}\lesssim \| f\|_{L^\infty}\| g\|_{H^{\frac12}}+\| f\|_{H^{\frac12}}\| g\|_{L^\infty}.\qedhere \]
\end{proof}

To justify the equation \eqref{eq:vR} or \eqref{eq:v} for $v$, we consider smooth approximation $v_N$ defined by
\begin{equation}\label{defn:v_N}
v_N:=e^{-i\lambda J(u_N)}u_N,\qquad u_N:=P_{\leq N}u.
\end{equation}
First, we prove the following:
\begin{lem}\label{lem:conv-u}
Let $u\in L^\infty_TH^{\frac12}$ be a distributional solution of \eqref{dnls} on $(-T,T)$, and define $u_N:=P_{\leq N}u$.
Then, the following holds.
\begin{enumerate}
\item[$\mathrm{(i)}$] $u\in L^p_TL^\infty$ and $u_N\to u$ in $L^p_TH^{\frac12}\cap L^\infty_TH^{\frac12-}\cap L^p_TL^\infty$ for any $1\leq p<\infty$.
\item[$\mathrm{(ii)}$] $P_{\leq N}(|u|^2u)-|u_N|^2u_N\to 0$ in $L^p_TH^{\frac12}\cap L^\infty_TH^{\frac12-}\cap L^p_TL^\infty$ for any $1\leq p<\infty$.
\item[$\mathrm{(iii)}$] $L^2$ conservation law: $\| u(t)\|_{L^2}=\| u(0)\|_{L^2}$ for $t\in [-T,T]$.%
\footnote{%
Recall that (after modifying on a set of measure zero) a distributional solution $u\in L^\infty_TH^{\frac12}$ to \eqref{dnls} has the regularity $C_TH^{\frac12-}$.%
}%
\end{enumerate}
\end{lem}

\begin{proof}
(i) Note that $P_{\leq N}f\to f$ in $H^s$ for any $f\in H^s$.
Then, $u\in L^\infty_TH^{\frac12}$ and the dominated convergence theorem imply the convergence in $L^p_TH^{\frac12}$.
The convergence in $L^\infty_TH^{\frac12-}$ follows from the fact that $u\in C_TH^{\frac12-}$.
By interpolation between \eqref{est:RS1} and
\[ \| P_Nu\|_{L^\infty_TL^\infty}\lesssim \| u\|_{L^\infty_TH^{\frac12}},\]
we have
\begin{equation}\label{est:PNu}
\| P_Nu\|_{L^p_TL^\infty}\lesssim_T N^{-\min\{ \frac1p,\frac14\}+}\big( 1+\| u\|_{L^\infty_TH^{\frac12}}^2\big) \| u\|_{L^\infty_TH^{\frac12}}.
\end{equation}
Therefore, we see $u\in L^p_TL^\infty$ and that 
\[ \| u_N-u\|_{L^p_TL^\infty}\lesssim _{p,T,\| u\|_{L^\infty_TH^{\frac12}}}N^{-\min \{ \frac1p,\frac14\} +}\to 0 \qquad (N\to \infty).\]

(ii) It suffices to show the convergence of $P_{>N}(|u|^2u)$ and $|u|^2u-|u_N|^2u_N$ to $0$.
By $u\in C_TH^{\frac12-}$ and the Sobolev inequality $\| fg\|_{H^{s_1}}\lesssim \| f\|_{H^{s_2}}\| g\|_{H^{s_2}}$ for $0<s_1<\frac12$ with $s_2:=\frac12(s_1+\frac12)$, we see $|u|^2u\in C_TH^{\frac12-}$.
Since the Moser estimate with $u\in L^p_TL^\infty$ implies $|u|^2u\in L^p_TH^{\frac12}$, the convergence of $P_{>N}(|u|^2u)$ in $L^p_TH^{\frac12}\cap L^\infty_TH^{\frac12-}$ follows similarly to (i). 
From the support property of $\psi_N$ we see that $P_{>N}(|P_{\leq N/4}u|^2P_{\leq N/4}u)=0$, from which we have
\[ \| P_{>N}(|u|^2u)\|_{L^p_TL^\infty}\lesssim \| u\|_{L^{3p}_TL^\infty}^2\| P_{>N/4}u\|_{L^{3p}_TL^\infty}.\]
The convergence in $L^p_TL^\infty$ follows from (i).
For $|u|^2u-|u_N|^2u_N$, the convergence in these norms can be shown by using Sobolev, Moser and H\"older and the convergence in (i).

(iii) Since $u_N$ is a solution (in the classical sense) of
\[ \partial_tu_N=i\partial_x^2u_N+\lambda \partial_xP_{\leq N}(|u|^2u),\]
a direct calculation shows that for any $N\geq 1$
\[ \| u_N(t)\|_{L^2}^2=\| u_N(0)\|_{L^2}^2+2\lambda \Re \int_0^t\int_{\mathcal{M}}\bar{u}_N\partial_x\Big( P_{\leq N}(|u|^2u)-|u_N|^2u_N\Big) \,dx\,dt',\quad t\in [-T,T].\]
The last integral is bounded by $2|\lambda|\| u_N\|_{L^\infty_TH^{\frac12}}\| P_{\leq N}(|u|^2u)-|u_N|^2u_N\|_{L^1_TH^{\frac12}}$.
Then, the $L^2$ conservation is verified by taking the limit $N\to \infty$ with the convergences in (i), (ii).
\end{proof}

From here, we consider the periodic and non-periodic cases separately.

\medskip
\noindent
\underline{Case $\mathcal{M}=\mathbb{T}$}

After a long calculation (see \cite[Section~5.1]{K-all} for details), we see that $v_N$ satisfies
\begin{align}
(\partial_t-i\partial_x^2)v_N
&=-\lambda \Big( v_N^2\partial_x\bar{v}_N-2P_c(v_N\partial_x\bar{v}_N)v_N\Big) +2\lambda P_c(|v_N|^2)\partial_xv_N \label{term1}\\
&\quad +i\lambda^2 \Big( \frac{1}{2}P_{\neq c}(|v_N|^4)v_N-P_c(|v_N|^2)P_{\neq c}(|v_N|^2)v_N\Big) \label{term2}\\
&\quad +\lambda e^{-i\lambda J(u_N)}\partial_x\Big( P_{\leq N}(|u|^2u)-|u_N|^2u_N\Big) \label{term3}\\
&\quad -2i\lambda^2 v_N\partial_x^{-1}P_{\neq c}\Re \Big[ \bar{u}_N\partial_x\Big( P_{\leq N}(|u|^2u)-|u_N|^2u_N\Big) \Big] .\label{term4}
\end{align}
We note that the equation is satisfied in the classical sense, because $u_N,v_N\in C^1_TH^\infty$ and $P_{\leq N}(|u|^2u)\in C_TH^\infty$.

\begin{lem}\label{lem:conv-v}
Let $\mathcal{M}=\mathbb{T}$ and $u\in L^\infty_TH^{\frac12}$ be a distributional solution of \eqref{dnls} on $(-T,T)$, and define $v,u_N,v_N$ as in Definition~\ref{defn:gauge} and in \eqref{defn:v_N}.
Then, the following holds.
\begin{enumerate}
\item[$\mathrm{(i)}$] $v\in L^\infty_TH^{\frac12}\cap L^p_TL^\infty$ and $v_N\to v$ in $L^p_TH^{\frac12}\cap L^p_TL^\infty$ for any $1\leq p<\infty$.
\item[$\mathrm{(ii)}$] $\eqref{term1}\to -\lambda (v^2\partial_x\bar{v}-2P_c(v\partial_x\bar{v})v) +2\lambda P_c(|v|^2)\partial_xv$\quad in $L^p_TH^{-\frac12-}$ for any $1\leq p<\infty$.
\item[$\mathrm{(iii)}$] $\eqref{term2}\to i\lambda^2 \big( \frac{1}{2}P_{\neq c}(|v|^4)v-P_c(|v|^2)P_{\neq c}(|v|^2)v\big)$\quad in $L^p_TH^{\frac12-}$ for any $1\leq p<\infty$.
\item[$\mathrm{(iv)}$] $\eqref{term3}\to 0$ in $L^\infty_TH^{-\frac12-}$.
\item[$\mathrm{(v)}$] $\eqref{term4}\to 0$ in $L^p_TH^{\frac12-}$ for any $1\leq p<\infty$.
\end{enumerate}
Consequently, $v$ is a distributional solution of \eqref{eq:v} on $(-T,T)$.
\end{lem}

\begin{proof}
(i) We have
\begin{gather*}
\begin{aligned}
\| e^{-i\lambda J(f)}\|_{H^1}&\lesssim \| e^{-i\lambda J(f)}\|_{L^2}+\| P_{\neq c}(|f|^2)e^{-i\lambda J(f)}\|_{L^2}\\
&\lesssim 1+\| f\|_{L^4}^2\lesssim 1+\| f\|_{H^{\frac14}}^2,
\end{aligned}\\
\begin{aligned}
\| e^{-i\lambda J(f)}-e^{-i\lambda J(g)}\|_{H^1}&\lesssim \| J(f)-J(g)\|_{L^2}+\| P_{\neq c}(|f|^2-|g|^2)\|_{L^2}\\
&\qquad +\| P_{\neq c}(|g|^2)[J(f)-J(g)]\|_{L^2}\\
&\lesssim (\| f\|_{H^{\frac14}}+\| g\|_{H^{\frac14}})(1+\| g\|_{H^{\frac14}}^2)\| f-g\|_{H^{\frac14}}.
\end{aligned}
\end{gather*}
From the Sobolev inequality $\| fg\|_{H^{\frac12}}\lesssim \| f\|_{H^1}\| g\|_{H^{\frac12}}$ and Lemma~\ref{lem:conv-u} (i), we see that
\begin{gather*}
\| v\|_{L^\infty_TH^{\frac12}}\lesssim \Big( 1+\| u\|_{L^\infty_TH^{\frac14}}^2\Big) \| u\|_{L^\infty_TH^{\frac12}},\\
\begin{aligned}
\| v_N-v\|_{L^p_TH^{\frac12}}&\lesssim \| e^{-i\lambda J(u_N)}-e^{-i\lambda J(u)}\| _{L^p_TH^1}\| u_N\|_{L^\infty_TH^{\frac12}}+\| e^{-i\lambda J(u)}\|_{L^\infty_TH^1}\| u_N-u\|_{L^p_TH^{\frac12}}\\
&\lesssim _{\| u\|_{L^\infty_TH^{\frac12}}}\| u_N-u\|_{L^p_TH^{\frac12}}\to 0\qquad (N\to \infty).\end{aligned}
\end{gather*}
Also, from Lemma~\ref{lem:conv-u} (i) we see $v\in L^p_TL^\infty$ and
\[ \begin{aligned}
\| v_N-v\|_{L^p_TL^\infty}&\lesssim \big\| P_{\neq c}(|u_N|^2-|u|^2)\big\| _{L^\infty_TL^1}\| u_N\|_{L^p_TL^\infty}+\| u_N-u\|_{L^p_TL^\infty}\\
&\lesssim _{p,\| u\|_{L^\infty_TH^{\frac12}}}\| u_N-u\|_{L^\infty_TL^2}+\| u_N-u\|_{L^p_TL^\infty}\to 0\qquad (N\to \infty).\end{aligned}\]

(ii) Lemma~\ref{lem:Sobolev-v} and (i) imply the convergence of the first term in $L^p_TH^{-\frac12-}$.
Since
\[ |P_c(v_1\partial_x\bar{v}_2)|+|P_c(v_1\bar{v}_2)|\lesssim \| v_1\|_{H^{1/2}}\| v_2\|_{H^{1/2}},\]
by using (i) we obtain the convergence of the last two terms in $L^p_TH^{\frac12}$ and in $L^p_TH^{-\frac12}$, respectively.

(iii) This is straightforward, using (i) and the Sobolev inequality $\| fg\|_{H^{s_1}}\lesssim \| f\|_{H^{s_2}}\| g\|_{H^{\frac12}}$ for $0<s_1<s_2<\frac12$.

(iv) By the Sobolev inequality $\| fg\|_{H^{-\frac12-}}\lesssim \| f\|_{H^1}\| g\|_{H^{-\frac12-}}$, the claim is reduced to
\[ \| P_{\leq N}(|u|^2u)-|u_N|^2u_N\|_{L^\infty_TH^{\frac12-}}\to 0 \qquad (N\to \infty),\]
which has been shown in Lemma~\ref{lem:conv-u} (ii).

(v) By the Sobolev inequalities 
\[ \| fg\|_{H^{s_1}}\lesssim \| f\|_{H^{\frac12}}\| g\|_{H^{s_2}}\quad (0<s_1<s_2<1/2),\qquad  \| fg\|_{H^{-\frac12-}}\lesssim \| f\|_{H^{\frac12}}\| g\| _{H^{-\frac12}}, \]
it suffices to show that
\[ \| P_{\leq N}(|u|^2u)-|u_N|^2u_N\|_{L^p_TH^{\frac12}}\to 0 \qquad (N\to \infty ),\]
which has been shown in Lemma~\ref{lem:conv-u} (ii).
\end{proof}

We have shown that $v$ is a solution of \eqref{eq:v} in the sense of distributions.
Noticing that $P_c(|v(t)|^2)$ is a constant by Lemma~\ref{lem:conv-u} (iii), it is easy to see that $w\in L^\infty_TH^{\frac12}$ and it satisfies the equation \eqref{eq:w} in the sense of distributions.

\medskip
\noindent
\underline{Case $\mathcal{M}=\mathbb{R}$}

A direct calculation shows that $v_N$ is a (classical) solution of 
\begin{align}
(\partial_t-i\partial_x^2)v_N
&=-\lambda v_N^2\partial_x\bar{v}_N+i\frac{\lambda^2}{2}|v_N|^4v_N \label{term1R}\\
&\quad +\lambda e^{-i\lambda J(u_N)}\partial_x\Big( P_{\leq N}(|u|^2u)-|u_N|^2u_N\Big) \label{term2R}\\
&\quad -2i\lambda^2 v_N\Re \int_{-\infty}^x \bar{u}_N\partial_x\Big( P_{\leq N}(|u|^2u)-|u_N|^2u_N\Big) \,dy .\label{term3R}
\end{align}
It suffices to prove the following lemma:
\begin{lem}\label{lem:conv-vR}
Let $\mathcal{M}=\mathbb{R}$ and $u\in L^\infty_TH^{\frac12}$ be a distributional solution of \eqref{dnls} on $(-T,T)$, and define $v,u_N,v_N$ as in Definition~\ref{defn:gauge} and in \eqref{defn:v_N}.
Then, the following holds.
\begin{enumerate}
\item[$\mathrm{(i)}$] $v\in L^\infty_TH^{\frac12}\cap L^p_TL^\infty$ and $v_N\to v$ in $L^p_TH^{\frac12}\cap L^p_TL^\infty$ for any $1\leq p<\infty$.
\item[$\mathrm{(ii)}$] $\eqref{term1R}\to -\lambda v^2\partial_x\bar{v}+i\frac{\lambda^2}{2}|v|^4v$\quad in $L^p_TH^{-\frac12-}$ for any $1\leq p<\infty$.
\item[$\mathrm{(iii)}$] $\eqref{term2R}\to 0$ in $L^\infty_TH^{-1}$.
\item[$\mathrm{(iv)}$] $\eqref{term3R}\to 0$ in $L^p_TL^\infty$ for any $1\leq p<\infty$.
\end{enumerate}
Consequently, $v$ is a distributional solution of \eqref{eq:vR} on $(-T,T)$.
\end{lem}
\begin{proof}
(i) The estimate in $L^p_TL^\infty$ is the same as the periodic case.
For the Sobolev norm, however, fractional Leibniz rule type estimates are needed to handle the exponential term which is not in $L^p$ for any $p<\infty$.
We recall the following:
\begin{lem}[{\cite[Lemma~2.7]{MP12}}] \label{lem:fracLeib}
Let $f,f_1,f_2$ be real-valued functions in $L^1(\mathbb{R})\cap L^2(\mathbb{R})$.
Then, for $0\leq s\leq \frac12$ we have
\begin{gather*}
\| e^{i\partial_x^{-1}f}g\|_{H^s}\lesssim \big( 1+\| f\|_{L^2}\big) \| g\|_{H^s},\\
\big\| \big( e^{i\partial_x^{-1}f_1}-e^{i\partial_x^{-1}f_2}\big) g\big\|_{H^s}\lesssim \Big( \| f_1-f_2\|_{L^2}+\big( 1+\| f_1\|_{L^2}\big) \| f_1-f_2\|_{L^1}\Big) \| g\|_{H^s}.
\end{gather*}
\end{lem}
From this lemma, we have
\begin{align*}
\| v\|_{L^\infty_TH^{\frac12}}&\lesssim \big(1+\| u\|_{L^\infty_TH^{\frac12}}^2\big) \| u\|_{L^\infty_TH^{\frac12}},\\
\| v_N-v\|_{L^pH^{\frac12}}&\lesssim \big( 1+\| u\|_{L^\infty_TH^{\frac12}}^4\big) \| u_N-u\|_{L^p_TH^{\frac12}}\to 0\qquad (N\to \infty)
\end{align*}
by Lemma~\ref{lem:conv-u} (i).

(ii) This is an immediate consequence of Lemma~\ref{lem:Sobolev-v}, the Sobolev inequality and the convergence in (i).

(iii) We write $F_N:=P_{\leq N}(|u|^2u)-|u_N|^2u_N$.
We can estimate this term as
\begin{align*}
&\| e^{-i\lambda J(u_N)}\partial_xF_N\|_{L^\infty_TH^{-1}}\\
&\lesssim \big\| \partial_x\big( e^{-i\lambda J(u_N)}F_N\big) \big\|_{L^\infty_TH^{-1}}+\| |u_N|^2F_N\|_{L^\infty_TL^1} \lesssim \big( 1+\| u_N\|_{L^\infty_TH^{\frac14}}^2\big) \| F_N\|_{L^\infty_TL^2}\\
&\lesssim \big( 1+\| u_N\|_{L^\infty_TH^{\frac14}}^2\big) \Big( \big\| P_{>N}(|u|^2u)\big\|_{L^\infty_TL^2}+\| u\|_{L^\infty_TH^{\frac13}}^2\| u_N-u\|_{L^\infty_TH^{\frac13}}\Big) \to 0\quad (N\to \infty),
\end{align*}
where we have used several Sobolev embeddings, Lemma~\ref{lem:conv-u} (i) and the fact that $u\in C_TH^{\frac12-}$ implies $|u|^2u\in C_TL^2$. 

(iv) Since $v\in L^p_TL^\infty$ for any $p<\infty$, it suffices to show
\[ \big\| \int_{-\infty}^xu_N\partial_xF_N\,dy \big\|_{L^p_TL^\infty_x}\to 0\qquad (N\to \infty) \]
for any $p<\infty$.
Decomposing $u_N$ and $F_N$ into dyadic pieces and using an integration by parts, we estimate it as
\begin{align*}
&\big\| \int_{-\infty}^x u_N\partial_xF_N\big\|_{L^p_TL^\infty}\\
&\leq \Big\| \sum_{K_1\leq N}\sum_{K_2\leq K_1}\big\| \int_{-\infty}^x P_{K_1}u\cdot \partial_xP_{K_2}F_N\big\| _{L^\infty}\Big\|_{L^p_T}\\
&\quad +\sum _{K_1\leq N}\| P_{K_1}u\cdot P_{K_1<\cdot \lesssim N}F_N\|_{L^p_TL^\infty} +\Big\| \sum_{K_1\leq N}\sum_{K_1<K_2\lesssim N}\big\| \int_{-\infty}^x \partial_xP_{K_1}u\cdot P_{K_2}F_N\big\| _{L^\infty}\Big\|_{L^p_T}\\
&\lesssim \Big\| \sum_{K_1,K_2}\min\Big\{ \frac{K_1}{K_2},\, \frac{K_2}{K_1}\Big\}^{\frac12}\| P_{K_1}u\|_{H^{\frac12}}\| P_{K_2}F_N\|_{H^{\frac12}}\Big\|_{L^p_T}+\sum_{K_1}\| P_{K_1}u\|_{L^{2p}_TL^\infty}\| F_N\|_{L^{2p}_TL^\infty}\\
&\lesssim \| u\|_{L^\infty_TH^{\frac12}}\| F_N\|_{L^p_TH^{\frac12}}+\sum_{K_1}\| P_{K_1}u\|_{L^{2p}_TL^\infty}\| F_N\|_{L^{2p}_TL^\infty}.
\end{align*}
Since $\sum_{N}\| P_Nu\|_{L^{2p}_TL^\infty}<\infty$ by \eqref{est:PNu}, we obtain the convergence from Lemma~\ref{lem:conv-u} (ii).
\end{proof}

This is the end of the proof of Proposition~\ref{prop:justify-eq}.
\end{proof}

We have already exploited the $L^4L^\infty$ refined Strichartz estimate with additional regularity (i.e., decay in \eqref{est:PNu}) for justification of the gauge transformed equations \eqref{eq:vR}, \eqref{eq:w}.
In the normal form reduction argument, we will also need the $L^4_TB^{0+}_{\infty,1}$ control of the gauge transform.
This can be obtained either from the control of $u$ in Proposition~\ref{prop:RS} and the definition of the gauge transformation (via fractional Leibniz rule) or by directly applying the short-time Strichartz argument to the equations \eqref{eq:vR}, \eqref{eq:w} we have just verified.
Here, we take the latter approach:
\begin{prop}\label{prop:RS-w}
Let $w\in L^\infty_TH^{\frac12}\cap L^4_TL^\infty$ be a distributional solution of \eqref{eq:vR} in the non-periodic case or \eqref{eq:w} in the periodic case.
Then, for $T\in (0,1]$ and $N\geq 1$, we have
\[ \| P_Nw\|_{L^4_TL^\infty}\lesssim N^{-\frac14+}\Big( 1+\| w\|_{L^\infty_TH^{\frac12}}^4+\| w\|_{L^4_TL^\infty}\| w\|_{L^\infty_TH^{\frac12}}\Big) \| w\|_{L^\infty_TH^{\frac12}}.\]
In particular, for $0\leq s<\frac14$ we have $w\in L^4_TB^s_{\infty,1}$ and 
\[ \| w\|_{L^4_TB^s_{\infty,1}}\lesssim T^{\frac14-s-}\Big( 1+\| w\|_{L^\infty_TH^{\frac12}}^4+\| w\|_{L^4_TL^\infty}\| w\|_{L^\infty_TH^{\frac12}}\Big) \| w\|_{L^\infty_TH^{\frac12}}.\]
Moreover, for two such solutions $w_1$ and $w_2$, we have
\[ \| w_1-w_2\|_{L^4_TB^s_{\infty,1}}\lesssim T^{\frac14-s-}\Big( 1+\| w_1\|_{L^\infty_TH^{\frac12}\cap L^4_TL^\infty}^4+\| w_2\|_{L^\infty_TH^{\frac12}\cap L^4_TL^\infty}^4\Big) \| w_1-w_2\|_{L^\infty_TH^{\frac12}\cap L^4_TL^\infty}.\]
\end{prop}

\begin{proof}
Similarly to the proof of Proposition~\ref{prop:RS}, the short-time Strichartz estimates imply
\begin{align*}
\| P_Nw\|_{L^4_TL^\infty}&\lesssim_T N^{-\frac14}\| w\|_{L^\infty_TH^{\frac12}}+N^{-\frac34}\| P_NF(w)\|_{L^4_TL^2}\\
&\lesssim N^{-\frac14}\| w\|_{L^\infty_TH^{\frac12}}+N^{-\frac14+}\| F(w)\|_{L^4_TH^{-\frac12-}},
\end{align*}
where
\[ F(w):=\left\{ \begin{alignedat}{2} &{-}\lambda w^2\partial_x\bar{w}+i\frac{\lambda^2}{2}|w|^4w & &(\mathcal{M}=\mathbb{R}),\\
&\begin{aligned}&{-}\lambda \Big( w^2\partial_x\bar{w}-2P_c(w\partial_x\bar{w})w\Big) \\
&\quad +i\lambda^2\Big( \frac{1}{2}P_{\neq c}(|w|^4)w-P_c(|w|^2)P_{\neq c}(|w|^2)w\Big)\end{aligned} &\qquad &(\mathcal{M}=\mathbb{T}).\end{alignedat}\right. \]
Applying Lemma~\ref{lem:Sobolev-v} and the Sobolev embedding, we have
\begin{align*}
\| F(w)\|_{H^{-\frac12-}}&\leq \| w^2\partial_x\bar{w}\|_{H^{-\frac12-}}+\| \text{other terms}\|_{L^2} \\
&\lesssim \Big( \| w\|_{L^\infty}+\| w\|_{L^2}+\| w\|_{H^{\frac12}}^3\Big) \| w\|_{H^{\frac12}}^2.
\end{align*}
This implies the first claimed estimate.
Given $M\geq 1$, summing up the estimate for $N\geq M$ implies
\[ \| P_{\geq M}w\|_{L^4_TB^s_{\infty,1}}\lesssim M^{s-\frac14+}\Big( 1+\| w\|_{L^\infty_TH^{\frac12}}^4+\| w\|_{L^4_TL^\infty}\| w\|_{L^\infty_TH^{\frac12}}\Big) \| w\|_{L^\infty_TH^{\frac12}}\]
for $s<\frac14$, while summing up for $N\leq M$ and H\"older in $t$, Bernstein in $x$ show
\[ \| P_{\leq M}w\|_{L^4_TB^s_{\infty,1}}\lesssim T^{\frac14}M^s\| w\|_{L^\infty_TH^{\frac12}}\]
for $s>0$ and with a factor $\langle \log M\rangle ^{\frac12}$ for $s=0$.
Choosing $M\sim T^{-1}$, we obtain the second estimate.
Finally, we can show the difference estimate by using
\[ \| P_N(w_1-w_2)\|_{L^4_TL^\infty}\lesssim_T N^{-\frac14}\| w_1-w_2\|_{L^\infty_TH^{\frac12}}+N^{-\frac14+}\| P_N(F(w_1)-F(w_2))\|_{L^4_TH^{-\frac12-}}\]
and repeating the same argument. 
\end{proof}


\subsection{Coifman-Meyer type lemmas}
\label{subsec:Coifman-Meyer}

All the results in this subsection hold equally for $\mathcal{M}=\mathbb{R}$ and $\mathbb{T}$.

In the NFR argument, we analyze the interaction representation of a solution $w(t,x)$ on the frequency space; i.e., consider the equation for $\omega (t,\xi) :=\mathcal{F}_x[S(-t)w(t,\cdot )](\xi)$ instead of $w$ itself.
On the other hand, the ($L^\infty_x$-type) refined Strichartz estimates for $w(t,x)$ are based on the Schr\"odinger dispersion and do not hold after taking the absolute value on the frequency side. 
In the previous result for the Benjamin-Ono equation in \cite{MP23}, the refined Strichartz estimates were used only to estimate $\partial_t\omega (t,\xi)$, and that part could be separated from the other ($L^2_x$-based) estimates, in which one can take absolute values of $\omega (t,\xi)$. 
In our problem, we need to use the refined Strichartz estimates more extensively, and it seems difficult to separate these steps in a similar manner.
Consequently, we will need Coifman-Meyer type estimates of trilinear Fourier multiplier operators.
For instance, consider the $L^2_x$ estimate of the following trilinear operator: 
\begin{align*}
&\mathcal{F}^{-1}\bigg[ \psi_N(\xi) \int _{\xi =\xi_{123}}\frac{e^{it\phi (\vec{\xi})}}{\phi (\vec{\xi})}i\xi_3(\psi_{N_1}\omega _1)(\xi_1)(\psi_{N_2}\omega _2)(\xi_2)(\psi_{N_3}\omega _3^*)(\xi_3)\bigg] \\
&\quad =S(-t)P_N\mathcal{F}^{-1}\bigg[ \int _{\xi =\xi_{123}}\frac{i\xi_3}{\phi (\vec{\xi})}\mathcal{F}[P_{N_1}w_1] (\xi_1)\mathcal{F}[P_{N_2}w_2](\xi_2)\mathcal{F}[P_{N_3}\bar{w}_3](\xi_3)\bigg] ,
\end{align*}
where 
\[ \omega^*(\xi):=\overline{\omega(-\xi)},\qquad \phi (\vec{\xi}):=\xi^2-\xi_1^2-\xi_2^2+\xi_3^2=2\xi_{13}\xi_{23},\]
and the integral is taken with respect to $(\xi_1,\xi_2,\xi_3)$ over the hyperplane $\{ \xi=\xi_{123}\}$ for given $\xi$.
Suppose that we want to estimate its $L^2_x$ norm by
\[ B\| P_{N_1}w_1\|_{L^2}\| P_{N_2}w_2\|_{L^\infty}\| P_{N_3}w_3\|_{L^\infty}, \]
for instance, when the multiplier $b(\xi_1,\xi_2,\xi_3):=i\xi_3/(\xi_{13}\xi_{23})$ obeys the bound
\[ |b(\xi_1,\xi_2,\xi_3)|\lesssim B,\qquad \text{whenever}\quad \psi_N(\xi _{123})\prod_{j=1}^3\psi_{N_j}(\xi_j)\neq 0.\]
The classical Coifman-Meyer theorem requires a condition like
\[ |\partial_\xi^\alpha b(\xi)|\lesssim _\alpha B|\xi |^{-|\alpha|},\qquad \xi=(\xi_1,\xi_2,\xi_3),\quad \alpha =(\alpha_1,\alpha_2,\alpha_3).\]
This condition seems too strong, however.
Moreover, the endpoint space $L^\infty$ is often excluded in Coifman-Meyer type estimates in a general setting.

Here, thanks to the frequency localization, we can show the desired estimates under weaker assumptions.

\begin{lem}\label{lem:CM1}
Let $N_1,N_2,N_3$ be dyadic integers, and assume that the multiplier $b:I_{N_1}\times I_{N_2}\times I_{N_3}\to \mathbb{C}$ is $C^2$ and satisfies 
\[ |\partial_{\xi_1}^{\alpha_1}\partial_{\xi_2}^{\alpha_2}\partial_{\xi_3}^{\alpha_3}b(\xi_1,\xi_2,\xi_3)|\lesssim N_1^{-\alpha_1}N_2^{-\alpha_2}N_3^{-\alpha_3}B\quad \text{on $I_{N_1}\times I_{N_2}\times I_{N_3}$}\]
for any $\alpha =(\alpha_1,\alpha_2,\alpha_3)\in \mathbb{Z}_{\geq 0}^3$ with $|\alpha|\leq 2$.
Then, given any $p,p_1,p_2,p_3\in [1,\infty]$ with $\frac1p=\frac{1}{p_1}+\frac{1}{p_2}+\frac{1}{p_3}$, it holds that
\begin{align*}
&\Big\| \mathcal{F}^{-1}\Big[ \int _{\xi =\xi_{123}}b(\xi_1,\xi_2,\xi_3)\mathcal{F}[P_{N_1}w_1] (\xi_1)\mathcal{F}[P_{N_2}w_2](\xi_2)\mathcal{F}[P_{N_3}\bar{w}_3](\xi_3)\Big] \Big\|_{L^p}\\
&\quad \lesssim B\| P_{N_1}w_1\|_{L^{p_1}}\| P_{N_2}w_2\|_{L^{p_2}}\| P_{N_3}w_3\|_{L^{p_3}}
\end{align*}
for any $w_j$ satisfying $P_{N_j}w_j\in L^2\cap L^{p_j}(\mathcal{M})$, $j=1,2,3$.
\end{lem} 

\begin{rem}\label{rem:CM}
(i) The additional assumption $P_{N_j}w_j\in L^2(\mathcal{M})$, which is relevant only if $p_j>2$ in view of the Bernstein inequality, ensures that $\mathcal{F}[P_{N_j}w_j]\in L^2\cap L^1(\hat{\mathcal{M}})$ by the frequency localization.
In particular, the integral
\[ \int _{\xi =\xi_{123}}b(\xi_1,\xi_2,\xi_3)\mathcal{F}[P_{N_1}w_1] (\xi_1)\mathcal{F}[P_{N_2}w_2](\xi_2)\mathcal{F}[P_{N_3}\bar{w}_3](\xi_3) \]
converges absolutely for each $\xi$ and gives a function in $\xi$ which belongs to $L^\infty \cap L^1(\hat{\mathcal{M}})$.
We will use the lemma in Section~\ref{sec:NFR1} only for functions $w_j$ belonging to either $H^\sigma$ for some $\sigma \in \mathbb{R}$ or $B^0_{1,\infty}$, so that the condition $P_{N_j}w_j\in L^2$ is always satisfied.
Similar comments apply to Lemmas~\ref{lem:CM2}, \ref{lem:CM3} below.

(ii) Such estimates of Coifman-Meyer type for frequency-localized multilinear Fourier multiplier operators have been used in \cite{MPV18}, where these estimates were obtained as direct consequences of more general (unlocalized) estimates established in \cite{MPTT04,B09}.
We note that Lemma~\ref{lem:CM1} also follows from the results in \cite{MPTT04,B09} (except for some specific values of $p,p_j$).
Below, we give a simple and self-contained proof, following the argument of \cite[Section~5]{H12} based on separation of variables by Fourier series expansion of the multiplier.
The same idea has been used for energy-type estimates; see, e.g., \cite{KT07,CHT12,S21,KT-shorttime}. 
We also remark that our proof has the advantage of covering the delicate case $p,p_j\in \{ 1,\infty\}$, which will actually be needed in Section~\ref{sec:NFR1}.
\end{rem}

\begin{proof}
Define a $C^2$ function $f$ on $\mathbb{R}^3$ by
\[ f(\eta_1,\eta_2,\eta_3):=\left\{ \begin{aligned}
&b(N_1\eta_1,N_2\eta_2,N_3\eta_3)\prod_{j=1}^3\tilde{\psi}_{N_j}(N_j\eta_j) \qquad \text{on $\tfrac{1}{N_1}I_{N_1}\times \tfrac{1}{N_2}I_{N_2}\times \tfrac{1}{N_3}I_{N_3}$},\\
&0\quad \text{on $[-\pi ,\pi ]^3\setminus \tfrac{1}{N_1}I_{N_1}\times \tfrac{1}{N_2}I_{N_2}\times \tfrac{1}{N_3}I_{N_3}$},\\
&\text{$2\pi$-periodic\quad in $\mathbb{R}^3$.} 
\end{aligned} \right. \]
Let $\{ \hat{f}(k)\}_{k\in \mathbb{Z}^3}$ be the Fourier coefficients of $f$, so that
\[ b(\xi_1,\xi_2,\xi_3)\prod _{j=1}^3\tilde{\psi}_{N_j}(\xi_j)=f\Big( \frac{\xi_1}{N_1},\frac{\xi_2}{N_2},\frac{\xi_3}{N_3}\Big) =\sum _{k\in \mathbb{Z}^3}e^{i(\frac{\xi_1}{N_1}k_1+\frac{\xi_2}{N_2}k_2+\frac{\xi_3}{N_3}k_3)}\hat{f}(k) \]
on $I_{N_1}\times I_{N_2}\times I_{N_3}$.
Note that $f\in C^2(\mathbb{T}^3)\subset H^2(\mathbb{T}^3)$ implies $\hat{f}\in \ell^1(\mathbb{Z}^3)$ and hence the Fourier series is absolutely convergent.
Using this representation, Minkowski, and H\"older, we have
\begin{align*}
&\Big\| \mathcal{F}^{-1}\Big[ \int _{\xi =\xi_{123}}b(\xi_1,\xi_2,\xi_3)\mathcal{F}[P_{N_1}w_1] (\xi_1)\mathcal{F}[P_{N_2}w_2](\xi_2)\mathcal{F}[P_{N_3}\bar{w}_3](\xi_3)\Big] \Big\|_{L^p}\\
&\quad \leq \sum_{k\in \mathbb{Z}^3}|\hat{f}(k)| \Big\| \mathcal{F}^{-1}\Big[ \int _{\xi =\xi_{123}}e^{i\frac{k_1}{N_1}\xi_1}\mathcal{F}[P_{N_1}w_1] (\xi_1)e^{i\frac{k_2}{N_2}\xi_2}\mathcal{F}[P_{N_2}w_2](\xi_2)e^{i\frac{k_3}{N_3}\xi_3}\mathcal{F}[P_{N_3}\bar{w}_3](\xi_3)\Big] \Big\|_{L^p}\\
&\quad =c\sum_{k\in \mathbb{Z}^3}|\hat{f}(k)| \Big\| (P_{N_1}w_1)(\cdot +\frac{k_1}{N_1})(P_{N_2}w_2)(\cdot +\frac{k_2}{N_2})(P_{N_3}\bar{w}_3)(\cdot +\frac{k_3}{N_3})\Big\|_{L^p}\\
&\quad \leq c\| \hat{f}\|_{\ell ^1(\mathbb{Z}^3)}\| P_{N_1}w_1\|_{L^{p_1}}\| P_{N_2}w_2\|_{L^{p_2}}\| P_{N_3}w_3\|_{L^{p_3}}.
\end{align*}
Now, we observe that $b(\xi_1,\xi_2,\xi_3)\prod _{j=1}^3\tilde{\psi}_{N_j}(\xi_j)$ satisfies the same estimates as that assumed for $b$, and this implies $\| f\|_{C^2(\mathbb{T}^3)}\lesssim B$, hence $\| \hat{f}\|_{\ell^1(\mathbb{Z}^3)}\lesssim  \| f\|_{C^2(\mathbb{T}^3)}\lesssim B$.
Substituting this bound to the above estimate, we obtain the claim.
\end{proof}

\begin{lem}\label{lem:CM2}
Let $N,N_1,N_2,N_3$ be dyadic integers, and assume that the multiplier $b:I_N\times I_{N_1}\times I_{N_2}\times I_{N_3}\to \mathbb{C}$ is $C^3$ and satisfies 
\[ |\partial_\xi^{\alpha_0}\partial_{\xi_1}^{\alpha_1}\partial_{\xi_2}^{\alpha_2}\partial_{\xi_3}^{\alpha_3}b(\xi,\xi_1,\xi_2,\xi_3)|\lesssim N^{-\alpha_0}N_1^{-\alpha_1}N_2^{-\alpha_2}N_3^{-\alpha_3}B\quad \text{on $I_N\times I_{N_1}\times I_{N_2}\times I_{N_3}$}\]
for any $\alpha =(\alpha_0,\alpha_1,\alpha_2,\alpha_3)\in \mathbb{Z}_{\geq 0}^4$ with $|\alpha|\leq 3$.
Then, for any $p,p_1,p_2,p_3\in [1,\infty]$ with $\frac1p=\frac{1}{p_1}+\frac{1}{p_2}+\frac{1}{p_3}$, it holds that
\begin{align*}
&\Big\| P_N\mathcal{F}^{-1}\Big[ \int _{\xi =\xi_{123}}b(\xi,\xi_1,\xi_2,\xi_3)\mathcal{F}[P_{N_1}w_1] (\xi_1)\mathcal{F}[P_{N_2}w_2](\xi_2)\mathcal{F}[P_{N_3}\bar{w}_3](\xi_3)\Big] \Big\|_{L^p}\\
&\quad \lesssim B\| P_{N_1}w_1\|_{L^{p_1}}\| P_{N_2}w_2\|_{L^{p_2}}\| P_{N_3}w_3\|_{L^{p_3}}
\end{align*}
for any $w_j$ satisfying $P_{N_j}w_j\in L^2\cap L^{p_j}(\mathcal{M})$, $j=1,2,3$.
\end{lem} 

\begin{proof}
Proof is similar to the preceding lemma.
This time there are four variables, so we use $\| \hat{f}\|_{\ell^1(\mathbb{Z}^4)}\lesssim \| f\| _{H^3(\mathbb{T}^4)}\lesssim \| f\| _{C^3(\mathbb{T}^4)}$.
\end{proof}

The above two lemmas cannot be applied directly when there is a cancellation in $\xi_{13}$ or $\xi_{23}$.
For instance, consider the case $N_1\gg N_2\sim N_3$.
We have $|\phi(\vec{\xi})|\sim N_1|\xi_{23}|$, but $|\xi_{23}|$ may be much smaller than $N_2$, and the assumption on the multiplier $b=\xi_3/(\xi_{13}\xi_{23})$ may not hold in this case (since $|\partial_{\xi_2}b|/|b|=|\xi_{23}|^{-1}\not\lesssim N_2^{-1}$). 
For such a case, we prepare the following lemma:
\begin{lem}\label{lem:CM3}
Let $N_1,N_2,N_3$ be dyadic integers, and assume that $N_1\not\sim N_2\sim N_3$.
Then, for any $p,q\in [1,\infty]$, it holds that
\begin{align*}
&\sup_{K\geq 1}\Big\| \mathcal{F}^{-1}\Big[ \int _{\begin{smallmatrix} \xi =\xi_{123} \\ K\leq |\xi_{23}|<2K\end{smallmatrix}}\frac{\xi_3}{\xi_{13}\xi_{23}}\mathcal{F}[P_{N_1}w_1] (\xi_1)\mathcal{F}[P_{N_2}w_2](\xi_2)\mathcal{F}[P_{N_3}\bar{w}_3](\xi_3)\Big] \Big\|_{L^p}\\
&\quad \lesssim \frac{N_3}{N_1\vee N_3}\| P_{N_1}w_1\|_{L^{p}}\| P_{N_2}w_2\|_{L^{q}}\| P_{N_3}w_3\|_{L^{q'}}
\end{align*}
for any $w_j$ satisfying $P_{N_1}w_1\in L^2\cap L^{p}(\mathcal{M})$, $P_{N_2}w_2\in L^2\cap L^{q}(\mathcal{M})$ and $P_{N_3}w_3\in L^2\cap L^{q'}(\mathcal{M})$.
\end{lem} 

\begin{proof}
This time we consider the Fourier series of the multiplier $b(\xi_1,\xi_3)=\frac{\xi_3}{\xi_{13}}\tilde{\psi}_{N_1}(\xi_1)\tilde{\psi}_{N_3}(\xi_3)$ defined on $I_{N_1}\times I_{N_3}$ to obtain
\[ \frac{\xi_3}{\xi_{13}}\tilde{\psi}_{N_1}(\xi_1)\tilde{\psi}_{N_3}(\xi_3)=\sum _{k=(k_1,k_3)\in \mathbb{Z}^2}e^{i(\frac{\xi_1}{N_1}k_1+\frac{\xi_3}{N_3}k_3)}\hat{f}(k),\]
with 
\[ \| \hat{f}\|_{\ell^1(\mathbb{Z}^2)}\lesssim \| f\|_{C^2(\mathbb{T}^2)}\lesssim \frac{N_3}{N_1\vee N_3}.\]
This expression gives 
\begin{align*}
&\Big\| \mathcal{F}^{-1}\Big[ \int _{\begin{smallmatrix} \xi =\xi_{123} \\ K\leq |\xi_{23}|<2K\end{smallmatrix}}\frac{\xi_3}{\xi_{13}\xi_{23}}\mathcal{F}[P_{N_1}w_1] (\xi_1)\mathcal{F}[P_{N_2}w_2](\xi_2)\mathcal{F}[P_{N_3}\bar{w}_3](\xi_3)\Big] \Big\|_{L^p}\\
&\leq \sum _{k\in \mathbb{Z}^2}|\hat{f}(k)|\Big\| \mathcal{F}^{-1}\Big[ \int _{\xi =\xi_1+\xi'}e^{i\frac{k_1}{N_1}\xi_1}\mathcal{F}[P_{N_1}w_1](\xi_1)\\
&\hspace{100pt} \times \frac{\mathbf{1}_{K\leq |\xi'|<2K}}{\xi'}\int_{\xi'=\xi_{23}}\mathcal{F}[P_{N_2}w_2](\xi_2)e^{i\frac{k_3}{N_3}\xi_3}\mathcal{F}[P_{N_3}\bar{w}_3](\xi_3)\Big] \Big\|_{L^p}\\
&\leq \| P_{N_1}w_1\|_{L^p}\sum _{k\in \mathbb{Z}^2}|\hat{f}(k)|\Big\| \mathcal{F}^{-1}\Big[ \frac{\mathbf{1}_{K\leq |\xi |<2K}}{\xi}\int _{\xi=\xi_{23}}\mathcal{F}[P_{N_2}w_2](\xi_2)e^{i\frac{k_3}{N_3}\xi_3}\mathcal{F}[P_{N_3}\bar{w}_3](\xi_3)\Big] \Big\|_{L^\infty}.
\end{align*}
By the Hausdorff-Young and the H\"older inequalities, we have
\begin{align*}
&\Big\| \mathcal{F}^{-1}\Big[ \frac{\mathbf{1}_{K\leq |\xi |<2K}}{\xi}\int _{\xi=\xi_{23}}\mathcal{F}[P_{N_2}w_2](\xi_2)e^{i\frac{k_3}{N_3}\xi_3}\mathcal{F}[P_{N_3}\bar{w}_3](\xi_3)\Big] \Big\|_{L^\infty}\\
&\quad \lesssim \Big\| \frac{\mathbf{1}_{K\leq |\xi |<2K}}{\xi}\int _{\xi=\xi_{23}}\mathcal{F}[P_{N_2}w_2](\xi_2)e^{i\frac{k_3}{N_3}\xi_3}\mathcal{F}[P_{N_3}\bar{w}_3](\xi_3) \Big\|_{L^1_\xi}\\
&\quad \lesssim \Big\| \int _{\xi=\xi_{23}}\mathcal{F}[P_{N_2}w_2](\xi_2)e^{i\frac{k_3}{N_3}\xi_3}\mathcal{F}[P_{N_3}\bar{w}_3](\xi_3) \Big\|_{L^{\infty}_\xi}\\
&\quad \lesssim \Big\| (P_{N_2}w_2)(\cdot )(P_{N_3}\bar{w}_3)(\cdot +\frac{k_3}{N_3})\Big\| _{L^1}\\
&\quad \leq \| P_{N_2}w_2\|_{L^{q}}\| P_{N_3}w_3\|_{L^{q'}}.
\end{align*}
Therefore, the claim follows.
\end{proof}


\section{Normal form reduction: Setup}
\label{sec:NFR0}

From now on, we set $\lambda =-1$ in \eqref{dnls} for simplicity.
Note that the sign of $\lambda$ is not relevant in our argument.
Furthermore, there will be no essential differences between the proofs for $\mathcal{M}=\mathbb{R}$ and $\mathbb{T}$.

In the previous section, we have seen that if $u\in L^\infty_TH^{\frac12}$ is a distributional solution of \eqref{dnls}, then its gauge transform%
\footnote{%
We will use the same letter $w$ to denote the gauge transform of the solution $u$ to \eqref{dnls} in both cases $\mathcal{M}=\mathbb{R}$ and $\mathbb{T}$.%
}
$w$ belongs to $L^\infty_TH^{\frac12}\cap L^4_TB^{0+}_{\infty,1}$ and it is a distributional solution of
\begin{equation}\label{eq:w2}
\partial_tw-i\partial_x^2w =\left\{ \begin{alignedat}{2}
&w^2\partial_x\bar{w}+\frac{i}{2}|w|^4w & &(\mathcal{M}=\mathbb{R}),\\
&w^2\partial_x\bar{w}\!-\!2P_c(w\partial_x\bar{w})w \!+\!i\Big( \frac{1}{2}P_{\neq c}(|w|^4)w\!-\!P_c(|w|^2)P_{\neq c}(|w|^2)w\Big) &\quad &(\mathcal{M}=\mathbb{T}).\end{alignedat}\right.
\end{equation}
Recall that each term on the right-hand side of \eqref{eq:w2} (especially the first cubic term) is a well-defined distribution in $L^4_TH^{-\frac12-}$.
Since the gauge transformation is invertible, the original problem of unconditional uniqueness for \eqref{dnls} is reduced to establishing uniqueness of solution $w$ to \eqref{eq:w2} with a fixed initial condition in the class $L^\infty_TH^{\frac12}\cap L^4_TB^{0+}_{\infty,1}$.


\subsection{Equation on the Fourier side}

We estimate the solution mainly on the Fourier side, introducing a new unknown function $\omega(t,\xi)$ defined by
\begin{gather}\label{def:omega}
\begin{gathered}
\omega (t,\xi ):=\mathcal{F}[S(-t)w(t)](\xi )=e^{it\xi^2}\hat{w}(t,\xi );\\
w(t)=\mathcal{F}^{-1}[e^{-it\xi^2}\omega(t,\xi)]=S(t)\check{\omega}(t).
\end{gathered}
\end{gather}
If $w\in L^\infty_TH^{\frac12}\cap L^4_TB^{0+}_{\infty,1}$ is a distributional solution of \eqref{eq:w2}, then $\omega$ solves
\begin{gather}\label{eq:omega}
\partial_t\omega (t,\xi) \ =\ \mathcal{N}^t[\omega(t)](\xi)+\mathcal{R}^t[\omega(t)](\xi),
\end{gather}
where
\begin{align*}
\mathcal{N}^t[\omega ](\xi) &:= e^{it\xi^2}\mathcal{F}\big[ \big( S(t)\check{\omega}\big)^2\partial_x\overline{S(t)\check{\omega}}\big] (\xi) -\frac{1}{2\pi}\int_{\begin{smallmatrix}\xi=\xi_{123}\\ |\xi_{13}|\wedge |\xi_{23}|<1\end{smallmatrix}}e^{it\phi(\vec{\xi})}i\xi_3\omega (\xi_1)\omega (\xi_2)\omega^*(\xi_3),\\
\mathcal{R}^t[\omega ](\xi) &:=\frac{1}{2\pi}\int_{\begin{smallmatrix}\xi=\xi_{123}\\ |\xi_{13}|\wedge |\xi_{23}|<1\end{smallmatrix}}e^{it\phi(\vec{\xi})}i\xi_3\omega (\xi_1)\omega (\xi_2)\omega^*(\xi_3) +e^{it\xi^2}\mathcal{F}\Big[ \frac{i}{2}|S(t)\check{\omega}|^4S(t)\check{\omega}\Big] (\xi )
\end{align*}
in the non-periodic case and 
\begin{align*}
\mathcal{N}^t[\omega ](\xi) &:= e^{it\xi^2}\mathcal{F}\big[ \big( S(t)\check{\omega}\big)^2\partial_x\overline{S(t)\check{\omega}}\big] (\xi) -\frac{1}{2\pi}\sum _{\begin{smallmatrix}\xi=\xi_{123}\\ \xi_{13}\xi_{23}=0\end{smallmatrix}}i\xi_3\omega (\xi_1)\omega (\xi_2)\omega^*(\xi_3),\\
\mathcal{R}^t[\omega ](\xi) &:=\frac{i\xi}{2\pi}|\omega(\xi)|^2\omega(\xi) \\
&\qquad +e^{it\xi^2}\mathcal{F}\Big[ \frac{i}{2}P_{\neq c}(|S(t)\check{\omega}|^4)S(t)\check{\omega}-iP_c(|S(t)\check{\omega}|^2)P_{\neq c}(|S(t)\check{\omega}|^2)S(t)\check{\omega}\Big] (\xi )
\end{align*}
in the periodic case,%
\footnote{%
We remove the contribution of the frequencies $(\xi_1,\xi_2,\xi_3)$ satisfying $|\xi_{13}|\wedge |\xi_{23}|<1$ from the main cubic term.
This will be useful in applying the normal form reduction.
Note that the condition $|\xi_{13}|\wedge |\xi_{23}|<1$ (or $\geq 1$) is equivalent to $\xi_{13}\xi_{23}=0$ (or $\neq 0$) in the periodic case.}
with the notation
\[ \omega^*(\xi):=\overline{\omega(-\xi)},\qquad \phi (\vec{\xi}):=\xi^2-\xi_1^2-\xi_2^2+\xi_3^2=2\xi_{13}\xi_{23}\quad \text{if $\xi=\xi_{123}$}.\]
Note that $\mathcal{N}^t[\omega ]$ is (formally) written as
\[ \mathcal{N}^t[\omega](\xi) = \frac{1}{2\pi}\int _{\begin{smallmatrix}\xi=\xi_{123}\\ |\xi_{13}|\wedge |\xi_{23}|\geq 1\end{smallmatrix}}e^{it\phi(\vec{\xi})}i\xi_3\omega (\xi_1)\omega (\xi_2)\omega ^*(\xi_3).\]
Next, we consider the dyadic decomposition of the nonlinearity:
\begin{equation}\label{defn:N1}
\mathcal{N}^t[\omega(t)]=\sum _{N,N_1,N_2,N_3}\mathcal{N}^t_{N,N_1,N_2,N_3}[\omega(t)],
\end{equation}
where
\begin{equation}\label{defn:N2}
\mathcal{N}_{N,N_1,N_2,N_3}^t[\omega](\xi):=\frac{1}{2\pi}\psi_{N}(\xi)\int_{\begin{smallmatrix}\xi=\xi_{123}\\ |\xi_{13}|\wedge |\xi_{23}|\geq 1\end{smallmatrix}}e^{it\phi(\vec{\xi})}i\xi_3\big( \psi_{N_1}\omega\big) (\xi_1)\big( \psi_{N_2}\omega\big) (\xi_2)\big( \psi_{N_3}\omega\big)^* (\xi_3).
\end{equation}
Hereafter, we use the notation 
\[ \omega_N:=\psi_N\omega,\quad\text{and similarly,}\quad w_N:=P_Nw.\]
Also, we will basically disregard constants and signs in front of the nonlinear terms which are not essential in the proof.
Furthermore, we will omit the superscript $t$ in $\mathcal{N}^t[\omega ]$, $\mathcal{R}^t[\omega]$ for notational convenience.
Hence, we simply write \eqref{defn:N1}--\eqref{defn:N2} as 
\begin{align*}
\mathcal{N}[\omega]&=\sum_{N,N_1,N_2,N_3}\mathcal{N}_{N,N_1,N_2,N_3}[\omega] \\
&:=\sum _{N,N_1,N_2,N_3}\psi_N(\xi )\int_{\begin{smallmatrix}\xi=\xi_{123}\\ |\xi_{13}|\wedge |\xi_{23}|\geq 1\end{smallmatrix}}e^{it\phi(\vec{\xi})}\xi_3\omega_{N_1}(\xi_1)\omega_{N_2}(\xi_2)\omega_{N_3}^* (\xi_3).
\end{align*}

In the proof of Proposition~\ref{prop:RS-w}, we have estimated the cubic nonlinearity in $H^{-\frac12-}$ using Lemma~\ref{lem:Sobolev-v}.
In the next proposition, we refine the argument to obtain slightly stronger control in the sum space $H^{-\frac12}+B^0_{1,\infty}:=\{ f_1+f_2:f_1\in H^{-\frac12},\,f_2\in B^0_{1,\infty}\}$, which is a Banach space equipped with the norm
\[ \| f\|_{H^{-\frac12}+B^0_{1,\infty}}:=\inf \big\{ \| f_1\|_{H^{-\frac12}}+\| f_2\|_{B^0_{1,\infty}}:f=f_1+f_2\big\} .\]
\begin{prop}\label{prop:N-weak}
We have
\[ \| S(t)\mathcal{F}^{-1}\mathcal{N}[\omega ]\|_{H^{-\frac12}+B^0_{1,\infty}}\lesssim \| w\|_{H^{\frac12}}^2\| w\|_{H^{\frac12}\cap B^{0+}_{\infty,1}}\]
for any $w\in H^{\frac12}\cap B^{0+}_{\infty,1}$ and $t\in \mathbb{R}$, where $\omega:=\mathcal{F}S(-t)w$.
More precisely, we have the following estimates:
\begin{align*}
&\Big\| \mathcal{N}[\omega ]-\Big( \sum _{N\sim N_3\gg N_1,N_2}+\sum _{N_1\vee N_2\sim N_3\gg N\gg N_1\wedge N_2}\Big) \mathcal{N}_{N,N_1,N_2,N_3}[\omega ]\Big\|_{\hat{H}^{-\frac12+}}\\
&\qquad +\Big\| \sum _{N\sim N_3\gg N_1,N_2}\mathcal{N}_{N,N_1,N_2,N_3}[\omega ]\Big\|_{\hat{H}^{-\frac12}}\\
&\qquad +\Big\| S(t)\mathcal{F}^{-1}\Big[ \sum _{N_1\vee N_2\sim N_3\gg N\gg N_1\wedge N_2}\mathcal{N}_{N,N_1,N_2,N_3}[\omega ]\Big] \Big\|_{B^0_{1,\infty}}\\
&\quad \lesssim \| w\|_{H^{\frac12}}^2\| w\|_{H^{\frac12}\cap B^{0+}_{\infty,1}}.
\end{align*}
\end{prop}

\begin{rem}
(i) Note that
\[ H^{-\frac12+}\hookrightarrow H^{-\frac12}\hookrightarrow H^{-\frac12}+B^0_{1,\infty}\hookrightarrow B^{-\frac12}_{2,\infty} \hookrightarrow H^{-\frac12-} .\]
We use the space $B^0_{1,\infty}$ only in the case $N_1\sim N_3\gg N\gg N_2$ due to the lack of summability in $N$, as we see in the proof below.
If we could also treat this case in $H^{-\frac12}$ (an $L^2$-based Sobolev space), then the normal form argument would become substantially simpler.

(ii) A multilinear version of the estimate can be shown by the same argument.
In particular, for two functions $w_1,w_2$ (and $\omega_m:=\mathcal{F}S(-t)w_m$, $m=1,2$) we have the following difference estimate:
\begin{align*}
&\big\| S(t)\mathcal{F}^{-1}\big( \mathcal{N}[\omega_1]-\mathcal{N}[\omega_2]\big) \big\|_{H^{-\frac12}+B^0_{1,\infty}}\\
&\lesssim \| (w_1,w_2)\|_{H^{\frac12}}^2\| w_1-w_2\|_{B^{0+}_{\infty,1}}+\| (w_1,w_2)\|_{H^{\frac12}}\| (w_1,w_2)\|_{H^{\frac12}\cap B^{0+}_{\infty,1}}\|w_1-w_2\|_{H^{\frac12}}.
\end{align*}
\end{rem}

\begin{proof}[Proof of Proposition~\ref{prop:N-weak}]
We first note that the ``resonant'' portion $|\xi_{13}|\wedge |\xi_{23}|<1$ excluded from $\mathcal{N}[\omega ]$ can be estimated even in $\hat{H}^{\frac12-}$; in fact, for each $N,N_1,N_2,N_3$ we have
\begin{align*}
&\Big\| \psi_N(\xi )\int_{\begin{smallmatrix}\xi=\xi_{123}\\ |\xi_{13}|\wedge |\xi_{23}|<1\end{smallmatrix}}\big| \xi_3\omega_{N_1}(t,\xi_1)\omega_{N_2}(t,\xi_2)\omega_{N_3}^*(t,\xi_3)\big| \Big\|_{\hat{H}^{\frac12-}} \\
&\lesssim \mathbf{1}_{N\sim N_2}\Big\| \int_{\xi=\xi_2+\xi'}\langle \xi_2\rangle^{\frac12-}\big| \omega_{N_2}(t,\xi_2)\big| \cdot \mathbf{1}_{|\xi'|<1}\int_{\xi'=\xi_{13}}|\xi_3|\big| \omega_{N_1}(t,\xi_1)\omega_{N_3}^*(t,\xi_3)\big| \Big\|_{L^2}+\text{(symm.~term)}\\
&\lesssim \mathbf{1}_{N\sim N_2}N^{0-}\| w_{N_2}(t)\|_{H^{\frac12}}\cdot \mathbf{1}_{N_1\sim N_3}\| w_{N_1}(t)\|_{H^{\frac12}}\| w_{N_3}(t)\|_{H^{\frac12}} +\text{(symm.~term)},
\end{align*}
which is clearly summable in $N,N_1,N_2,N_3$.
Therefore, we may estimate the whole interaction $w^2\partial_x\bar{w}$ instead of $S(t)\mathcal{F}^{-1}\mathcal{N}[\omega ]$.
We consider the estimate of localized terms
\[ I(X):=\Big\| \sum_{N,N_1,N_2,N_3}P_N(w_{N_1}w_{N_2}\partial_x\bar{w}_{N_3}) \Big\| _{X},\]
assuming $N_1\geq N_2$ and dividing the analysis into the following four cases.

(i) $N_2\gtrsim N_3$. In this case, we have $N\lesssim N_1$.
By H\"older and Bernstein, we treat this case in $H^{-\frac12+}$ as
\begin{align*}
I(H^{-\frac12+})&\lesssim \sum_{\begin{smallmatrix} N_1,N_2,N_3\\ N_1\geq N_2\gtrsim N_3\end{smallmatrix}}N_1^{0+}\| w_{N_1}\|_{L^\infty}\| w_{N_2}\|_{L^2}N_3\| w_{N_3}\|_{L^2}\lesssim \| w\|_{B^{0+}_{\infty,1}}\| w\|_{H^{\frac12}}^2.
\end{align*}

(ii) $N_1\gg N_3\gg N_2$. Similarly,
\begin{align*}
I(H^{-\frac12+})&\lesssim \sum_{N_1\gg N_3\gg N_2}N_1^{-\frac12+}\| w_{N_1}\|_{L^2}\| w_{N_2}\|_{L^\infty}N_3\| w_{N_3}\|_{L^\infty}\lesssim \| w\|_{H^{\frac12}}^2\| w\|_{B^{0+}_{\infty,1}}.
\end{align*}

(iii) $N_1\sim N_3\gg N_2$.
We further divide the case according to the size of $N$.
If $N\lesssim N_2$, we can estimate in $H^{-\frac12+}$ as
\begin{align*}
I(H^{-\frac12+})&\lesssim \sum_{N_1\sim N_3\gg N_2}N_2^{0+}\| w_{N_1}\|_{L^2}\| w_{N_2}\|_{L^\infty}N_3\| w_{N_3}\|_{L^2}\lesssim \| w\|_{H^{\frac12}}\| w\|_{B^{0+}_{\infty,1}}\| w\|_{H^{\frac12}}.
\end{align*}
If $N\sim N_1\sim N_3$, then
\begin{align*}
I(H^{-\frac12+})&\lesssim \sum_{N_1\sim N_3\gg N_2}N_1^{-\frac12+}\| w_{N_1}\|_{L^2}\| w_{N_2}\|_{L^\infty}N_3\| w_{N_3}\|_{L^\infty}\lesssim \| w\|_{H^{\frac12}}^2\| w\|_{B^{0+}_{\infty,1}}.
\end{align*}
Finally, when $N_1\sim N_3\gg N\gg N_2$, we estimate in $B^0_{1,\infty}$ to avoid summation in $N$ as
\begin{align*}
I(B^0_{1,\infty})&\lesssim \sum_{N_1\sim N_3\gg N_2}\| w_{N_1}\|_{L^2}\| w_{N_2}\|_{L^\infty}N_3\| w_{N_3}\|_{L^2}\lesssim \| w\|_{H^{\frac12}}\| w\|_{B^{0}_{\infty,1}}\| w\|_{H^{\frac12}}.
\end{align*}

(iv) $N_3\gg N_1$. We have $N\sim N_3$ and 
\begin{align*}
I(H^{-\frac12})&\lesssim \bigg[ \sum _{N}\Big( N^{-\frac12}\sum_{N_1\geq N_2}\| w_{N_1}\|_{L^\infty}\| w_{N_2}\|_{L^\infty}N\| w_{\sim N}\|_{L^2}\Big) ^2\bigg] ^{1/2}\lesssim \| w\|_{B^{0+}_{\infty,1}}\| w\|_{H^{\frac12}}^2.
\end{align*}
The proof is completed.
\end{proof}


\subsection{Classification of frequency interactions}

Starting from the equation \eqref{eq:omega}, the term $\mathcal{R}[\omega ]$ is already controlled in $\hat{H}^{\frac12}$ (see Lemma~\ref{lem:R} below).
The problem is how to treat the cubic term $\mathcal{N}[\omega]$.
Since this part cannot be controlled in $\hat{H}^{\frac12}$ due to the derivative loss, we need to apply the normal form reduction (NFR).
Basically, one application of NFR consists of a differentiation by parts%
\footnote{%
In practice, we work with the equations in the integral form. 
Hence, we actually apply an integration by parts in $t$.%
}
in $t$ and application of the product rule,
\begin{align*}
&\int_{\xi=\xi_{123}}e^{it\phi}\xi_3\omega(\xi_1)\omega(\xi_2)\omega^*(\xi_3) =\partial_t\Big[ \int_{\xi=\xi_{123}}\frac{e^{it\phi}}{i\phi}\xi_3\omega(\xi_1)\omega(\xi_2)\omega^*(\xi_3)\Big] \\
&-\int_{\xi=\xi_{123}}\frac{e^{it\phi}}{i\phi}\xi_3\Big\{ (\partial_t\omega)(\xi_1)\omega(\xi_2)\omega^*(\xi_3)+\omega(\xi_1)(\partial_t\omega)(\xi_2)\omega^*(\xi_3)+\omega(\xi_1)\omega(\xi_2)(\partial_t\omega)^*(\xi_3)\Big\} ,
\end{align*}
followed by substitution of the equation \eqref{eq:omega},
\begin{align*}
&\int_{\xi=\xi_{123}}\frac{e^{it\phi}}{i\phi}\xi_3(\partial_t\omega)(\xi_1)\omega(\xi_2)\omega^*(\xi_3)\\
&=\int_{\xi=\xi_{123}}\frac{e^{it\phi}}{i\phi}\xi_3\mathcal{R}[\omega](\xi_1)\omega(\xi_2)\omega^*(\xi_3)+\int_{\xi=\xi_{123}}\frac{e^{it\phi}}{i\phi}\xi_3\mathcal{N}[\omega ](\xi_1)\omega(\xi_2)\omega^*(\xi_3)\quad \text{(and so on)}.
\end{align*}
We will apply this procedure to each localized cubic nonlinearity $\mathcal{N}_{N,N_1,N_2,N_3}[\omega]$ in a different manner, as outlined below.
Assume $N_1\geq N_2$ by symmetry, and consider the following case separation:
\begin{align*}
\left\{ \begin{alignedat}{2}
&N_1\sim N_3 &\quad &\left\{ 
	\begin{alignedat}{2}
	&N_1\sim N_2 &\quad &\left\{ 
		\begin{aligned}
		&N\sim N_1\\
		&N_1\gg N
		\end{aligned} \right. \\
	&N_1\gg N_2 &\quad &\left\{ 
		\begin{aligned}
		&N\sim N_1\\
		&N_1\gg N\gg N_2\\
		&N\sim N_2\\
		&N_2\gg N
		\end{aligned} \right. 
	\end{alignedat} \right. \\
&N_1\gg N_3 &\quad &\left\{ 
	\begin{aligned}
	&N_1\sim N_2 \\
	&N_1\gg N_2\gg N_3 \\
	&N_2\sim N_3 \\
	&N_3\gg N_2 
	\end{aligned} \right. \\
&N_3\gg N_1 & &
\end{alignedat} \right.
\quad \begin{aligned}
\cdots \quad& N_1\sim N_2\sim N_3\sim N \\
\cdots \quad& N_1\sim N_2\sim N_3\gg N\\[2pt]
\cdots \quad& N_1\sim N_3\sim N\gg N_2\\
\cdots \quad& N_1\sim N_3\gg N\gg N_2\\
\cdots \quad& N_1\sim N_3\gg N\sim N_2\\
\cdots \quad& N_1\sim N_3\gg N_2\gg N\\[2pt]
\cdots \quad& N_1\sim N_2\gg N_3\\
\cdots \quad& (N\sim)\,N_1\gg N_2\gg N_3\\
\cdots \quad& (N\sim)\,N_1\gg N_2\sim N_3\\
\cdots \quad& (N\sim)\,N_1\gg N_3\gg N_2\\[2pt]
\cdots \quad& (N\sim)\,N_3\gg N_1\geq N_2
\end{aligned}
\end{align*}
We classify these 11 cases as in Table~\ref{table1}.
\begin{table}
\begin{center}
\begin{tabular}{|c|r@{~~}l|c|c|c|}
\hline
\multicolumn{3}{|c|}{Classification} & Deriv.~loss$^{\text{1)}}$ & Resonance$^{\text{2)}}$ & ~Control$^{\text{3)}}$~ \\ \hline
\texttt{A} & & $N_1\sim N_2\sim N_3\sim N$ & -- & $|\xi_{13}||\xi_{23}|$ & \checkmark \\ \hline
\multirow{2}{*}{\texttt{B}} & \texttt{B1}: & $N_1\sim N_3\gg N\sim N_2$ & -- & $N_1|\xi_{13}|$ & \checkmark \\ \cline{2-6}
& \texttt{B2}: & $N_1\sim N\gg N_2\sim N_3$ & -- & $N_1|\xi_{23}|$ & \checkmark \\ \hline
\multirow{2}{*}{\texttt{C}} & \texttt{C1}: & $N_1\sim N_3\gg N\gg N_2$ & $1/2$ & $NN_3$ & $B^0_{1,\infty}$ \\ \cline{2-6}
& \texttt{C2}: & $N_3\sim N\gg N_1\geq N_2$ & $1$ & $NN_3$ & $H^{-1/2}$ \\ \hline
\multirow{2}{*}{\texttt{D}} & \texttt{D1}: & $N_1\sim N\gg N_2\gg N_3$ & -- & $NN_2$ & \checkmark \\ \cline{2-6}
& \texttt{D2}: & $N_1\sim N\gg N_3\gg N_2$ & $1/2$ & $NN_3$ & \checkmark \\ \hline
\multirow{4}{*}{\texttt{E}}& \texttt{E1}: & $N_1\sim N_2\sim N_3\gg N$ & -- & $N_1N_3$ & \checkmark \\ \cline{2-6}
& \texttt{E2}: & $N_1\sim N_2\gg N_3$ & -- & $N_1^2$ & \checkmark \\ \cline{2-6}
& \texttt{E3}: & $N_1\sim N_3\sim N\gg N_2$ & $1/2$ & $NN_3$ & \checkmark \\ \cline{2-6}
& \texttt{E4}: & $N_1\sim N_3\gg N_2\gg N$ & -- & $N_2N_3$ & \checkmark \\ \hline
\end{tabular}
\caption{%
1) ``No derivative loss'' (--) means that $N^{\frac12}N_3\lesssim N_1^{\frac12}N_2^{\frac12}N_3^{\frac12}$ holds, and ``$\theta$ derivative loss'' means that one needs $N_{\max}^{\frac12+\theta}$ to bound $N^{\frac12}N_3$.
2) This means that $|\phi(\vec{\xi})|$ is similar to the indicated quantity.
Here, ``$|\xi_{13}|$'' can be much smaller than $N_1$ or $N_3$ due to cancellation,  and similarly for ``$|\xi_{23}|$''.
3) ``Good control'' (\checkmark) means that the nonlinearity (in the $w$-side) can be estimated in $H^{-\frac12+}$.
Otherwise, it is estimated only in the indicated space.}
\label{table1}
\end{center}
\end{table}
The classification is based on the following viewpoints:
\begin{itemize}
\item Resonance property.
The amount of regularity gained after NFR is determined by the size of $|\phi(\vec{\xi})|\sim |\xi_{13}||\xi_{23}|$.
In this sense, the most dangerous is Case~\texttt{A}: $N_1\sim N_2\sim N_3\sim N$, in which both $|\xi_{13}|$ and $|\xi_{23}|$ can be very small compared to the size of each frequency.
On the other hand, for instance Case~\texttt{E1}: $N_1\sim N_2\sim N_3\gg N$ is one of the most favorable, because $|\phi(\vec{\xi})|\sim |\xi-\xi_1||\xi-\xi_2|\sim N_{\max}^2$.
Considering the other cases as well, we see that in all the cases but Case~\texttt{A} we have at least one $N_{\max}$, while in Case~\texttt{B} there remains an unreliable factor $|\xi_{13}|$ or $|\xi_{23}|$.
\item Controllability in $H^{-\frac12+}$.
By Proposition~\ref{prop:N-weak}, in most cases the nonlinearity admits an estimate in $H^{-\frac12+}$ in terms of $\| w\|_{H^{\frac12}\cap B^{0+}_{\infty,1}}$.
In fact, the $H^{-\frac12+}$ control will make the argument substantially simpler.
In this sense, Case~\texttt{C} must be treated carefully.
\end{itemize}

\begin{rem}
The derivative loss in Cases~\texttt{C}, \texttt{D2}, \texttt{E3} will not affect the estimates so much, since the resonance property is good enough in these cases (especially in \texttt{C2}).
This is the feature of the nonlinearity $w^2\partial_x\bar{w}$ and the main reason why we need the gauge transformation.
In fact, if we consider a similar classification for the nonlinear term $w\bar{w}\partial_xw$, then the interaction of type \texttt{C2} has unfavorable resonance property ($|\phi|\sim N_{\max}|\xi_{12}|$), which prevents us from dealing with the derivative loss occurring in this case by the normal form argument.
\end{rem}

The detailed NFR procedure will be demonstrated in the following two sections.
We will see that infinite NFR argument is necessary only for Case~\texttt{A} (``high$\times$high$\times$high$\to$high'' interaction), while in the other cases an $\hat{H}^{\frac12}$-control is obtained after finite (at most twice) applications of NFR.

In Section~\ref{sec:NFR1}, we only transform the cubic nonlinearity of type \texttt{B}--\texttt{E}; i.e., we rewrite \eqref{eq:omega-I} as
\[ \partial_t \omega =\mathcal{R}[\omega ]+\mathcal{N}^{\texttt{A}}[\omega]+\mathcal{N}^{\texttt{B},\texttt{C},\texttt{D},\texttt{E}}[\omega] \]
(the superscripts \texttt{A}, \texttt{B}, $\ldots$ indicate the restriction of the range of $N,N_1,N_2,N_3$ to the corresponding set) and apply NFR to the last term.
This stage proceeds as follows: 
\begin{itemize}
\item After applying NFR, $\partial_t\omega$ is replaced by either $\mathcal{R}[\omega]$ or $\mathcal{N}[\omega]$.
If $\mathcal{R}[\omega]$ replaces, the resulting term can be estimated in $\hat{H}^{\frac12}$ and we can stop there.
\item If we start with $\mathcal{N}^{\texttt{E}}[\omega ]$, then after applying NFR (and replacing $\partial_t\omega$ with $\mathcal{N}[\omega]$) the good resonance property allows us to estimate the resulting term by putting $\mathcal{N}[\omega]$ in a good norm as $\| S(t)\mathcal{F}^{-1}\mathcal{N}[\omega]\|_{H^{-\frac12}+B^0_{1,\infty}}$.
Then, by Proposition~\ref{prop:N-weak} we can finish the reduction.
\item 
On the other hand, when we apply NFR to $\mathcal{N}^{\texttt{B},\texttt{C},\texttt{D}}[\omega]$, it is possible to put $\mathcal{N}[\omega]$ in $\hat{H}^{-\frac12+}$ but not in $H^{-\frac12}+B^0_{1,\infty}$ as for $\mathcal{N}^{\texttt{E}}[\omega]$.
In view of Proposition~\ref{prop:N-weak}, we can finish the procedure unless the substituting nonlinearity is $\mathcal{N}^{\texttt{C}}[\omega]$.
\item Hence, we need to apply the second NFR if the first NFR is applied to type $\mathcal{N}^{\texttt{B},\texttt{C},\texttt{D}}[\omega]$ and then $\partial_t\omega$ is replaced by $\mathcal{N}^{\texttt{C}}[\omega]$.
Then, the resonance property for \texttt{C} turns out to be good enough to control the resulting terms in $\hat{H}^{\frac12}$, and we can finish the procedure. 
\end{itemize}
Consequently, by applying NFR twice we will obtain the equation of the form
\[ \partial_t \omega=\mathcal{N}^{\texttt{A}}[\omega] +\partial_t\mathcal{W}_0[\omega]+\mathcal{W}_1[\omega] ,\]
where each of $\mathcal{W}_0[\omega]$, $\mathcal{W}_1[\omega]$ can be controlled in $\hat{H}^{\frac12}$.
We will see that $\mathcal{W}_0[\omega]$ has additional good properties so that
\begin{itemize}
\item $\| \mathcal{W}_0[\omega]\|_{L^\infty_T\hat{H}^{\frac12}}\ll \| \omega\|_{L^\infty_T\hat{H}^{\frac12}}$, and
\item $\mathcal{N}^{\texttt{A}}[\omega]-\mathcal{N}^{\texttt{A}}[\omega -\mathcal{W}_0[\omega]]$ can be controlled in $\hat{H}^{\frac12}$ (while $\mathcal{N}^{\texttt{A}}[\omega]$ is only in $\hat{H}^{-\frac12}$).
\end{itemize}
Therefore, we introduce a new unknown function $\varpi:=\omega -\mathcal{W}_0[\omega]$ and consider the equation for $\varpi$ of the form
\[ \partial_t \varpi =\mathcal{N}^{\texttt{A}}[\varpi] +\mathcal{W}_1[\omega],\]
where the term $\mathcal{W}_1[\omega]$ has been changed but it is still controlled in $\hat{H}^{\frac12}$.

In Section~\ref{sec:NFR2}, we will apply the infinite NFR scheme to the above equation (similarly to the previous results for $s>1/2$ in \cite{MY20,K-all}) and obtain a priori $L^\infty_T\hat{H}^{\frac12}$ bound for $\varpi$, which then implies that for $\omega$.


\section{Normal form reduction: I. Finite reductions}\label{sec:NFR1}

In this and the next sections, we will see the NFR procedure in detail.
For all of the nonlinear estimates in these two sections, suitable multilinear versions and difference estimates can be shown by the same argument; we will not particularly mention that. 
Computations in each step will be done only \emph{formally}, and we postpone the justification argument until Section~\ref{subsec:just}.

In this section, we carry out the first stage: the finite NFR iteration procedure to reduce the derivative nonlinearity $\mathcal{N}[\omega]$ into $\mathcal{N}^{\texttt{A}}[\omega]$.
As we see in the following lemma, the last term $\mathcal{R}[\omega]$ in the equation \eqref{eq:omega} is controlled in $\hat{H}^{\frac12}$.
\begin{lem}\label{lem:R}
We have
\[ \| \mathcal{R}[\omega ]\|_{\hat{H}^{\frac12}}\lesssim \| w\|_{H^{\frac12}}^3+\| w\|_{H^{\frac12}}^4\| w\|_{H^{\frac12}\cap B^{0+}_{\infty,1}}\]
for any $w\in H^{\frac12}\cap B^{0+}_{\infty,1}$ and $t\in \mathbb{R}$, where $\omega=\mathcal{F}S(-t)w$.
\end{lem}
\begin{proof}
We begin with the periodic case.
The contribution from the first term in $\mathcal{R}[\omega (t)]$ is easily bounded by $\| w\|_{H^{\frac12}}^3$.
For the remainder, it suffices to prove that 
\[ \big\| |w|^4w\big\|_{H^{\frac12}}+\big\| \| w\|_{L^4}^4w\big\|_{H^{\frac12}}+\big\| \| w\|_{L^2}^2|w|^2w\big\|_{H^{\frac12}}+\big\| \| w\|_{L^2}^4w\big\|_{H^{\frac12}}\lesssim \| w\|_{H^{\frac12}}^4\| w\|_{H^{\frac12}\cap B^{0+}_{\infty,1}}.\]
We focus on the first term, but the others are easier to treat.
Apply the Littlewood-Paley decomposition and write $w^5=\sum_{N,N_1,\dots ,N_5}P_N(w_{N_1}\cdots w_{N_5})$, noticing that the complex conjugation plays no role here.
We may assume $N_1\geq N_2\geq \cdots \geq N_5$.
If $N_1\gg N_2$ we estimate as
\begin{align*}
\| w^5\|_{H^{\frac12}}&\lesssim \Big[ \sum _{N}\Big( N^{\frac12}\sum _{N_1\sim N}\| w_{N_1}\|_{L^2}\sum _{N_5\leq \cdots \leq N_2\ll N_1}\| w_{N_2}\|_{L^\infty}\cdots \| w_{N_5}\|_{L^\infty}\Big) ^2\Big] ^{\frac12} \\
&\lesssim \Big[ \sum _{N}\Big( \sum _{N_1\sim N}\| w_{N_1}\|_{H^{\frac12}}\sum _{N_2\ll N_1}N_2^{0+}\| w_{N_2}\|_{L^\infty}\Big) ^2\Big] ^{\frac12}\| w\|_{B^0_{\infty,\infty}}^3 \\
&\lesssim \| w\|_{H^{\frac12}}\| w\|_{B^{0+}_{\infty,1}}\| w\|_{H^{\frac12}}^3.
\end{align*} 
If $N_1\sim N_2$, we can estimate as
\begin{align*}
\| w^5\|_{H^{\frac12}}&\lesssim \sum_{N_1}\sum _{N\lesssim N_1}\sum _{N_2\sim N_1}\sum _{N_5\leq N_4\leq N_3\leq N_2}N^{\frac12}\| w_{N_1}\|_{L^2}\| w_{N_2}\|_{L^\infty}\cdots \| w_{N_5}\|_{L^{\infty}} \\
&\lesssim \sum _{N_2}N_2^{0+}\| w_{N_2}\|_{L^\infty} \| w\|_{B^{\frac12}_{2,\infty}}\| w\|_{B^0_{\infty,\infty}}^3\\
&\lesssim \| w\|_{B^{0+}_{\infty,1}}\| w\|_{H^{\frac12}}^4.
\end{align*} 
This completes the proof in the periodic case.

For the non-periodic problem, the $\hat{H}^{\frac12}$ estimate of the first term in $\mathcal{R}[\omega ]$ has been done in the proof of Proposition~\ref{prop:N-weak}; note that we can take $\ell^2$ summation in $N$ without the factor $N^{0-}$.
The second quintic term is estimated exactly as above.
\end{proof}


\subsection{Estimate of the resonant part and the boundary term after NFR}

We separate a certain part with relatively small $|\phi(\vec{\xi})|$ as the resonant%
\footnote{%
More precisely, this part is ``near resonant''. 
In fact, the ``true resonant'' portion $|\phi (\vec{\xi})|<1$ have already been removed by the condition $|\xi_{13}|\wedge |\xi_{23}|\geq 1$.%
}
part $\mathcal{N}_R[\omega]$ from the cubic nonlinearity $\mathcal{N}[\omega]$.
Let $M>1$ be the threshold value for the resonant part to be chosen later.
For Case~\texttt{A}, we define it by
\begin{align*}
\mathcal{N}^{\texttt{A}}_R[\omega ]&:=\sum _{N,N_1,N_2,N_3}^{\texttt{A}}\mathcal{N}_{N,N_1,N_2,N_3}^{\leq M}[\omega ] \\
&:=\sum _{N,N_1,N_2,N_3}^{\texttt{A}}\psi_N(\xi )\int _{\begin{smallmatrix}\xi=\xi_{123}\\ |\xi_{13}|\wedge |\xi_{23}|\geq 1\end{smallmatrix}}e^{it\phi(\vec{\xi})}\mathbf{1}_{|\phi(\vec{\xi})|\leq M}\xi_3\omega_{N_1}(\xi_1)\omega_{N_2}(\xi_2)\omega ^*_{N_3}(\xi_3).
\end{align*}
(We put \texttt{A}, \texttt{B}, $\ldots$ or \texttt{B1}, \texttt{C1}, $\ldots$ on the summation symbol to indicate the restriction of the range of $N,N_1,N_2,N_3$ to the corresponding set.)
For Cases~\texttt{C}, \texttt{D} and \texttt{E}, we do not use the cutoff $\mathbf{1}_{|\phi(\vec{\xi})|\leq M}$, but instead restrict the range of $N,N_1,N_2,N_3$ as
\begin{align*}
\mathcal{N}^{\texttt{C},\texttt{D},\texttt{E}}_R[\omega ]&:=\sum _{\begin{smallmatrix} N,N_1,N_2,N_3\\ \Phi\leq M\end{smallmatrix}}^{\texttt{C},\texttt{D},\texttt{E}}\mathcal{N}_{N,N_1,N_2,N_3}[\omega ]\\
&\;=\sum _{\begin{smallmatrix} N,N_1,N_2,N_3\\ \Phi\leq M\end{smallmatrix}}^{\texttt{C},\texttt{D},\texttt{E}}\psi_N(\xi )\int _{\xi=\xi_{123}}e^{it\phi(\vec{\xi})}\xi_3\omega_{N_1}(\xi_1)\omega_{N_2}(\xi_2)\omega ^*_{N_3}(\xi_3), 
\end{align*}
where $\Phi$ denotes the dyadic size of $|\phi(\vec{\xi})|$ in each case given in Table~\ref{table1}; for instance, $\Phi=N_1N_3$ in the case \texttt{E1}, $\Phi=NN_3$ for \texttt{E3}, and so on.
Note that the condition $|\xi_{13}|\wedge |\xi_{23}|\geq 1$ is automatically satisfied in these cases due to the dyadic localization.
The resonant part for \texttt{B} is defined similarly, but with an additional dyadic decomposition, as follows:
\begin{align*}
\mathcal{N}^{\texttt{B}}_R[\omega ]&:=\sum _{\begin{smallmatrix}N,N_1,N_2,N_3,K\\ N_1K\leq M\end{smallmatrix}}^{\texttt{B1}}\mathcal{N}_{N,N_1,N_2,N_3}^{1;K}[\omega ]+\sum _{\begin{smallmatrix}N,N_1,N_2,N_3,K\\ N_1K\leq M\end{smallmatrix}}^{\texttt{B2}}\mathcal{N}_{N,N_1,N_2,N_3}^{2;K}[\omega ]\\
&:=\sum _{\begin{smallmatrix}N,N_1,N_2,N_3,K\\ N_1K\leq M\end{smallmatrix}}^{\texttt{B1}}\psi_N(\xi)\int _{\begin{smallmatrix}\xi=\xi_{123}\\ K\leq |\xi_{13}|<2K\end{smallmatrix}}e^{it\phi(\vec{\xi})}\xi_3\omega_{N_1}(\xi_1)\omega_{N_2}(\xi_2)\omega^*_{N_3}(\xi_3) \\
&\;\quad +\sum _{\begin{smallmatrix}N,N_1,N_2,N_3,K \\ N_1K\leq M\end{smallmatrix}}^{\texttt{B2}}\psi_N(\xi)\int _{\begin{smallmatrix}\xi=\xi_{123}\\ K\leq |\xi_{23}|<2K\end{smallmatrix}}e^{it\phi(\vec{\xi})}\xi_3\omega_{N_1}(\xi_1)\omega_{N_2}(\xi_2)\omega^*_{N_3}(\xi_3) .
\end{align*}
We will need the decomposition in $K$ to estimate the non-resonant part later.
 
\begin{lem}\label{lem:N_R}
We have
\[ \| \mathcal{N}_R[\omega ]\|_{\hat{H}^{\frac12}}\lesssim M^{\frac12+}\| \omega\|_{\hat{H}^{\frac12}}^3\]
for any $\omega \in \hat{H}^{\frac12}$ and $t\in \mathbb{R}$.
\end{lem}
\begin{proof}
For Cases~\texttt{A} and \texttt{B}, there is no derivative loss; i.e., it holds that $N^{\frac12}N_3\lesssim N_1^{\frac12}N_2^{\frac12}N_3^{\frac12}$.
Then, for each fixed $N,N_1,N_2,N_3$ in these cases, the localized component of $\| \mathcal{N}_R[\omega ]\|_{\hat{H}^{\frac12}}$ is bounded (pointwise in $t$) by
\begin{align*}
&N^{\frac12}\Big\| \int _{\begin{smallmatrix}\xi=\xi_{123}\\ |\xi_{13}|\wedge |\xi_{23}|\geq 1\end{smallmatrix}}\Big( \frac{M}{|\xi_{13}||\xi_{23}|}\Big) ^{\frac12+}N_3|\omega_{N_1}(\xi_1)||\omega_{N_2}(\xi_2)||\omega^*_{N_3}(\xi_3)| \Big\| _{L^2_\xi}\\
&\lesssim M^{\frac12+}N_1^{\frac12}N_2^{\frac12}N_3^{\frac12}\sup_{\xi} \Big( \int _{\begin{smallmatrix}\xi=\xi_{123}\\ |\xi_{13}|\wedge |\xi_{23}|\geq 1\end{smallmatrix}}\frac{1}{|\xi_{13}|^{1+}|\xi_{23}|^{1+}}\Big) ^{\frac12}\prod _{l=1}^3\| \omega_{N_l}\|_{L^2}\\
&\lesssim \left\{ \begin{alignedat}{2}
&M^{\frac12+}\prod _{l=1}^3\| \omega_{N_l}\|_{\hat{H}^{\frac12}} &\quad &\text{(Case \texttt{A})},\\
&M^{\frac12+}N_{\max}^{0-}\prod _{l=1}^3\| \omega_{N_l}\|_{\hat{H}^{\frac12}} &\quad &\text{(Case \texttt{B})},
\end{alignedat} \right.
\end{align*}
where we have used the Cauchy-Schwarz inequality in $\xi_j$'s.
(For Case~\texttt{B}, we can derive a factor $N_{\max}^{0-}$ from $(|\xi_{13}||\xi_{23}|)^{0-}$.) 
Since for Case~\texttt{A} all the dyadic frequencies are comparable, in both cases we can take the $\ell^1_{N,N_1,N_2,N_3}$ summation to obtain the bound $M^{\frac12+}\| \omega\|_{\hat{H}^{\frac12}}^3$.

For Cases~\texttt{C}, \texttt{D} and \texttt{E}, we have a good resonance condition $|\phi (\vec{\xi})|\gtrsim NN_3$.
Hence, we can insert $(\frac{M}{NN_3})^{(1/2)+}$ and apply the Young, Cauchy-Schwarz inequalities, obtaining the bound
\begin{align*}
&\sum_{N,N_l}N^{\frac12}\Big\| \int _{\xi=\xi_{123}}\Big( \frac{M}{NN_3}\Big) ^{\frac12+}N_3|\omega_{N_1}(\xi_1)||\omega_{N_2}(\xi_2)||\omega^*_{N_3}(\xi_3)| \Big\| _{L^2_\xi}\\
&= M^{\frac12+}\sum_{N,N_l}N^{0-}N_3^{\frac12-}\Big\| |\omega_{N_1}|*|\omega_{N_2}|*|\omega^*_{N_3}|\Big\|_{L^2}\\
&\leq M^{\frac12+}\sum_{N,N_l}N^{0-}N_3^{\frac12-}\| \omega_{N_{\max}}\|_{L^2} \| \omega_{N_{med}}\|_{L^1}\| \omega_{N_{\min}}\|_{L^1}\\
&\lesssim M^{\frac12+}\sum_{N,N_l}N^{0-}N_{\max}^{0-}\prod_{l=1}^3 \| \omega_{N_l}\|_{\hat{H}^{\frac12}}\\
&\lesssim M^{\frac12+}\| \omega\|_{\hat{H}^{\frac12}}^3.\qedhere
\end{align*}
\end{proof}

For the non-resonant part $\mathcal{N}_{N\!R}[\omega ]:=\mathcal{N}[\omega ]-\mathcal{N}_R[\omega]$, we apply differentiation by parts as
\[ \mathcal{N}_{N\!R}[\omega ] = \partial_t \mathcal{N}_0[\omega ] + \mathcal{N}_1[\omega ,\partial_t\omega] .\]
The precise definition of $\mathcal{N}_0$ and $\mathcal{N}_1$ depends on the cases.
For instance, in Case~\texttt{B1} they are defined as
\begin{align*}
\mathcal{N}^{\texttt{B1}}_0[\omega ]&:=\sum _{\begin{smallmatrix}N_1\sim N_3\gg N\sim N_2 \\ K;\,N_1K>M\end{smallmatrix}}\psi_N(\xi )\int _{\begin{smallmatrix}\xi=\xi_{123}\\ K\leq |\xi_{13}|<2K\end{smallmatrix}}\frac{e^{it\phi(\vec{\xi})}}{\phi(\vec{\xi})}\xi_3\omega_{N_1}(\xi_1)\omega_{N_2}(\xi_2)\omega^*_{N_3}(\xi_3),\\
\mathcal{N}^{\texttt{B1}}_1[\omega ,\zeta]&:=\sum _{\begin{smallmatrix}N_1\sim N_3\gg N\sim N_2 \\ K;\,N_1K>M\end{smallmatrix}}\psi_N(\xi )\int _{\begin{smallmatrix}\xi=\xi_{123}\\ K\leq |\xi_{13}|<2K\end{smallmatrix}}\frac{e^{it\phi(\vec{\xi})}}{\phi(\vec{\xi})}\xi_3\\
&\quad \times \Big\{ \zeta_{N_1}(\xi_1)\omega_{N_2}(\xi_2)\omega^*_{N_3}(\xi_3)+\omega_{N_1}(\xi_1)\zeta_{N_2}(\xi_2)\omega^*_{N_3}(\xi_3)+\omega_{N_1}(\xi_1)\omega_{N_2}(\xi_2)\zeta^*_{N_3}(\xi_3)\Big\} .
\end{align*}
(Recall that we neglect some constants and signs.)
The next lemma gives the control of the boundary term $\mathcal{N}_0$.
\begin{lem}\label{lem:N_0}
We have
\[ \| \mathcal{N}_0[\omega ]\|_{\hat{H}^{\frac12}}\lesssim M^{-\frac12+}\| \omega\|_{\hat{H}^{\frac12}}^3,\qquad \| \mathcal{N}_1[\omega ,\zeta ]\|_{\hat{H}^{\frac12}}\lesssim M^{-\frac12+}\| \omega\|_{\hat{H}^{\frac12}}^2\| \zeta \|_{\hat{H}^{\frac12}}\]
for any $\omega ,\zeta \in \hat{H}^{\frac12}$ and $t\in \mathbb{R}$.
\end{lem}
\begin{proof}
These estimates are shown by modifying the proof of the preceding lemma.
For Cases~\texttt{A} and \texttt{B}, it suffices to replace $(\frac{M}{|\xi_{13}||\xi_{23}|})^{\frac12+}$ by $(|\xi_{13}||\xi_{23}|)^{-\frac12-}M^{-\frac12+}$ and make the same argument.
For Cases~\texttt{C}, \texttt{D}, \texttt{E}, it suffices to replace $(\frac{M}{NN_3})^{\frac12+}$ by $(NN_3)^{-\frac12-}M^{-\frac12+}$.
\end{proof}

\begin{rem}\label{rem:N_R}
The proof of Lemma~\ref{lem:N_R} shows the absolute convergence of the sum of localized components:
\begin{align*}
&\sum_{N,N_1,N_2,N_3}^{\texttt{A}}\| \mathcal{N}_{N,N_1,N_2,N_3}^{\leq M}[\omega ]\|_{\hat{H}^{\frac12}}+\sum _{\begin{smallmatrix} N,N_1,N_2,N_3\\ \Phi\leq M\end{smallmatrix}}^{\texttt{C},\texttt{D},\texttt{E}}\| \mathcal{N}_{N,N_1,N_2,N_3}[\omega ]\|_{\hat{H}^{\frac12}}\\
&\quad +\sum _{\begin{smallmatrix}N,N_1,N_2,N_3,K\\ N_1K\leq M\end{smallmatrix}}^{\texttt{B1}}\| \mathcal{N}_{N,N_1,N_2,N_3}^{1;K}[\omega ]\|_{\hat{H}^{\frac12}}+\sum _{\begin{smallmatrix}N,N_1,N_2,N_3,K\\ N_1K\leq M\end{smallmatrix}}^{\texttt{B2}}\| \mathcal{N}_{N,N_1,N_2,N_3}^{2;K}[\omega ]\|_{\hat{H}^{\frac12}}<\infty
\end{align*}
for any $\omega \in \hat{H}^{\frac12}$.
In the same manner, from the proof of Lemma~\ref{lem:N_0} we see absolute convergence of the localized components for $\mathcal{N}_0[\omega]$ and $\mathcal{N}_1[\omega ,\zeta]$ in $\hat{H}^{\frac12}$, for any $\omega ,\zeta \in \hat{H}^{\frac12}$.
\end{rem}

Next, we plug the equation $\partial_t\omega =\mathcal{N}[\omega ]+\mathcal{R}[\omega ]$ into $\mathcal{N}_1[\omega,\partial_t\omega ]$.
The term $\mathcal{N}_1[\omega ,\mathcal{R}[\omega ]]$ is estimated in $\hat{H}^{\frac12}$ by the second estimate in Lemma~\ref{lem:N_0} and Lemma~\ref{lem:R}, while the term $\mathcal{N}_1[\omega ,\mathcal{N}[\omega]]$ will be considered in the following subsections.
In practice, we do not apply the first NFR to the term $\mathcal{N}^{\texttt{A}}[\omega]$; i.e., as a new equation we consider 
\begin{equation}\label{eq:NFR1}
\partial_t \omega= \mathcal{R}[\omega ]+\mathcal{N}^{\texttt{A}}[\omega] + \mathcal{N}^{\neg \texttt{A}}_{R}[\omega ] +\partial_t \mathcal{N}^{\neg \texttt{A}}_0[\omega ] + \mathcal{N}^{\neg \texttt{A}}_1[\omega,\mathcal{R}[\omega]]+\mathcal{N}^{\neg \texttt{A}}_1[\omega,\mathcal{N}[\omega]] ,
\end{equation}
where the superscript $\neg \texttt{A}$ denotes the sum of all the cases but \texttt{A}.


\subsection{Estimate of the remainder after NFR: Case \texttt{E}}

We proceed to handling the term $\mathcal{N}^{\neg \texttt{A}}_1[\omega ,\mathcal{N}[\omega ]]$, which requires the $B^{0+}_{\infty,1}$ bound on the $w$ side.
Here, we first consider easy Case~\texttt{E}.
In fact, the $\hat{H}^{\frac12}$-estimate of $\mathcal{N}^{\texttt{E}}_1[\omega ,\mathcal{N}[\omega]]$ follows from the next lemma and Proposition~\ref{prop:N-weak}.
\begin{lem}\label{lem:N_1E}
We have
\[ \| \mathcal{N}_1^{\texttt{E}}[\omega ,\zeta]\|_{\hat{H}^{\frac12}}\lesssim \| w\|_{H^{\frac12}}\| w\|_{B^{0+}_{\infty,1}}\| z\|_{H^{-\frac12}+B^0_{1,\infty}} \]
for any $w\in H^{\frac12}\cap B^{0+}_{\infty,1}$ and $z\in H^{-\frac12}+B^0_{1,\infty}$, where $\omega=\mathcal{F}S(-t)w$ and $\zeta=\mathcal{F}S(-t)z$.
\end{lem}

\begin{proof}
We focus on the term in which $\zeta$ replaces the first $\omega$ in $\mathcal{N}_0[\omega ]$:
\begin{align*}
&\sum _{N,N_j;\, \Phi >M}^{\texttt{E}}\psi_N(\xi)\int _{\xi=\xi_{123}}\frac{e^{it\phi(\vec{\xi})}}{\phi(\vec{\xi})}\xi_3 \zeta_{N_1}(\xi_1) \omega_{N_2}(\xi_2) \omega^*_{N_3}(\xi_3) \\
&=c\sum _{N,N_j;\, \Phi >M}^{\texttt{E}}\psi_N(\xi)e^{it\xi^2}\int _{\xi=\xi_{123}}\frac{\xi_3}{\xi_{13}\xi_{23}}\hat{z}_{N_1}(\xi_1)\hat{w}_{N_2}(\xi_2)\hat{\bar{w}}_{N_3}(\xi_3).
\end{align*}
(The following argument is applicable to the other two terms as well.)

\medskip\noindent
\underline{Case \texttt{E1} ($N_1\sim N_2\sim N_3\gg N$)}:
We apply Lemma~\ref{lem:CM2} with
\[ b(\xi,\xi_1,\xi_2,\xi_3)=\frac{\xi_3}{(\xi-\xi_1)(\xi-\xi_2)}.\]
It is easy to see that this $b$ satisfies the assumption for the lemma with $B=\frac{N_3}{N_1N_2}\sim N_1^{-1}$.
We then obtain the bound
\[ \sum_{N,N_j}^{\texttt{E1}}N^{\frac12}N_1^{-1}\| z_{N_1}\|_{L^2}\| w_{N_2}\|_{L^\infty}\| w_{N_3}\|_{L^\infty} \lesssim \| z\|_{H^{-\frac12-}}\| w\|_{B^{0+}_{\infty,1}}\| w\|_{H^{\frac12}}.\]

\medskip\noindent
\underline{Case \texttt{E2} ($N_1\sim N_2\gg N_3$)}: 
We use Lemma~\ref{lem:CM1} with
\[ b(\xi_1,\xi_2,\xi_3)=\frac{\xi_3}{\xi_{13}\xi_{23}},\]
which satisfies the assumption with $B=\frac{N_3}{N_1N_2}\ll N_1^{-1}$.
The rest of the proof is similar to \texttt{E1}.

\medskip\noindent
\underline{Case \texttt{E3} ($N_1\sim N_3\sim N\gg N_2$)}:
We apply H\"older to obtain the bound
\[ \sum _{N_1\sim N_3\sim N\gg N_2}N\Big\| P_N\mathcal{F}^{-1}\Big[ \int _{\xi=\xi_{123}}\frac{\xi_3}{\xi_{13}\xi_{23}}\hat{z}_{N_1}(\xi_1)\hat{w}_{N_2}(\xi_2)\hat{\bar{w}}_{N_3}(\xi_3)\Big] \Big\|_{L^1}. \] 
Then, Lemma~\ref{lem:CM2} with 
\[ b(\xi,\xi_1,\xi_2,\xi_3)=\frac{\xi_3}{\xi_{23}(\xi-\xi_2)},\]
which satisfies the assumption with $B=\frac{N_3}{N_3N}=N^{-1}$, implies the bounds
\[ \sum _{N_1\sim N_3\gg N_2}\| z_{N_1}\|_{L^2}\| w_{N_2}\|_{L^\infty}\| w_{N_3}\|_{L^2}\lesssim \| z\|_{H^{-\frac12}}\| w\|_{B^{0}_{\infty,1}}\| w\|_{H^{\frac12}} \]
and 
\[ \sum _{N_1\sim N_3\gg N_2}\| z_{N_1}\|_{L^1}\| w_{N_2}\|_{L^\infty}\| w_{N_3}\|_{L^\infty}\lesssim \| z\|_{B^{0}_{1,\infty}}\| w\|_{H^{\frac12}}\| w\|_{B^{0+}_{\infty,1}}. \]
Combining them, we obtain the desired bound $\| z\|_{H^{-\frac12}+B^0_{1,\infty}}\| w\|_{B^{0+}_{\infty,1}}\| w\|_{H^{\frac12}}$.

\medskip\noindent
\underline{Case \texttt{E4} ($N_1\sim N_3\gg N_2\gg N$)}:
Proof is similar to \texttt{E3}. 
We use Lemma~\ref{lem:CM2} with the same $b$, which satisfies the assumption with $B=\frac{N_3}{N_3N_2}\ll N^{-1}$.
Note that the summation in $N$ gives a harmless factor $N_2^{0+}$.
\end{proof}


\subsection{Estimate of the remainder after NFR: Cases \texttt{C}, \texttt{D}}

Next, we consider Cases~\texttt{C} and \texttt{D}, where the treatment depends on which $\omega$ in $\mathcal{N}_0[\omega ]$ is replaced by $\zeta =\mathcal{N}[\omega ]$. 
We write $\mathcal{N}_1[\omega,\zeta]=\sum_{j=1}^3\mathcal{N}_{1,j}[\omega,\zeta]$, where $\mathcal{N}_{1,j}[\omega,\zeta]$ denotes the term in which $\zeta$ replaces the $j$-th $\omega$ in $\mathcal{N}_0[\omega]$.
\begin{lem}\label{lem:N_1CD}
We have
\begin{align*}
\| \mathcal{N}_{1,1}^{\texttt{C2}}[\omega ,\zeta]\|_{\hat{H}^{\frac12}}+\| \mathcal{N}_{1,2}^{\texttt{C},\texttt{D}}[\omega ,\zeta]\|_{\hat{H}^{\frac12}}+\| \mathcal{N}_{1,3}^{\texttt{D}}[\omega ,\zeta]\|_{\hat{H}^{\frac12}}
&\lesssim \| w\|_{H^{\frac12}}\| w\|_{B^{0+}_{\infty,1}}\| z\|_{H^{-\frac12-}} ,\\
\| \mathcal{N}_{1,1}^{\texttt{C1}}[\omega ,\zeta]\|_{\hat{H}^{\frac12}}+\| \mathcal{N}_{1,3}^{\texttt{C1}}[\omega ,\zeta]\|_{\hat{H}^{\frac12}}
&\lesssim \| w\|_{H^{\frac12}}\| w\|_{H^{\frac12}\cap B^{0+}_{\infty,1}}\| z\|_{H^{-\frac12+}+B^0_{1,\infty}},\\
\| \mathcal{N}_{1,1}^{\texttt{D}}[\omega ,\zeta]\|_{\hat{H}^{\frac12}}+\| \mathcal{N}_{1,3}^{\texttt{C2}}[\omega ,\zeta]\|_{\hat{H}^{\frac12}}
&\lesssim \| w\|_{H^{\frac12}}\| w\|_{B^{0+}_{\infty,1}}\| z\|_{H^{-\frac12}},
\end{align*}
where $\omega=\mathcal{F}S(-t)w$ and $\zeta=\mathcal{F}S(-t)z$.
\end{lem}

\begin{proof}
We need to estimate
\begin{equation}\label{norm:NFR0-N1-C}
N^{\frac12}\Big\| P_N\mathcal{F}^{-1}\Big[ \int _{\xi=\xi_{123}}\frac{\xi_3}{\phi(\vec{\xi})}\hat{w}_{N_1}(\xi_1)\hat{w}_{N_2}(\xi_2)\hat{\bar{w}}_{N_3}(\xi_3)\Big] \Big\|_{L^2}\quad \text{with one of $w$'s replaced by $z$}
\end{equation}
for each $N,N_j$ and sum them up.
Note that $|\phi (\vec{\xi})|\gtrsim NN_3$ in the cases we consider; in particular, we can use either Lemma~\ref{lem:CM1} or Lemma~\ref{lem:CM2} with $B\lesssim N^{-1}$.
We denote by $N_\sharp$ the frequency of $z$, the function we want to estimate in $H^{-\frac12-}$ or $H^{-\frac12}+B^0_{1,\infty}$, and the other two frequencies by $N_a$ and $N_b$, assuming $N_a\geq N_b$.
We also define $N_{\max}:=\max \{ N,N_1,N_2,N_3\}\sim \max \{ N_{\sharp},N_a\}$.
\begin{enumerate}
\item[(I)] For $\mathcal{N}_{1,1}^{\texttt{C2}}$, $\mathcal{N}_{1,2}^{\texttt{C},\texttt{D}}$ and $\mathcal{N}_{1,3}^{\texttt{D}}$, we have $N_a\sim N_{\max}$ and $N\gg N_\sharp$.
Estimating as
\[ \eqref{norm:NFR0-N1-C}\lesssim N^{-\frac12}\| z_{N_\sharp}\|_{L^2}\| w_{N_a}\|_{L^\infty}\| w_{N_b}\|_{L^\infty}\lesssim N_{\max}^{0-}\| z_{N_\sharp} \|_{H^{-\frac12-}}N_a^{0+}\| w_{N_a}\|_{L^\infty}\| w_{N_b}\|_{H^{\frac12}},\]
which is summable in $N,N_1,N_2,N_3$, we obtain the bound $\| z\|_{H^{-\frac12-}}\| w\|_{B^{0+}_{\infty,1}}\| w\|_{H^{\frac12}}$.
\item[(II)] For $\mathcal{N}_{1,1}^{\texttt{C1}}$ and $\mathcal{N}_{1,3}^{\texttt{C1}}$, we have $N_\sharp \sim N_a\gg N\gg N_b$.
We first apply Bernstein to replace $L^2_x$ by $L^1_x$ with additional $N^{\frac12}$, then apply Lemma~\ref{lem:CM2} with $B\sim N^{-1}$, to obtain the bound
\[ \eqref{norm:NFR0-N1-C}\lesssim \| z_{N_\sharp}\|_{L^1}\| w_{N_a}\|_{L^\infty}\| w_{N_b}\|_{L^\infty}\lesssim N_{\max}^{0-}\| z_{N_\sharp}\|_{L^1}N_a^{0+}\| w_{N_a}\|_{L^\infty}\| w_{N_b}\|_{H^{\frac12}}.\]
This yields the bound $\| z\|_{B^0_{1,\infty}}\| w\|_{B^{0+}_{\infty,1}}\| w\|_{H^{\frac12}}$.
Also, we have
\[ \eqref{norm:NFR0-N1-C}\lesssim \| z_{N_\sharp}\|_{L^2}\| w_{N_a}\|_{L^2}\| w_{N_b}\|_{L^\infty}\lesssim N_{\max}^{0-}\| z_{N_\sharp}\|_{H^{-\frac12+}}\| w_{N_a}\|_{H^{\frac12}}\| w_{N_b}\|_{H^{\frac12}},\]
which yields the bound $\| z\|_{H^{-\frac12+}}\| w\|_{H^{\frac12}}^2$.
\item[(III)] For $\mathcal{N}_{1,1}^{\texttt{D}}$ and $\mathcal{N}_{1,3}^{\texttt{C2}}$, we have $N\sim N_\sharp \gg N_a\geq N_b$. 
Similarly to (I), we estimate as 
\[ \eqref{norm:NFR0-N1-C}\lesssim N^{-\frac12}\| z_{N_\sharp}\|_{L^2}\| w_{N_a}\|_{L^\infty}\| w_{N_b}\|_{L^\infty}\lesssim N_{a}^{0-}\| z_{N_\sharp} \|_{H^{-\frac12}}N_a^{0+}\| w_{N_a}\|_{L^\infty}\| w_{N_b}\|_{H^{\frac12}}\]
(and take $\ell^2$ sum in $N_\sharp\sim N$) to obtain the bound $\| z\|_{H^{-\frac12}}\| w\|_{B^{0+}_{\infty,1}}\| w\|_{H^{\frac12}}$.\qedhere
\end{enumerate}
\end{proof}

\begin{rem}\label{rem:N_1CD}
For (II), the latter argument shows
\begin{align*}
&\| \psi_N\mathcal{N}_{1,1}^{\texttt{C1}}[\omega ,\zeta]\|_{\hat{H}^{\frac12}}+\| \psi_N\mathcal{N}_{1,3}^{\texttt{C1}}[\omega ,\zeta]\|_{\hat{H}^{\frac12}}\\
&\quad \lesssim \sum_{N_\sharp\sim N_a\gg N_b}\| z_{N_\sharp}\|_{H^{-\frac12}}\| w_{N_a}\|_{H^{\frac12}}\| w_{N_b}\|_{L^\infty} \lesssim \| z\|_{H^{-\frac12}}\| w\|_{H^{\frac12}}\| w\|_{B^0_{\infty,1}},
\end{align*}
which yields the bound with $\| z\|_{H^{-\frac12}}$ for \emph{fixed} $N$.
However, we cannot take $\ell^2$ sum in $N$.

For (III), the argument in (II) shows 
\[ \eqref{norm:NFR0-N1-C}\lesssim \| z_{N_\sharp}\|_{L^1}\| w_{N_a}\|_{L^\infty}\| w_{N_b}\|_{H^{\frac12}}.\]
Although we can show a bound with $\| z\|_{B^0_{1,2}}$ by taking $\ell^2$ sum in $N_\sharp\sim N$, this is not sufficient for the bound with $\| z\|_{B^0_{1,\infty}}$.
\end{rem}

By Lemma~\ref{lem:N_1CD} and Proposition~\ref{prop:N-weak}, the $\hat{H}^{\frac12}$ norms of 
\begin{gather*}
\mathcal{N}_{1,1}^{\texttt{C2}}[\omega ,\mathcal{N}[\omega ]],\qquad \mathcal{N}_{1,2}^{\texttt{C},\texttt{D}}[\omega ,\mathcal{N}[\omega ]],\qquad \mathcal{N}_{1,3}^{\texttt{D}}[\omega ,\mathcal{N}[\omega ]],\\
\mathcal{N}_{1,1}^{\texttt{C1}}[\omega ,\mathcal{N}^{\neg \texttt{C2}}[\omega ]],\quad \mathcal{N}_{1,3}^{\texttt{C1}}[\omega ,\mathcal{N}^{\neg \texttt{C2}}[\omega ]],\quad 
\mathcal{N}_{1,1}^{\texttt{D}}[\omega ,\mathcal{N}^{\neg \texttt{C1}}[\omega ]],\quad \mathcal{N}_{1,3}^{\texttt{C2}}[\omega ,\mathcal{N}^{\neg \texttt{C1}}[\omega ]]
\end{gather*}
are bounded by $\| w\|_{H^{\frac12}}^3\| w\|_{H^{\frac12}\cap B^{0+}_{\infty,1}}^2$.
We now apply NFR once more to the remaining terms
\begin{equation}\label{defn:Lj1}
\begin{alignedat}{2}
\mathcal{L}^1[\omega] &:= \mathcal{N}_{1,1}^{\texttt{C1}}[\omega ,\mathcal{N}^{\texttt{C2}}[\omega ]], & \mathcal{L}^2[\omega] &:= \mathcal{N}_{1,3}^{\texttt{C1}}[\omega ,\mathcal{N}^{\texttt{C2}}[\omega ]],\\
\mathcal{L}^3[\omega] &:= \mathcal{N}_{1,1}^{\texttt{D}}[\omega ,\mathcal{N}^{\texttt{C1}}[\omega ]], &\qquad \mathcal{L}^4[\omega] &:= \mathcal{N}_{1,3}^{\texttt{C2}}[\omega ,\mathcal{N}^{\texttt{C1}}[\omega ]].
\end{alignedat}
\end{equation}
As an example, $\mathcal{L}^1[\omega]$ is given by
\begin{align*}
\mathcal{L}^1[\omega]=\sum _{\begin{smallmatrix} N_1\sim N_3\gg N\gg N_2 \\ NN_3>M \\ L_3\sim N_1\gg L_1,L_2\end{smallmatrix}} \int _{\begin{smallmatrix}\xi =\xi_{123} \\ \xi_1=\eta_{123}\end{smallmatrix}}&\frac{e^{it\phi (\vec{\xi})+it\phi (\vec{\eta})}}{\phi (\vec{\xi})}\xi_3\eta_3\psi_{N}(\xi)\psi_{N_1}(\xi_1) \\[-15pt]
&\quad \times \omega_{L_1}(\eta_1)\omega_{L_2}(\eta_2)\omega^*_{L_3}(\eta_3)\omega_{N_2}(\xi_2)\omega^*_{N_3}(\xi_3),
\end{align*}
where $\vec{\xi}=(\xi,\xi_1,\xi_2,\xi_3)$ and $\vec{\eta}=(\xi_1,\eta_1,\eta_2,\eta_3)$.
We separate from this term the resonant part roughly defined by ``$|\phi (\vec{\xi})+\phi (\vec{\eta})|\lesssim |\phi (\vec{\xi})|$''.
Since $|\phi (\vec{\eta})|\sim N_1L_3$, this separation can be done by imposing $\Phi \gtrsim N_1L_3$ (recall that $\Phi$ denotes the dyadic size of $\phi (\vec{\xi})$).
After differentiating by parts in $t$, we have
\[ \mathcal{L}^1[\omega ]=\mathcal{L}^1_R[\omega ]+\partial_t\mathcal{L}^1_0[\omega ] +\mathcal{L}^1_1[\omega ,\partial_t\omega ],\]
where 
\begin{align*}
\mathcal{L}^1_R[\omega]&:=\sum _{\begin{smallmatrix} N_1\sim N_3\gg N\gg N_2 \\ L_3\sim N_1\gg L_1,L_2\\ NN_3>M,\,NN_3\gtrsim N_1L_3\end{smallmatrix}} \int _{\begin{smallmatrix}\xi =\xi_{123} \\ \xi_1=\eta_{123}\end{smallmatrix}}\frac{e^{it\phi (\vec{\xi})+it\phi (\vec{\eta})}}{\phi (\vec{\xi})}\xi_3\eta_3\psi_{N}(\xi)\psi_{N_1}(\xi_1) \\[-20pt]
&\hspace{140pt} \times \omega_{L_1}(\eta_1)\omega_{L_2}(\eta_2)\omega^*_{L_3}(\eta_3)\omega_{N_2}(\xi_2)\omega^*_{N_3}(\xi_3)\\[5pt]
\mathcal{L}^1_0[\omega ]&:=\sum _{\begin{smallmatrix} N_1\sim N_3\gg N\gg N_2 \\ L_3\sim N_1\gg L_1,L_2\\ M<NN_3\ll N_1L_3\end{smallmatrix}} \int _{\begin{smallmatrix}\xi =\xi_{123} \\ \xi_1=\eta_{123}\end{smallmatrix}}\frac{e^{it\phi (\vec{\xi})+it\phi (\vec{\eta})}}{\phi (\vec{\xi})\big[ \phi (\vec{\xi})+\phi (\vec{\eta})\big]}\xi_3\eta_3\psi_{N}(\xi)\psi_{N_1}(\xi_1) \\[-20pt]
&\hspace{140pt} \times \omega_{L_1}(\eta_1)\omega_{L_2}(\eta_2)\omega^*_{L_3}(\eta_3)\omega_{N_2}(\xi_2)\omega^*_{N_3}(\xi_3)\\[5pt]
\mathcal{L}^1_1[\omega ,\zeta ]&:=\sum _{\begin{smallmatrix} N_1\sim N_3\gg N\gg N_2 \\ L_3\sim N_1\gg L_1,L_2\\ M<NN_3\ll N_1L_3\end{smallmatrix}} \int _{\begin{smallmatrix}\xi =\xi_{123} \\ \xi_1=\eta_{123}\end{smallmatrix}}\frac{e^{it\phi (\vec{\xi})+it\phi (\vec{\eta})}}{\phi (\vec{\xi})\big[ \phi (\vec{\xi})+\phi (\vec{\eta})\big]}\xi_3\eta_3\psi_{N}(\xi)\psi_{N_1}(\xi_1) \\[-15pt]
&\hspace{130pt} \times \Big\{ \zeta _{L_1}(\eta_1)\omega_{L_2}(\eta_2)\omega^*_{L_3}(\eta_3)\omega_{N_2}(\xi_2)\omega^*_{N_3}(\xi_3)+\cdots \Big\} .
\end{align*}
In a similar manner, we apply NFR to the other terms $\mathcal{L}^j[\omega ]$, $j=2,3,4$.
For $j=2,4$, we have $|\phi (\vec{\eta})|\sim N_3L_3$, and thus $\Phi \gtrsim$/$\ll N_1L_3$ in the above definition is replaced by $\Phi \gtrsim$/$\ll N_3L_3$.%
\footnote{%
Since the condition $L_3\sim N_1\sim N_3\gg N$ for $\mathcal{L}^1$ implies $NN_3\ll N_1L_3$, the resonant part for $\mathcal{L}^1$ actually does not appear if the implicit constants for ``$\sim$'' and ``$\ll$'' are chosen appropriately.
We do not need such a precise consideration, however.%
}%

To finish the NFR procedure, it suffices to prove the following estimates:
\begin{lem}\label{lem:L1}
For $j=1,2,3,4$, we have
\begin{align}
\| \mathcal{L}^j_R[\omega ]\|_{\hat{H}^{\frac12}}&\lesssim \| w\|_{H^{\frac12}}^4\| w\|_{B^{0+}_{\infty,1}}, \label{est:LjR} \\
\| \mathcal{L}^j_0[\omega ]\|_{\hat{H}^{\frac12}}&\lesssim M^{-1+}\| w\|_{H^{\frac12}}^5,\qquad \| \mathcal{L}^j_1[\omega ,\zeta ]\|_{\hat{H}^{\frac12}}\lesssim M^{-1+}\| w\|_{H^{\frac12}}^4\| z\|_{H^{\frac12}}, \label{est:Lj0} \\
\| \mathcal{L}^j_1[\omega ,\zeta ]\|_{\hat{H}^{\frac12}}&\lesssim M^{-\frac12+}\| w\|_{H^{\frac12}}^3\| w\|_{H^{\frac12}\cap B^0_{\infty,1}}\| z\|_{H^{-\frac12}+B^0_{1,\infty}}, \label{est:Lj1}
\end{align}
where $\omega=\mathcal{F}S(-t)w$ and $\zeta=\mathcal{F}S(-t)z$.
\end{lem}

\begin{proof}
We first consider $\mathcal{L}^1$ (the argument for $\mathcal{L}^2$ is almost the same and thus omitted).
Observe that
\begin{align*}
\mathcal{L}^1_R[\omega]&=\sum _{\begin{smallmatrix} N_1\sim N_3\gg N\gg N_2 \\ L_3\sim N_1\gg L_1,L_2\\ NN_3>M,\,NN_3\gtrsim N_1L_3\end{smallmatrix}} e^{it\xi^2}\psi_{N}(\xi)\int _{\xi =\xi_{123}}\frac{\xi_3}{\phi (\vec{\xi})} \mathcal{F}P_{N_1}\Big[ w_{L_1}w_{L_2}\partial_x\bar{w}_{L_3}\Big] (\xi_1)\hat{w}_{N_2}(\xi_2)\hat{\bar{w}}_{N_3}(\xi_3).
\end{align*}
We apply Lemma~\ref{lem:CM2} with
\[ b(\xi,\xi_1,\xi_2,\xi_3)=\frac{\xi_3}{\xi_{23}(\xi-\xi_2)},\]
which satisfies the assumption with $B=\frac{N_3}{NN_3}\lesssim \frac{N_3}{N_1L_3}\sim L_3^{-1}$.
As a result, we obtain the bound
\begin{align*}
&\sum _{\begin{smallmatrix} N_1\sim N_3\gg N\gg N_2 \\ L_3\sim N_1\gg L_1,L_2\end{smallmatrix}}\frac{N^{\frac12}}{L_3}\Big\| w_{L_1}w_{L_2}\partial_x\bar{w}_{L_3}\Big\|_{L^\infty}\| w_{N_2}\|_{L^\infty}\| w_{N_3}\|_{L^2} \\
&\lesssim \sum _{\begin{smallmatrix} N_1\sim N_3\gg N\gg N_2 \\ L_3\sim N_1\gg L_1,L_2\end{smallmatrix}}N^{\frac12}\| w_{L_1}\|_{L^\infty}\| w_{L_2}\|_{L^\infty}\| w_{L_3}\|_{L^\infty}\| w_{N_2}\|_{L^\infty}\| w_{N_3}\|_{L^2} \\
&\lesssim \sum_{L_3}\| w_{L_3}\|_{L^\infty}\sum _{N_1,N_3\sim L_3}\| w_{N_3}\|_{H^{\frac12}}\sum _{L_1,L_2,N_2\ll L_3}\| w_{L_1}\|_{H^{\frac12}}\| w_{L_2}\|_{H^{\frac12}}\| w_{N_2}\|_{H^{\frac12}} \\
&\lesssim \| w\|_{B^{0+}_{\infty,1}}\| w\|_{H^{\frac12}}^4,
\end{align*}
which implies \eqref{est:LjR}.
For $\mathcal{L}^1_0$, we can take the absolute value:
\begin{align*}
|\mathcal{L}^1_0[\omega ]|&\lesssim \sum _{\begin{smallmatrix} N_1\sim N_3\gg N\gg N_2 \\ L_3\sim N_1\gg L_1,L_2\end{smallmatrix}} \int _{\begin{smallmatrix}\xi =\xi_{123} \\ \xi_1=\eta_{123}\end{smallmatrix}}\frac{N_3L_3}{|\phi (\vec{\xi})| \big| \phi (\vec{\xi})+\phi (\vec{\eta})\big|}\mathbf{1}_{|\phi (\vec{\xi})|\sim NN_3\gtrsim M}\mathbf{1}_{|\phi (\vec{\xi})+\phi (\vec{\eta})|\sim N_1L_3}\\
&\qquad\qquad \times \psi_{N}(\xi)\psi_{N_1}(\xi_1)|\omega_{L_1}(\eta_1)||\omega_{L_2}(\eta_2)||\omega^*_{L_3}(\eta_3)||\omega_{N_2}(\xi_2)||\omega^*_{N_3}(\xi_3)|.
\end{align*}
We use Young and H\"older to have
\begin{align*}
\| \mathcal{L}^1_0[\omega ]\|_{\hat{H}^{\frac12}}&\lesssim \sum _{\begin{smallmatrix} N_1\sim N_3\gg N\gg N_2 \\ L_3\sim N_1\gg L_1,L_2\end{smallmatrix}} \frac{N^{\frac12}N_3L_3}{(NN_3)^{0+}M^{1-}N_1L_3}\Big\| |\omega_{L_1}|*\cdots *|\omega^*_{N_3}|\Big\| _{L^2}\\
&\lesssim M^{-1+}\sum _{\begin{smallmatrix} N_1\sim N_3\gg N\gg N_2 \\ L_3\sim N_1\gg L_1,L_2\end{smallmatrix}}N_3^{0-}\| \omega_{L_1}\|_{L^1}\| \omega_{L_2}\| _{L^1}\| \omega_{L_3}\|_{L^1}\| \omega_{N_2}\|_{L^1}N_3^{\frac12}\| \omega_{N_3}\|_{L^2} \\
&\lesssim M^{-1+}\sum _{\begin{smallmatrix} N_1\sim N_3\gg N\gg N_2 \\ L_3\sim N_1\gg L_1,L_2\end{smallmatrix}}N_3^{0-}\| \omega_{L_1}\|_{\hat{H}^{\frac12}}\| \omega_{L_2}\| _{\hat{H}^{\frac12}}\| \omega_{L_3}\|_{\hat{H}^{\frac12}}\| \omega_{N_2}\|_{\hat{H}^{\frac12}}\| \omega_{N_3}\|_{\hat{H}^{\frac12}} \\
&\lesssim M^{-1+}\| \omega \|_{\hat{H}^{\frac12}}^5.
\end{align*}
This implies \eqref{est:Lj0}.
Finally, by a similar argument we can estimate $\mathcal{L}^1_1[\omega ,\zeta ]$ in terms of $\| z\|_{H^{-\frac12-}}$:
\begin{align*}
&\| \mathcal{L}^1_1[\omega ,\zeta ]\|_{\hat{H}^{\frac12}}\\
&\lesssim M^{-\frac12+}\sum _{\begin{smallmatrix} N_1\sim N_3\gg N\gg N_2 \\ L_3\sim N_1\gg L_1,L_2\end{smallmatrix}} \frac{NN_3L_3}{(NN_3)^{\frac12+}N_1L_3}\Big\{ \Big\| |\zeta_{L_1}|*|\omega_{L_2}|*\cdots *|\omega^*_{N_3}|\Big\| _{L^\infty} +\cdots \Big\} \\
&\lesssim M^{-\frac12+}\sum _{\begin{smallmatrix} N_1\sim N_3\gg N\gg N_2 \\ L_3\sim N_1\gg L_1,L_2\end{smallmatrix}} N_{\max(2)}^{0-}\Big\{ \| \zeta_{L_1}\|_{L^1}\| \omega_{L_2}\|_{L^1}\| \omega_{L_3}\| _{L^2}\| \omega_{N_2}\|_{L^1}\| \omega_{N_3}\|_{L^2} +\cdots \Big\} \\
&\lesssim M^{-\frac12+}\sum _{\begin{smallmatrix} N_1\sim N_3\gg N\gg N_2 \\ L_3\sim N_1\gg L_1,L_2\end{smallmatrix}} N_{\max(2)}^{0-}\cdot N_{\max(2)}^{-1-}\Big\{ \| \zeta_{L_1}\|_{\hat{H}^{\frac12}}\cdots \| \omega_{N_3}\|_{\hat{H}^{\frac12}} +\cdots \Big\} \\
&\lesssim M^{-\frac12+}\| z\|_{H^{-\frac12-}}\| w\|_{H^{\frac12}}^4,
\end{align*}
where $N_{\max(2)}:=\max \{ N,N_j,L_j:j=1,2,3\} \sim \max \{ N_2,N_3,L_1,L_2,L_3\}$.
Since $H^{-\frac12}+B^0_{1,\infty}\hookrightarrow H^{-\frac12-}$, this estimate implies \eqref{est:Lj1}.

We next consider $\mathcal{L}^3$ and $\mathcal{L}^4$.
Recall the frequency conditions for these terms:
\[ \mathcal{L}^3:\left\{ \begin{aligned}
&N_1\sim N\gg N_2,N_3,\quad N_2\not\sim N_3, \\
&L_1\sim L_3\gg N_1\gg L_2,
\end{aligned}\right. 
\qquad \mathcal{L}^4:\left\{ \begin{aligned}
&N_3\sim N\gg N_1,N_2, \\
&L_1\sim L_3\gg N_3\gg L_2.
\end{aligned}\right. \]
The estimate of $\mathcal{L}^j_R$ is done almost in the same way as for $\mathcal{L}^1_R$.
Concerning $\mathcal{L}^j_0$, we have the multiplier bound 
\[ M^{-1+}\frac{N^{\frac12}N_3L_3}{(NN_3)^{0+}NL_3}\lesssim M^{-1+}N_3^{0-}L_1^{\frac12-},\]
and then the previous argument works, since we have a factor $L_1^{0-}\sim N_{\max(2)}^{0-}$.
On the other hand, the estimate of $\mathcal{L}^j_1[\omega ,\zeta]$ becomes harder, and we need Coifman-Meyer type argument.
After using H\"older to obtain the $L^\infty_\xi$ norm, we can use the multiplier bound
\[ M^{-\frac12+}\frac{NN_3L_3}{(NN_3)^{\frac12+}NL_3}\lesssim M^{-\frac12+}N^{0-}.\]
Since we have $L_1\sim L_3\gg N\gtrsim $ (others), this bound together with the previous argument is enough to obtain $\| z\|_{H^{-\frac12-}}\| w\|_{H^{\frac12}}^4$, unless $\omega_{L^1}$ or $\omega_{L^3}$ is replaced by $\zeta$.
Therefore, let us concentrate on the following term (one of the terms in $\mathcal{L}^4_1[\omega ,\zeta]$):%
\footnote{%
The other remaining terms in $\mathcal{L}^3_1[\omega,\zeta]$ and $\mathcal{L}^4_1[\omega ,\zeta]$ can be treated similarly.%
}%
\begin{align*}
&\sum _{\begin{smallmatrix} N_3\sim N\gg N_1,N_2 \\ L_1\sim L_3\gg N_3\gg L_2\\ M<NN_3\ll N_3L_3\end{smallmatrix}} \int _{\begin{smallmatrix}\xi =\xi_{123} \\ \xi_3=\eta_{123}\end{smallmatrix}}\frac{e^{it\phi (\vec{\xi})-it\phi (\vec{\eta})}}{\phi (\vec{\xi})\big[ \phi (\vec{\xi})-\phi (\vec{\eta})\big]}\xi_3\eta_3\psi_{N}(\xi)\psi_{N_3}(\xi_3) \\[-10pt]
&\hspace{120pt} \times \omega_{N_1}(\xi_1)\omega_{N_2}(\xi_2)\zeta _{L_1}^*(\eta_1)\omega_{L_2}^*(\eta_2)\omega_{L_3}(\eta_3).
\end{align*}
The key idea is the following decomposition of the multiplier:
\[ \frac{1}{\phi (\vec{\xi})\big[ \phi (\vec{\xi})-\phi (\vec{\eta})\big]}=\frac{1}{\phi (\vec{\eta})\big[ \phi (\vec{\xi})-\phi (\vec{\eta})\big]}-\frac{1}{\phi (\vec{\xi})\phi (\vec{\eta})} .\]
For the contribution from the first term, we have a better multiplier bound
\[ \Big| \frac{\xi_3\eta_3}{\phi (\vec{\eta})\big[ \phi (\vec{\xi})-\phi (\vec{\eta})\big]}\Big| \lesssim M^{-\frac12+}\frac{N_3L_3}{(N_3L_3)^{\frac12+}N_3L_3}\lesssim M^{-\frac12+}N^{-1}N_{\max(2)}^{0-}.\]
Then, the previous argument for $\mathcal{L}^1_1[\omega ,\zeta]$ is sufficient to obtain the bound $\| z\|_{H^{-\frac12-}}\| w\|_{H^{\frac12}}^4$.
For the contribution from the second term, the bound with $\| z\|_{H^{-\frac12}}$ can be obtained by the same argument:
\begin{align*}
&M^{-\frac12+}\sum _{\begin{smallmatrix} N_3\sim N\gg N_1,N_2 \\ L_1\sim L_3\gg N_3\gg L_2\end{smallmatrix}}\frac{NN_3L_3}{(NN_3)^{\frac12+}N_3L_3}\| \omega _{N_1}\|_{L^1}\| \omega_{N_2}\|_{L^1}\| \zeta_{L_1}\|_{L^2}\| \omega_{L_2}\|_{L^1}\| \omega_{L^3}\|_{L^2}\\
&\lesssim M^{-\frac12+}\sum _{L_1\sim L_3}\| \zeta_{L_1}\|_{\hat{H}^{-\frac12}}\| \omega_{L^3}\|_{\hat{H}^{\frac12}} \sum_{N\gtrsim \text{(others)}}N^{0-}\| \omega _{N_1}\|_{\hat{H}^{\frac12}}\| \omega_{N_2}\|_{\hat{H}^{\frac12}}\| \omega_{L_2}\|_{\hat{H}^{\frac12}}\\
&\lesssim M^{-\frac12+}\| z\|_{H^{-\frac12}}\| w\|_{H^{\frac12}}^4.
\end{align*}
To obtain the bound with $\| z\|_{B^0_{1,\infty}}$, we rewrite this term as 
\begin{align*}
\sum _{\begin{smallmatrix} N_3\sim N\gg N_1,N_2 \\ L_1\sim L_3\gg N_3\gg L_2\\ M<NN_3\ll N_3L_3\end{smallmatrix}} e^{it\xi^2}&\psi_{N}(\xi)\int _{\xi =\xi_{123}}\frac{\xi_3}{\phi (\vec{\xi})}\hat{w}_{N_1}(\xi_1)\hat{w}_{N_2}(\xi_2) \\[-10pt]
\times \Big[ &\psi_{N_3}(\xi_3)\int _{\xi_3=\eta_{123}}\frac{\eta_3}{\phi (\vec{\eta})}\hat{\bar{z}}_{L_1}(\eta_1)\hat{\bar{w}}_{L_2}(\eta_2)\hat{w}_{L_3}(\eta_3)\Big] .
\end{align*}
We use H\"older and first apply Lemma~\ref{lem:CM1} with 
\[ b(\xi_1,\xi_2,\xi_3)=\frac{\xi_3}{\xi_{13}\xi_{23}},\qquad B=\frac{N_3}{NN_3}\lesssim M^{-\frac12+}\frac{N_3}{(NN_3)^{\frac12+}}\sim M^{-\frac12+}N^{0-},\]
to obtain the bound
\begin{align*}
M^{-\frac12+}\sum _{\begin{smallmatrix} N_3\sim N\gg N_1,N_2 \\ L_1\sim L_3\gg N_3\gg L_2\end{smallmatrix}}&N^{1-}\| w_{N_1}\|_{L^\infty}\| w_{N_2}\|_{L^\infty} \\[-10pt]
&\times \Big\| \mathcal{F}^{-1}\Big[ \psi_{N_3}(\xi_3)\int _{\xi_3=\eta_{123}}\frac{\eta_3}{\phi (\vec{\eta})}\hat{\bar{z}}_{L_1}(\eta_1)\hat{\bar{w}}_{L_2}(\eta_2)\hat{w}_{L_3}(\eta_3)\Big] \Big\|_{L^1}.
\end{align*}
We then apply Lemma~\ref{lem:CM2} with
\[ b(\xi _3,\eta_1,\eta_2,\eta_3)=\frac{\eta_3}{\eta_{23}(\xi_3-\eta_2)},\qquad B=\frac{L_3}{N_3L_3}\sim N^{-1} ,\]
which gives the bound
\begin{align*}
&M^{-\frac12+}\sum _{\begin{smallmatrix} N_3\sim N\gg N_1,N_2 \\ L_1\sim L_3\gg N_3\gg L_2\end{smallmatrix}}\| w_{N_1}\|_{L^\infty}\| w_{N_2}\|_{L^\infty}N^{0-}\| z_{L_1}\|_{L^1}\| w_{L_2}\|_{L^\infty}\| w_{L_3}\|_{L^\infty} \\
&\lesssim M^{-\frac12+}\sum _{L_1\sim L_3}\| z_{L_1}\|_{L^1}\| w_{L_3}\|_{L^\infty} \sum_{N\gtrsim \text{(others)}}N^{0-}\| w_{N_1}\|_{H^{\frac12}}\| w_{N_2}\|_{H^{\frac12}}\| w_{L_2}\|_{H^{\frac12}}\\
&\lesssim M^{-\frac12+}\| z\|_{B^0_{1,\infty}}\| w\|_{B^0_{\infty,1}}\| w\|_{H^{\frac12}}^3.
\end{align*}
Hence, for $j=3,4$ we have 
\[ \| \mathcal{L}^j_1[\omega ,\zeta]\|_{\hat{H}^{\frac12}}\lesssim M^{-\frac12+}\| z\|_{H^{-\frac12}+B^0_{1,\infty}}\| w\|_{H^{\frac12}\cap B^0_{\infty,1}}\| w\|_{H^{\frac12}}^3,\]
which implies \eqref{est:Lj1}.
\end{proof}


\subsection{Estimate of the remainder after NFR: Case \texttt{B}}

Let us consider the estimate of $\mathcal{N}^{\texttt{B}}_1[\omega ,\mathcal{N}[\omega]]$, which is a little more delicate.
Recall that
\begin{align*}
&\mathcal{N}^{\texttt{B}}_1[\omega ,\zeta ]=\sum_{j=1}^3\mathcal{N}^{\texttt{B1}}_{1,j}[\omega ,\zeta] +\sum_{j=1}^3\mathcal{N}^{\texttt{B2}}_{1,j}[\omega ,\zeta ] \\
&:=\sum _{\begin{smallmatrix}N_1\sim N_3\gg N\sim N_2\\ N_1K>M\end{smallmatrix}}\psi_N(\xi)\int _{\begin{smallmatrix}\xi=\xi_{123}\\ K\leq |\xi_{13}|<2K\end{smallmatrix}}\frac{e^{it\phi(\vec{\xi})}}{\phi (\vec{\xi})}\xi_3\Big\{ \zeta_{N_1}(\xi_1)\omega_{N_2}(\xi_2)\omega^*_{N_3}(\xi_3)+\cdots \Big\}  \\
&\;\quad +\sum _{\begin{smallmatrix}N_1\sim N\gg N_2\sim N_3 \\ N_1K>M\end{smallmatrix}}\psi_N(\xi)\int _{\begin{smallmatrix}\xi=\xi_{123}\\ K\leq |\xi_{23}|<2K\end{smallmatrix}}\frac{e^{it\phi(\vec{\xi})}}{\phi (\vec{\xi})}\xi_3\Big\{ \zeta_{N_1}(\xi_1)\omega_{N_2}(\xi_2)\omega^*_{N_3}(\xi_3)+\cdots \Big\} .
\end{align*}

\begin{lem}\label{lem:N_1B}
We have
\[ \| \mathcal{N}^{\mathtt{B}}_1[\omega ,\zeta ]\|_{\hat{H}^{\frac12}}\lesssim \| w\|_{H^{\frac12}}^2\| z\|_{H^{-\frac12+}}\]
and
\[ \| \Big( \mathcal{N}^{\texttt{B1}}_{1,1}+\mathcal{N}^{\texttt{B1}}_{1,3}+\mathcal{N}^{\texttt{B2}}_{1,2}+\mathcal{N}^{\texttt{B2}}_{1,3}\Big) [\omega ,\zeta ]\|_{\hat{H}^{\frac12}} \lesssim \| w\|_{H^{\frac12}}\| w\|_{B^{0+}_{\infty,1}}\| z\|_{B^0_{1,\infty}}. \]
\end{lem}

\begin{proof}
For the case \texttt{B1}, we apply Young/H\"older as follows:
\begin{align*}
\| \mathcal{N}^{\texttt{B1}}_1[\omega ,\zeta]\|_{\hat{H}^{\frac12}}&\lesssim \sum_{\begin{smallmatrix} N_1\sim N_3\gg N\sim N_2 \\ K\lesssim N_1\end{smallmatrix}}\frac{N^{\frac12}N_3}{N_1K}\Big\{ \Big\| \Big( \mathbf{1}_{|\cdot |\sim K}\big( |\zeta _{N_1}|*|\omega^*_{N_3}|\big) \Big) *|\omega_{N_2}|\Big\|_{L^2} +\cdots \Big\} \\
&\lesssim \sum_{\begin{smallmatrix} N_1\sim N_3\gg N\sim N_2 \\ K\lesssim N_1\end{smallmatrix}}N^{\frac12}\Big\{ \Big\| |\zeta _{N_1}|*|\omega^*_{N_3}|\Big\|_{L^\infty}\|\omega_{N_2}\|_{L^2} +\cdots \Big\} \\
&\lesssim \sum _{\begin{smallmatrix} N_1\sim N_3\gg N\sim N_2 \\ K\lesssim N_1\end{smallmatrix}}N^{\frac12}\Big\{ \| z_{N_1}\|_{L^2}\| w_{N_2}\|_{L^2}\| w_{N_3}\|_{L^2}+\cdots \Big\} .
\end{align*}
In the same manner, we have
\[ \| \mathcal{N}^{\texttt{B2}}_1[\omega ,\zeta]\|_{\hat{H}^{\frac12}}\lesssim \sum _{\begin{smallmatrix} N_1\sim N\gg N_2\sim N_3 \\ K\lesssim N_2\end{smallmatrix}}\frac{N^{\frac12}N_3}{N_1}\Big\{ \| z_{N_1}\|_{L^2}\| w_{N_2}\|_{L^2}\| w_{N_3}\|_{L^2}+\cdots \Big\} .\]
In both cases, we have the bound
\[ \sum _{N,N_j,K}N_{\max}^{-1}\Big\{ \| z_{N_1}\|_{H^{\frac12}}\| w_{N_2}\|_{H^{\frac12}}\| w_{N_3}\|_{H^{\frac12}}+\cdots \Big\} \lesssim \| z\|_{H^{-\frac12+}}\| w\|_{H^{\frac12}}^2,\]
which implies the first estimate.%
\footnote{%
This argument is not sufficient to get $\| z\|_{H^{-\frac12}}$ instead of $\| z\|_{H^{-\frac12+}}$, due to the logarithmic divergence in $K$.%
}%

To show the second estimate, we use Lemma~\ref{lem:CM3} with $p=2$ and $q=1$.
For $\mathcal{N}^{\texttt{B1}}_{1,1}$ (and similarly for $\mathcal{N}^{\texttt{B1}}_{1,3}$), the bound
\[ \sum _{\begin{smallmatrix} N_1\sim N_3\gg N\sim N_2 \\ K\lesssim N_1\end{smallmatrix}}N^{\frac12}\| z_{N_1}\|_{L^1}\| w_{N_2}\|_{L^2}\| w_{N_3}\|_{L^\infty} \lesssim \| z\|_{B^0_{1,\infty}}\| w\|_{H^{\frac12}}\| w\|_{B^{0+}_{\infty,1}} \]
is available.
Similarly, we have the bound
\[ \sum _{\begin{smallmatrix} N_1\sim N\gg N_2\sim N_3 \\ K\lesssim N_2\end{smallmatrix}}\frac{N^{\frac12}N_3}{N_1}\| w_{N_1}\|_{L^2}\| z_{N_2}\|_{L^1}\| w_{N_3}\|_{L^\infty} \lesssim \| w\|_{H^{\frac12}}\| z\|_{B^0_{1,\infty}}\| w\|_{B^{0+}_{\infty,1}}\]
for $\mathcal{N}^{\texttt{B2}}_{1,2}$ (and similarly for $\mathcal{N}^{\texttt{B2}}_{1,3}$).
Then, the second estimate follows.
\end{proof}

In view of the above lemma and Proposition~\ref{prop:N-weak}, the $\hat{H}^{\frac12}$ norms of 
\[ \Big( \mathcal{N}^{\texttt{B1}}_{1,1} + \mathcal{N}^{\texttt{B1}}_{1,3} + \mathcal{N}^{\texttt{B2}}_{1,2} + \mathcal{N}^{\texttt{B2}}_{1,3}\Big) [\omega ,\mathcal{N}^{\neg \texttt{C2}}[\omega]],\quad \Big( \mathcal{N}^{\texttt{B1}}_{1,2} + \mathcal{N}^{\texttt{B2}}_{1,1}\Big) [\omega ,\mathcal{N}^{\neg \texttt{C}}[\omega]] \]
are bounded by $\| w\|_{H^{\frac12}}^3\| w\|_{H^{\frac12}\cap B^{0+}_{\infty,1}}^2$.
Therefore, we apply the second NFR to the remaining terms:
\begin{equation}\label{defn:Lj2}
\begin{gathered}
\begin{alignedat}{2}
\mathcal{L}^5[\omega]&:=\Big( \mathcal{N}^{\texttt{B2}}_{1,2} + \mathcal{N}^{\texttt{B2}}_{1,3}\Big) [\omega ,\mathcal{N}^{\texttt{C2}}[\omega]],&\qquad \mathcal{L}^6[\omega]&:=\mathcal{N}^{\texttt{B2}}_{1,1}[\omega ,\mathcal{N}^{\texttt{C2}}[\omega]],\\
\mathcal{L}^7[\omega ]&:=\Big( \mathcal{N}^{\texttt{B1}}_{1,1} + \mathcal{N}^{\texttt{B1}}_{1,3}\Big) [\omega ,\mathcal{N}^{\texttt{C2}}[\omega]],&\qquad \mathcal{L}^8[\omega]&:=\mathcal{N}^{\texttt{B1}}_{1,2}[\omega ,\mathcal{N}^{\texttt{C2}}[\omega]],
\end{alignedat}\\
\mathcal{L}^9[\omega ]:=\mathcal{N}^{\texttt{B2}}_{1,1}[\omega ,\mathcal{N}^{\texttt{C1}}[\omega]],\qquad 
\mathcal{L}^{10}[\omega ]:= \mathcal{N}^{\texttt{B1}}_{1,2}[\omega ,\mathcal{N}^{\texttt{C1}}[\omega]].
\end{gathered}
\end{equation}
Recall the relation between $N$, $N_j$'s and $L_j$'s in each case:
\begin{alignat*}{2}
\mathcal{L}^5:&~
\begin{tabular}{|r@{\;}c@{\;}l|}
\hline 
\ \ $N_1\sim N\gg$ & $N_2$ & \\
& $\wr$ & \\
$L_3\sim$ & $N_3$ & $\gg L_1,L_2$ \\ \hline
\multicolumn{3}{|c|}{\:\ $|\xi_{23}|\sim K$, $|\phi (\vec{\eta})|\sim N_3L_3$\ \ } \\ \hline
\end{tabular}
&\qquad \mathcal{L}^6:&~
\begin{tabular}{|r@{\;}c@{\;}l|}
\hline 
& $N$ & $\gg N_2\sim N_3$\\
& $\wr$ & \\
\quad \ $L_3\sim$ & $N_1$ & $\gg L_1,L_2$ \\ \hline
\multicolumn{3}{|c|}{\:\ $|\xi_{23}|\sim K$, $|\phi (\vec{\eta})|\sim N_1L_3$\ \ } \\ \hline
\end{tabular}
\\
\mathcal{L}^7:&~
\begin{tabular}{|r@{\;}c@{\;}l|}
\hline 
& $N_1$ & $\gg N\sim N_2$\\
& $\wr$ & \\
\quad \ $L_3\sim$ & $N_3$ & $\gg L_1,L_2$ \\ \hline
\multicolumn{3}{|c|}{\:\ $|\xi_{13}|\sim K$, $|\phi (\vec{\eta})|\sim N_3L_3$\ \ } \\ \hline
\end{tabular}
&\qquad \mathcal{L}^8:&~
\begin{tabular}{|r@{\;}c@{\;}l|}
\hline 
\ \,$N_1\sim N_3\gg$ & $N$ & \\
& $\wr$ & \\
$L_3\sim$ & $N_2$ & $\gg L_1,L_2$ \\ \hline
\multicolumn{3}{|c|}{\:\ $|\xi_{13}|\sim K$, $|\phi (\vec{\eta})|\sim N_2L_3$\ \ } \\ \hline
\end{tabular}
\\
\mathcal{L}^9:&~
\begin{tabular}{|r@{\;}c@{\;}l|}
\hline 
& $N$ & $\gg N_2\sim N_3$\\
& $\wr$ & \\
$L_1\sim L_3\gg$ & $N_1$ & $\gg L_2$ \\ \hline
\multicolumn{3}{|c|}{$|\xi_{23}|\sim K$, $|\phi (\vec{\eta})|\sim N_1L_3$} \\ \hline
\end{tabular}
&\qquad \mathcal{L}^{10}:&~
\begin{tabular}{|r@{\;}c@{\;}l|}
\hline 
\quad \ $N_1\sim N_3\gg$ & $N$ & \\
& $\wr$ & \\
$L_1\sim L_3\gg$ & $N_2$ & $\gg L_2$ \\ \hline
\multicolumn{3}{|c|}{\:\ $|\xi_{13}|\sim K$, $|\phi (\vec{\eta})|\sim N_2L_3$\ \ } \\ \hline
\end{tabular}
\end{alignat*}
Note that $|\phi(\vec{\xi})|\sim N_1K\sim N_{\max}K$ in all the cases.
Also, the treatments of two terms in $\mathcal{L}^5$ or $\mathcal{L}^7$ are analogous, and we will focus on the first term.

As in the previous case ($\mathcal{L}^j$ with $j=1,\dots ,4$), the second NFR is performed as
\[ \mathcal{L}^j[\omega ]=\mathcal{L}^j_R[\omega ]+\partial_t\mathcal{L}^j_0[\omega ] +\mathcal{L}^j_1[\omega ,\partial_t\omega ]\qquad (j=5,\dots ,10).\]
Here, the separation of the resonant part $\mathcal{L}^j_R$ is again based on the condition $|\phi(\vec{\xi})|\gtrsim |\phi(\vec{\eta})|$; this is actually done by imposing the relation between dyadic frequencies.
For instance, for $\mathcal{L}^{10}$ we define the resonant part by $N_1K\gtrsim N_2L_3$.
The following is a precise definition of $\mathcal{L}^{10}_R,\mathcal{L}^{10}_0,\mathcal{L}^{10}_1$, and the corresponding terms for $\mathcal{L}^j$, $j=5,\dots ,9$ are defined similarly.
\begin{align*}
\mathcal{L}^{10}_R[\omega ]&:=\sum _{\begin{smallmatrix}N_1\sim N_3\gg N\sim N_2\\ L_1\sim L_3\gg N_2\gg L_2\\ N_1K>M,\,N_1K\gtrsim N_2L_3\end{smallmatrix}}\int _{\begin{smallmatrix}\xi=\xi_{123},\,\xi_2=\eta_{123}\\ K\leq |\xi_{13}|<2K\end{smallmatrix}}\frac{e^{it\phi(\vec{\xi})+it\phi(\vec{\eta})}}{\phi (\vec{\xi})}\psi_N(\xi)\psi_{N_2}(\xi_2)\xi_3\eta_3\\[-20pt]
&\hspace{140pt} \times \omega_{N_1}(\xi_1)\omega^*_{N_3}(\xi_3)\omega_{L_1}(\eta_1)\omega_{L_2}(\eta_2)\omega^*_{L_3}(\eta_3) ,\\[5pt]
\mathcal{L}^{10}_0[\omega ]&:=\sum _{\begin{smallmatrix}N_1\sim N_3\gg N\sim N_2\\ L_1\sim L_3\gg N_2\gg L_2\\ M<N_1K\ll N_2L_3\end{smallmatrix}}\int _{\begin{smallmatrix}\xi=\xi_{123},\,\xi_2=\eta_{123}\\ K\leq |\xi_{13}|<2K\end{smallmatrix}}\frac{e^{it\phi(\vec{\xi})+it\phi(\vec{\eta})}}{\phi (\vec{\xi})\big[ \phi(\vec{\xi})+\phi(\vec{\eta})\big]}\psi_N(\xi)\psi_{N_2}(\xi_2)\xi_3\eta_3\\[-20pt]
&\hspace{140pt} \times \omega_{N_1}(\xi_1)\omega^*_{N_3}(\xi_3)\omega_{L_1}(\eta_1)\omega_{L_2}(\eta_2)\omega^*_{L_3}(\eta_3) ,\\[5pt]
\mathcal{L}^{10}_1[\omega ,\zeta ]&:=\sum _{\begin{smallmatrix}N_1\sim N_3\gg N\sim N_2\\ L_1\sim L_3\gg N_2\gg L_2\\ M<N_1K\ll N_2L_3\end{smallmatrix}}\int _{\begin{smallmatrix}\xi=\xi_{123},\,\xi_2=\eta_{123}\\ K\leq |\xi_{13}|<2K\end{smallmatrix}}\frac{e^{it\phi(\vec{\xi})+it\phi(\vec{\eta})}}{\phi (\vec{\xi})\big[ \phi(\vec{\xi})+\phi(\vec{\eta})\big]}\psi_N(\xi)\psi_{N_2}(\xi_2)\xi_3\eta_3\\[-15pt]
&\hspace{120pt} \times \Big\{ \zeta _{N_1}(\xi_1)\omega^*_{N_3}(\xi_3)\omega_{L_1}(\eta_1)\omega_{L_2}(\eta_2)\omega^*_{L_3}(\eta_3) +\cdots \Big\} .
\end{align*}

We need to prove:
\begin{lem}\label{lem:L2}
For $j=5,6,\dots ,10$, we have the same estimates \eqref{est:LjR}--\eqref{est:Lj1} as in Lemma~\ref{lem:L1}.
\end{lem}

\begin{proof}
{\bf Proof of \eqref{est:LjR}:}
The common strategy is that we take the absolute value and replace $|\phi(\vec{\xi})|^{-1}$ by the reciprocal of the dyadic size of $\phi(\vec{\eta})$, but keeping the product structure of $w_{L_1}w_{L_2}\bar{w}_{L_3}$ in the $x$ variable, and then undo the decomposition in $K$.
For instance, $\mathcal{L}^6_R$ can be rewritten as follows:
\begin{align*}
\mathcal{L}^6_R[\omega]&=\sum _{\begin{smallmatrix}N_1\sim N\gg N_2\sim N_3\\ L_3\sim N_1\gg L_1,L_2 \\ N_1K>M,\,N_1K\gtrsim N_1L_3\end{smallmatrix}}\psi_N(\xi)\int _{\begin{smallmatrix}\xi=\xi_{123}\\ K\leq |\xi_{23}|<2K\end{smallmatrix}}\frac{e^{it\phi(\vec{\xi})}}{\phi (\vec{\xi})}\xi_3\Omega (\xi_1)\omega_{N_2}(\xi_2)\omega^*_{N_3}(\xi_3),
\end{align*}
where
\begin{align*}
\Omega (\xi_1)=\Omega _{N_1,L_1,L_2,L_3}(\xi_1)&:=\psi_{N_1}(\xi_1)e^{it\xi_1^2}\int_{\xi_1=\eta_{123}}\eta_3\hat{w}_{L_1}(\eta_1)\hat{w}_{L_2}(\eta_2)\hat{\bar{w}}_{L_3}(\eta_3) \\
&\;=\mathcal{F}\Big[ S(-t)P_{N_1}\big[ w_{L_1}w_{L_2}\partial_x\bar{w}_{L_3}\big] \Big] (\xi_1).
\end{align*}
We then estimate $|\mathcal{L}^6_R[\omega ]|$ as follows:
\begin{align*}
|\mathcal{L}^6_R[\omega ]|&\lesssim \sum _{\begin{smallmatrix}N_1\sim N\gg N_2\sim N_3\\ L_3\sim N_1\gg L_1,L_2 \\ N_1K>M,\,N_1K\gtrsim N_1L_3\end{smallmatrix}}\frac{N_3}{N_1L_3}\psi_N(\xi)\int _{\begin{smallmatrix}\xi=\xi_{123}\\ K\leq |\xi_{23}|<2K\end{smallmatrix}}|\Omega (\xi_1)||\omega_{N_2}(\xi_2)||\omega^*_{N_3}(\xi_3)|\\
&\lesssim \sum _{\begin{smallmatrix}N_1\sim N\gg N_2\sim N_3\\ L_3\sim N_1\gg L_1,L_2\end{smallmatrix}}\frac{N_3}{N_1L_3}\psi_N(\xi) \Big( |\Omega |*|\omega_{N_2}|*|\omega^*_{N_3}|\Big) (\xi ).
\end{align*}
To estimate the $\hat{H}^{\frac12}$ norm, we first use Young to obtain the bound
\[ \sum _{\begin{smallmatrix}N_1\sim N\gg N_2\sim N_3\\ L_3\sim N_1\gg L_1,L_2\end{smallmatrix}}\frac{N^{\frac12}N_3}{N_1L_3}\| \Omega \|_{L^2}\| \omega_{N_2}\|_{L^1}\| \omega_{N_3}\|_{L^1}, \]
and then insert the H\"older estimate on the physical side
\[ \| \Omega \|_{L^2}=\| P_{N_1}[ w_{L_1}w_{L_2}\partial_x\bar{w}_{L_3}]\|_{L^2}\lesssim L_3\| w_{L_1}\|_{L^\infty}\| w_{L_2}\|_{L^\infty}\| w_{L_3}\|_{L^2} \]
to get the bound
\begin{align*}
&\sum _{\begin{smallmatrix}N_1\sim N\gg N_2\sim N_3\\ L_3\sim N_1\gg L_1,L_2\end{smallmatrix}}\frac{N^{\frac12}N_3L_3}{N_1L_3}\| w_{L_1}\|_{L^\infty}\| w_{L_2}\|_{L^\infty}\| w_{L_3}\|_{L^2}\| \omega_{N_2}\|_{L^1}\| \omega_{N_3}\|_{L^1}\\
&\lesssim \sum _{\begin{smallmatrix}N_1\sim N\gg N_2\sim N_3\\ L_3\sim N_1\gg L_1,L_2\end{smallmatrix}}\frac{N^{\frac12}N_3L_3^{\frac12}}{N_1L_3}\| w_{L_1}\|_{L^\infty}\| w_{L_2}\|_{L^\infty}\| w_{L_3}\|_{H^{\frac12}}\| w_{N_2}\|_{H^{\frac12}}\| w_{N_3}\|_{H^{\frac12}}\\
&\lesssim \sum _{L_1,L_2}\| w_{L_1\vee L_2}\|_{L^\infty}\| w_{L_1\wedge L_2}\|_{H^{\frac12}}\sum _{N_2\sim N_3}\| w_{N_2}\|_{H^{\frac12}}\| w_{N_3}\|_{H^{\frac12}}\sum _{N_1\sim N\sim L_3\gg N_3}\frac{N_3}{L_3}\| w_{L_3}\|_{H^{\frac12}}\\
&\lesssim \| w\|_{B^{0+}_{\infty,1}}\| w\|_{H^{\frac12}}^4.
\end{align*}

For $\mathcal{L}^5_R$, a similar argumet (but use Young/H\"older in a different way) yields the bound
\begin{align*}
&\sum _{\begin{smallmatrix}N_1\sim N\gg N_2\sim N_3\\ L_3\sim N_3\gg L_1,L_2\end{smallmatrix}}\frac{N^{\frac12}N_3}{N_3L_3}\| \omega_{N_1}\|_{L^2}\| \mathcal{F}[S(-t)P_{N_2}(w_{L_1}w_{L_2}\partial_x\bar{w}_{L_3})\|_{L^1}\| \omega_{N_3}\|_{L^1}\\
&\lesssim \sum _{\begin{smallmatrix}N_1\sim N\gg N_2\sim N_3\\ L_3\sim N_3\gg L_1,L_2\end{smallmatrix}}\frac{N^{\frac12}N_3N_2^{\frac12}L_3^{\frac12}}{N_3L_3N_1^{\frac12}}\| w_{N_1}\|_{H^{\frac12}}\| w_{L_1}\|_{L^\infty}\| w_{L_2}\|_{L^\infty}\| w_{L_3}\|_{H^{\frac12}}\| w_{N_3}\|_{H^{\frac12}}\\
&\lesssim \sum _{L_1,L_2}\| w_{L_1\vee L_2}\|_{L^\infty}\| w_{L_1\wedge L_2}\|_{H^{\frac12}}\sum _{N_2\sim N_3\sim L_3}\| w_{N_3}\|_{H^{\frac12}}\| w_{L_3}\|_{H^{\frac12}}\sum _{N_1\sim N}\| w_{N_1}\|_{H^{\frac12}}\\
&\lesssim \| w\|_{B^{0+}_{\infty,1}}\| w\|_{H^{\frac12}}^4,
\end{align*}
where at the last inequality we have used the fact that the sum in $N_1\sim N$ is originally the $\ell^2$ sum (this kind of observation will be made below without being mentioned).

For $\mathcal{L}^7_R$ to $\mathcal{L}^{10}_R$, we first apply H\"older in $\xi$ to obtain the $L^\infty_\xi$ norm multiplied by $N$.
Then, a similar argument with Young/H\"older yields the following bounds:
\begin{align*}
\| \mathcal{L}^7_R\|_{\hat{H}^{\frac12}}&\lesssim \sum _{\begin{smallmatrix}N_1\sim N_3\gg N\sim N_2\\ L_3\sim N_3\gg L_1,L_2\end{smallmatrix}}\frac{N}{N_3}\| w_{L_1}\|_{L^\infty}\| w_{L_2}\|_{L^\infty}\| w_{L_3}\|_{H^{\frac12}}\| w_{N_2}\|_{H^{\frac12}}\| w_{N_3}\|_{H^{\frac12}},\\
\| \mathcal{L}^8_R\|_{\hat{H}^{\frac12}}&\lesssim \sum _{\begin{smallmatrix}N_1\sim N_3\gg N\sim N_2\\ L_3\sim N_2\gg L_1,L_2\end{smallmatrix}}\| w_{L_1}\|_{L^\infty}\| w_{L_2}\|_{L^\infty}\| w_{L_3}\|_{H^{\frac12}}\| w_{N_1}\|_{H^{\frac12}}\| w_{N_3}\|_{H^{\frac12}},\\
\| \mathcal{L}^9_R\|_{\hat{H}^{\frac12}}&\lesssim \sum _{\begin{smallmatrix}N_1\sim N\gg N_2\sim N_3\\ L_1\sim L_3\gg N_1\gg L_2\end{smallmatrix}}\frac{N_1^{\frac12}}{L_1^{\frac12}}\| w_{L_1}\|_{H^{\frac12}}\| w_{L_2}\|_{L^\infty}\| w_{L_3}\|_{H^{\frac12}}\| w_{N_2}\|_{H^{\frac12}}\| w_{N_3}\|_{H^{\frac12}},\\
\| \mathcal{L}^{10}_R\|_{\hat{H}^{\frac12}}&\lesssim \sum _{\begin{smallmatrix}N_1\sim N_3\gg N\sim N_2\\ L_1\sim L_3\gg N_2 \gg L_2\end{smallmatrix}}\frac{N_2^{\frac12}}{L_1^{\frac12}}\| w_{L_1}\|_{H^{\frac12}}\| w_{L_2}\|_{L^\infty}\| w_{L_3}\|_{H^{\frac12}}\| w_{N_1}\|_{H^{\frac12}}\| w_{N_3}\|_{H^{\frac12}}.
\end{align*}
The summations are performed appropriately in each case, yielding the bound $\| w\|_{B^{0+}_{\infty,1}}\|w\|_{H^{\frac12}}^4$.

\medskip
{\bf Proof of \eqref{est:Lj0}:}
This time we can take the absolute value of each function and replace $|\phi(\vec{\xi})|^{-1}$ by $M^{-1+}(N_{\max}K)^{0-}\leq M^{-1+}N_{\max}^{0-}$, also $|\phi(\vec{\xi})+\phi(\vec{\eta})|^{-1}$ by the reciprocal of the dyadic size of $|\phi(\vec{\eta})|$, and then undo the decomposition in $K$.
Similarly to the estimate of $\mathcal{L}^j_R$, we begin with H\"older to get the $L^\infty_\xi$ norm multiplied by $N$ for $\mathcal{L}^7_0$ to $\mathcal{L}^{10}_0$.
This time we simply apply Young to the convolution of five $|\omega|$'s.
Then, we get the following bounds:
\begin{align*}
\| \mathcal{L}^5_0\|_{\hat{H}^{\frac12}}&\lesssim M^{-1+}\sum _{\begin{smallmatrix}N_1\sim N\gg N_2\sim N_3\\ L_3\sim N_3\gg L_1,L_2\end{smallmatrix}}N_{\max}^{0-}\| \omega _{N_1}\|_{\hat{H}^{\frac12}}\| \omega_{N_3}\|_{L^1}\| \omega_{L_1}\|_{L^1}\| \omega_{L_2}\|_{L^1}\| \omega_{L_3}\|_{L^1},\\
\| \mathcal{L}^6_0\|_{\hat{H}^{\frac12}}&\lesssim M^{-1+}\sum _{\begin{smallmatrix}N_1\sim N\gg N_2\sim N_3\\ L_3\sim N_1\gg L_1,L_2\end{smallmatrix}}N_{\max}^{0-}\frac{N_3}{N_1}\| \omega _{N_2}\|_{L^1}\| \omega_{N_3}\|_{L^1}\| \omega_{L_1}\|_{L^1}\| \omega_{L_2}\|_{L^1}\| \omega_{L_3}\|_{\hat{H}^{\frac12}},\\
\| \mathcal{L}^7_0\|_{\hat{H}^{\frac12}}&\lesssim M^{-1+}\sum _{\begin{smallmatrix}N_1\sim N_3\gg N\sim N_2\\ L_3\sim N_3\gg L_1,L_2\end{smallmatrix}}N_{\max}^{0-}\frac{N}{N_3}\| \omega _{N_2}\|_{L^1}\| \omega_{N_3}\|_{\hat{H}^{\frac12}}\| \omega_{L_1}\|_{L^1}\| \omega_{L_2}\|_{L^1}\| \omega_{L_3}\|_{\hat{H}^{\frac12}},\\
\| \mathcal{L}^8_0\|_{\hat{H}^{\frac12}}&\lesssim M^{-1+}\sum _{\begin{smallmatrix}N_1\sim N_3\gg N\sim N_2\\ L_3\sim N_2\gg L_1,L_2\end{smallmatrix}}N_{\max}^{0-}\| \omega _{N_1}\|_{\hat{H}^{\frac12}}\| \omega_{N_3}\|_{\hat{H}^{\frac12}}\| \omega_{L_1}\|_{L^1}\| \omega_{L_2}\|_{L^1}\| \omega_{L_3}\|_{L^1},\\
\| \mathcal{L}^9_0\|_{\hat{H}^{\frac12}}&\lesssim M^{-1+}\sum _{\begin{smallmatrix}N_1\sim N\gg N_2\sim N_3\\ L_1\sim L_3\gg N_1\gg L_2\end{smallmatrix}}N_{\max}^{0-}\frac{N_3}{L_3}\| \omega _{N_2}\|_{L^1}\| \omega_{N_3}\|_{L^1}\| \omega_{L_1}\|_{\hat{H}^{\frac12}}\| \omega_{L_2}\|_{L^1}\| \omega_{L_3}\|_{\hat{H}^{\frac12}},\\
\| \mathcal{L}^{10}_0\|_{\hat{H}^{\frac12}}&\lesssim M^{-1+}\sum _{\begin{smallmatrix}N_1\sim N_3\gg N\sim N_2\\ L_1\sim L_3\gg N_2 \gg L_2 \\ N_3\geq L_3 \end{smallmatrix}}N_{\max}^{0-}\| \omega _{N_1}\|_{\hat{H}^{\frac12}}\| \omega_{N_3}\|_{\hat{H}^{\frac12}}\| \omega_{L_1}\|_{L^1}\| \omega_{L_2}\|_{L^1}\| \omega_{L_3}\|_{L^1}\\
&\quad +M^{-1+}\sum _{\begin{smallmatrix}N_1\sim N_3\gg N\sim N_2\\ L_1\sim L_3\gg N_2 \gg L_2 \\ N_3<L_3 \end{smallmatrix}}N_{\max}^{0-}\frac{N_3}{L_3}\| \omega _{N_1}\|_{L^1}\| \omega_{N_3}\|_{L^1}\| \omega_{L_1}\|_{\hat{H}^{\frac12}}\| \omega_{L_2}\|_{L^1}\| \omega_{L_3}\|_{\hat{H}^{\frac12}}.
\end{align*}
For $\mathcal{L}^5_0$ to $\mathcal{L}^8_0$ and the first term of $\mathcal{L}^{10}_0$, $N_{\max}$ bounds all of $N,N_j,L_j$, and hence the factor $N_{\max}^{0-}$ ensures summability.
For $\mathcal{L}^9_0$ and the second term of $\mathcal{L}^{10}_0$, the factor $N_{\max}^{0-}\frac{N_3}{L_3}\lesssim L_3^{0-}$ is sufficient for summability.
Therefore, in each case we obtain the bound $M^{-1+}\| w\|_{H^{\frac12}}^5$.

\medskip
{\bf Proof of \eqref{est:Lj1}:}
Here, we need to exploit the decomposition in $K$.
In most cases (for $j=5,6,7,8$ and partially for $j=9,10$), it turns out that we can take the absolute value and obtain the estimate with $\| z\|_{H^{-\frac12-}}$.
Let us begin with $\mathcal{L}^6$.
We bound $\mathcal{L}^6_0[\omega]$ as follows (and obtain the estimate of $\mathcal{L}^6_1[\omega ,\zeta]$ by replacing one of $\omega$'s with $\zeta$):
\begin{align*}
|\mathcal{L}^6_0[\omega ]|\lesssim M^{-\frac12+}\sum _{\begin{smallmatrix}N_1\sim N\gg N_2\sim N_3\\ L_3\sim N_1\gg L_1,L_2 \\ K\lesssim N_2\end{smallmatrix}}\frac{N_3L_3}{(N_1K)^{\frac12+}N_1L_3}\psi_N\Big( \psi_{N_1}&\big( |\omega_{L_1}|*|\omega_{L_2}|*|\omega^*_{L_3}|\big) \Big) \\[-20pt]
&*\Big( \mathbf{1}_{|\cdot |\sim K}\big( |\omega_{N_2}|*|\omega^*_{N_3}|\big) \Big) ,
\end{align*}
and by Young/H\"older,
\begin{align*}
&\| \mathcal{L}^6_0[\omega ]\|_{\hat{H}^{\frac12}}\\
&\lesssim M^{-\frac12+}\sum _{\begin{smallmatrix}N_1\sim N\gg N_2\sim N_3\\ L_3\sim N_1\gg L_1,L_2 \\ K\lesssim N_2\end{smallmatrix}}\frac{N^{\frac12}N_3L_3}{(N_1K)^{\frac12+}N_1L_3}\| |\omega_{L_1}|*|\omega_{L_2}|*|\omega^*_{L_3}|\|_{L^2}\| \mathbf{1}_{|\cdot |\sim K}\big( |\omega_{N_2}|*|\omega^*_{N_3}|\big) \|_{L^1}\\
&\lesssim M^{-\frac12+}\sum _{\begin{smallmatrix}N_1\sim N\gg N_2\sim N_3\\ L_3\sim N_1\gg L_1,L_2 \\ K\lesssim N_2\end{smallmatrix}}\frac{N^{\frac12}N_3L_3K}{(N_1K)^{\frac12+}N_1L_3}\| \omega_{L_1}\|_{L^1}\| \omega_{L_2}\|_{L^1} \| \omega_{L_3}\|_{L^2} \| \omega_{N_2}\|_{L^2}\| \omega_{N_3}\|_{L^2}\\
&\lesssim M^{-\frac12+}\sum _{\begin{smallmatrix}N_1\sim N\gg N_2\sim N_3\\ L_3\sim N_1\gg L_1,L_2\end{smallmatrix}}\frac{1}{N_1^{1+}}\| \omega_{L_1}\|_{\hat{H}^{\frac12}}\| \omega_{L_2}\|_{\hat{H}^{\frac12}} \| \omega_{L_3}\|_{\hat{H}^{\frac12}} \| \omega_{N_2}\|_{\hat{H}^{\frac12}}\| \omega_{N_3}\|_{\hat{H}^{\frac12}}.
\end{align*}
Since $N_1\sim N_{\max(2)}:=\max \{N,N_j,L_j:j=1,2,3\}$, we can measure any of five $w$'s in $H^{-\frac12-}$ and still have $N_{\max(2)}^{0-}$ to ensure summability.
$\mathcal{L}^7$ is handled in a similar way: we have
\begin{align*}
&\| \mathcal{L}^7_0[\omega ]\|_{\hat{H}^{\frac12}}\\
&\lesssim M^{-\frac12+}\sum _{\begin{smallmatrix}N_1\sim N_3\gg N\sim N_2\\ L_3\sim N_3\gg L_1,L_2 \\ K\lesssim N\end{smallmatrix}}\frac{N^{\frac12}N_3L_3}{(N_1K)^{\frac12+}N_3L_3}\Big\| \Big( \textbf{1}_{|\cdot |\sim K}\big[ \big( |\omega_{L_1}|*|\omega_{L_2}|*|\omega^*_{L_3}|\big) * |\omega^*_{N_3}|\big] \Big) *|\omega_{N_2}|\Big\|_{L^2}\\
&\lesssim M^{-\frac12+}\sum _{\begin{smallmatrix}N_1\sim N_3\gg N\sim N_2\\ L_3\sim N_3\gg L_1,L_2 \\ K\lesssim N\end{smallmatrix}}\frac{N^{\frac12}N_3L_3K^{\frac12-}}{N_1^{\frac12+}N_3L_3}\| \omega_{L_1}\|_{L^1}\| \omega_{L_2}\|_{L^1} \| \omega_{L_3}\|_{L^2} \| \omega_{N_2}\|_{L^2}\| \omega_{N_3}\|_{L^2}\\
&\lesssim M^{-\frac12+}\sum _{\begin{smallmatrix}N_1\sim N_3\gg N\sim N_2\\ L_3\sim N_3\gg L_1,L_2\end{smallmatrix}}\frac{1}{N_1^{1+}}\| \omega_{L_1}\|_{\hat{H}^{\frac12}}\| \omega_{L_2}\|_{\hat{H}^{\frac12}} \| \omega_{L_3}\|_{\hat{H}^{\frac12}} \| \omega_{N_2}\|_{\hat{H}^{\frac12}}\| \omega_{N_3}\|_{\hat{H}^{\frac12}}.
\end{align*}
We next consider $\mathcal{L}^5$.
From the non-resonant condition $N_1K\ll N_3L_3$, we have $K\ll \frac{L_3^2}{N_1}$ and $N_1\ll L_3^2$.
Then, by a similar argument we obtain that
\begin{align*}
\| \mathcal{L}^5_0[\omega ]\|_{\hat{H}^{\frac12}}
&\lesssim M^{-\frac12+} \sum_{\begin{smallmatrix}N_1\sim N\gg N_2\sim N_3\\ L_3\sim N_3\gg L_1,L_2\\ N_1K\ll N_3L_3\end{smallmatrix}}\frac{K^{\frac12-}}{N_1^{\frac12+}L_3}\| \omega_{L_1}\|_{\hat{H}^{\frac12}}\| \omega_{L_2}\|_{\hat{H}^{\frac12}} \| \omega_{L_3}\|_{\hat{H}^{\frac12}} \| \omega_{N_1}\|_{\hat{H}^{\frac12}}\| \omega_{N_3}\|_{\hat{H}^{\frac12}}\\
&\lesssim M^{-\frac12+} \sum_{\begin{smallmatrix}N_1\sim N\gg N_2\sim N_3\\ L_3\sim N_3\gg L_1,L_2\\ N_1\ll L_3^2\end{smallmatrix}}\frac{1}{N_1L_3^{0+}}\| \omega_{L_1}\|_{\hat{H}^{\frac12}}\| \omega_{L_2}\|_{\hat{H}^{\frac12}} \| \omega_{L_3}\|_{\hat{H}^{\frac12}} \| \omega_{N_1}\|_{\hat{H}^{\frac12}}\| \omega_{N_3}\|_{\hat{H}^{\frac12}}\\
&\lesssim M^{-\frac12+} \sum_{\begin{smallmatrix}N_1\sim N\gg N_2\sim N_3\\ L_3\sim N_3\gg L_1,L_2\end{smallmatrix}}\frac{1}{N_1^{1+}}\| \omega_{L_1}\|_{\hat{H}^{\frac12}}\| \omega_{L_2}\|_{\hat{H}^{\frac12}} \| \omega_{L_3}\|_{\hat{H}^{\frac12}} \| \omega_{N_1}\|_{\hat{H}^{\frac12}}\| \omega_{N_3}\|_{\hat{H}^{\frac12}}.
\end{align*}
Since $N_1\sim N_{\max(2)}$, we have $\| z\|_{H^{-\frac12-}}$ and still $N_{\max(2)}^{0-}$ to perform the summation.
The treatment of $\mathcal{L}^8$ is similar: 
\begin{align*}
\| \mathcal{L}^8_0[\omega]\|_{\hat{H}^{\frac12}}&\lesssim M^{-\frac12+}\sum_{\begin{smallmatrix}N_1\sim N_3\gg N\sim N_2\\ L_3\sim N_2\gg L_1,L_2\\ N_1K\ll N_2L_3\end{smallmatrix}}\frac{K^{\frac12-}}{N_1^{\frac12+}L_3}\| \omega_{L_1}\|_{\hat{H}^{\frac12}}\| \omega_{L_2}\|_{\hat{H}^{\frac12}} \| \omega_{L_3}\|_{\hat{H}^{\frac12}} \| \omega_{N_1}\|_{\hat{H}^{\frac12}}\| \omega_{N_3}\|_{\hat{H}^{\frac12}}\\
&\lesssim M^{-\frac12+}\sum_{N,N_j,L_j}N_{\max(2)}^{-1-}\| \omega_{L_1}\|_{\hat{H}^{\frac12}}\| \omega_{L_2}\|_{\hat{H}^{\frac12}} \| \omega_{L_3}\|_{\hat{H}^{\frac12}} \| \omega_{N_1}\|_{\hat{H}^{\frac12}}\| \omega_{N_3}\|_{\hat{H}^{\frac12}}.
\end{align*}

Let us next consider $\mathcal{L}^9$.
A similar argument gives the following bound:
\begin{align*}
\| \mathcal{L}^9_0[\omega]\|_{\hat{H}^{\frac12}}&\lesssim M^{-\frac12+}\sum_{\begin{smallmatrix}N_1\sim N\gg N_2\sim N_3\\ L_1\sim L_3\gg N_1\gg L_2\\ K\lesssim N_2\end{smallmatrix}}\frac{NN_3L_3}{(N_1K)^{\frac12+}N_1L_3}\Big\| \big( |\omega_{L_1}|*|\omega_{L_2}|*|\omega^*_{L_3}|\big) \\[-25pt]
&\hspace{220pt} * \Big( \mathbf{1}_{|\cdot |\sim K}\big( |\omega_{N_2}|*|\omega^*_{N_3}|\big) \Big) \Big\| _{L^\infty}\\
&\lesssim M^{-\frac12+}\sum_{\begin{smallmatrix}N_1\sim N\gg N_2\sim N_3\\ L_1\sim L_3\gg N_1\gg L_2\\ K\lesssim N_2\end{smallmatrix}}\frac{NN_3L_3K^{\frac12-}}{N_1^{\frac12+}N_1L_3}\| |\omega_{L_1}|*|\omega_{L_2}|*|\omega^*_{L_3}|\|_{L^\infty} \| |\omega_{N_2}|*|\omega^*_{N_3}| \|_{L^\infty}\\
&\lesssim M^{-\frac12+}\sum_{\begin{smallmatrix}N_1\sim N\gg N_2\sim N_3\\ L_1\sim L_3\gg N_1\gg L_2\end{smallmatrix}}\frac{1}{N_1^{0+}L_3}\| \omega_{L_1}\|_{\hat{H}^{\frac12}}\| \omega_{L_2}\|_{\hat{H}^{\frac12}} \| \omega_{L_3}\|_{\hat{H}^{\frac12}} \| \omega_{N_2}\|_{\hat{H}^{\frac12}}\| \omega_{N_3}\|_{\hat{H}^{\frac12}}.
\end{align*}
Here, we have $L_1\sim L_3\gg N_1\gtrsim$ (others).
Hence, if $\zeta =\mathcal{N}[\omega ]$ replaces one of $\omega_{L_2},\omega_{N_2},\omega_{N_3}$, then we have $\| z\|_{H^{-\frac12-}}$ and still $L_3^{0-}$ to ensure the summability.
Moreover, even if $\omega_{L_1}$ or $\omega_{L_3}$ is replaced, the above estimate is sufficient to obtain the bound with $\| z\|_{H^{-\frac12}}$.
Therefore, from here we focus on the situation that $\zeta$ replaces $\omega_{L_1}$, and prove the estimate with $\| z\|_{B^0_{1,\infty}}$.
Similarly to the argument for $\mathcal{L}^4_1[\omega,\zeta]$ in the preceding subsection, we use the following decomposition of the multiplier:
\begin{gather}\label{mult-id2}
\frac{1}{\phi (\vec{\xi})\big[ \phi (\vec{\xi})+\phi (\vec{\eta})\big]}=\frac{1}{\phi(\vec{\xi})\phi(\vec{\eta})}-\frac{1}{\phi (\vec{\eta})\big[ \phi (\vec{\xi})+\phi (\vec{\eta})\big]}.
\end{gather}
The contribution from the second term is easier to estimate.
In fact, we have the multiplier bound $M^{-\frac12+}(N_1L_3)^{-\frac32-}$, and a crude estimate with Young (after undoing the decomposition in $K$ and applying H\"older in $\xi$) implies the bound
\begin{align*}
&M^{\frac12+}\sum_{\begin{smallmatrix}N_1\sim N\gg N_2\sim N_3\\ L_1\sim L_3\gg N_1\gg L_2\end{smallmatrix}}\frac{NN_3L_3}{N_1^{\frac32+}L_3^{\frac32+}}\| \omega_{L_1}\|_{L^2}\| \omega_{L_2}\|_{L^1}\| \omega_{L_3}\|_{L^2}\| \omega_{N_2}\|_{L^1}\| \omega_{N_3}\|_{L^1} \\
&\lesssim M^{-\frac12+}\sum_{N,N_j,L_j}N_{\max(2)}^{-1-}\| \omega_{L_1}\|_{\hat{H}^{\frac12}}\| \omega_{L_2}\|_{\hat{H}^{\frac12}} \| \omega_{L_3}\|_{\hat{H}^{\frac12}} \| \omega_{N_2}\|_{\hat{H}^{\frac12}}\| \omega_{N_3}\|_{\hat{H}^{\frac12}}.
\end{align*}
On the other hand, thanks to the separation of variables, the contribution from the first term can be written as
\begin{align*}
&\sum _{\begin{smallmatrix}N_1\sim N\gg N_2\sim N_3\\ L_1\sim L_3\gg N_1\gg L_2\\ M<N_1K\ll N_1L_3\end{smallmatrix}}\int _{\begin{smallmatrix}\xi=\xi_{123},\,\xi_1=\eta_{123}\\ K\leq |\xi_{23}|<2K\end{smallmatrix}}\frac{e^{it\phi(\vec{\xi})+it\phi(\vec{\eta})}}{\phi (\vec{\xi})\phi(\vec{\eta})}\psi_N(\xi)\psi_{N_1}(\xi_1)\xi_3\eta_3\\[-10pt]
&\hspace{160pt} \times \zeta_{L_1}(\eta_1)\omega_{L_2}(\eta_2)\omega^*_{L_3}(\eta_3)\omega _{N_2}(\xi_2)\omega^*_{N_3}(\xi_3)\\[5pt]
&\quad =\sum _{\begin{smallmatrix}N_1\sim N\gg N_2\sim N_3\\ L_1\sim L_3\gg N_1\gg L_2\\ M<N_1K\ll N_1L_3\end{smallmatrix}}\psi_N(\xi)\int _{\begin{smallmatrix}\xi=\xi_{123}\\ K\leq |\xi_{23}|<2K\end{smallmatrix}}\frac{e^{it\phi(\vec{\xi})}}{\phi (\vec{\xi})}\xi_3 \Omega (\xi_1)\omega _{N_2}(\xi_2)\omega^*_{N_3}(\xi_3),
\end{align*}
where 
\begin{align*}
\Omega (\xi_1)&=\Omega_{N_1,L_1,L_2,L_3}(\xi_1) :=\psi_{N_1}(\xi_1)\int _{\xi_1=\eta_{123}}\frac{e^{it\phi(\vec{\eta})}}{\phi(\vec{\eta})}\eta_3\zeta_{L_1}(\eta_1)\omega_{L_2}(\eta_2)\omega^*_{L_3}(\eta_3)\\
&=\mathcal{F}\Big[ S(-t)P_{N_1}\mathcal{F}^{-1}\Big[ \int _{\xi_1=\eta_{123}}\frac{\eta_3}{\phi(\vec{\eta})}\hat{z}_{L_1}(\eta_1)\hat{w}_{L_2}(\eta_2)\hat{\bar{w}}_{L_3}(\eta_3)\Big] \Big] .
\end{align*}
We first take the absolute value (keeping the structure of $\Omega$) and applying Young, which gives the bound
\begin{align*}
&M^{-\frac12+}\sum _{\begin{smallmatrix}N_1\sim N\gg N_2\sim N_3\\ L_1\sim L_3\gg N_1\gg L_2\\ K\lesssim N_2\end{smallmatrix}}\frac{N^{\frac12}N_3}{(N_1K)^{\frac12+}}\Big\| |\Omega|*\Big( \mathbf{1}_{|\cdot |\sim K}\big( |\omega_{N_2}|*|\omega^*_{N_3}|\big) \Big) \Big\|_{L^2} \\
&\lesssim M^{-\frac12+}\sum _{\begin{smallmatrix}N_1\sim N\gg N_2\sim N_3\\ L_1\sim L_3\gg N_1\gg L_2\end{smallmatrix}}\frac{N^{\frac12}N_2^{\frac12-}}{N_1^{\frac12+}}\| \Omega\|_{L^2}\| \omega_{N_2}\|_{\hat{H}^{\frac12}}\| \omega_{N_3}\|_{\hat{H}^{\frac12}}.
\end{align*}
Next, application of Plancherel, Bernstein and Lemma~\ref{lem:CM2} with
\[ b(\xi_1,\eta_1,\eta_2,\eta_3)=\frac{\eta_3}{\eta_{23}(\xi_1-\eta_2)},\qquad B=\frac{L_3}{L_3N_1}=N_1^{-1}\]
imply that
\[ \| \Omega \|_{L^2}\lesssim N_1^{\frac12}\| S(t)\mathcal{F}^{-1}\Omega \|_{L^1}\lesssim N_1^{-\frac12}\| z_{L_1}\|_{L^1}\| w_{L_2}\|_{L^\infty}\| w_{L_3}\|_{L^\infty}.\]
Inserting this into the above, we obtain the bound
\begin{align*}
&M^{-\frac12+}\sum_{L_1\sim L_3}\| z_{L_1}\|_{L^1}\| w_{L_3}\|_{L^\infty}\sum _{N,N_j,L_2\lesssim N_1}N_1^{0-}\| w_{L_2}\|_{H^{\frac12}}\| w_{N_2}\|_{H^{\frac12}}\| w_{N_3}\|_{H^{\frac12}} \\
&\lesssim M^{-\frac12+}\| z\|_{B^0_{1,\infty}}\| w\|_{B^0_{\infty,1}}\| w\|_{H^{\frac12}}^3.
\end{align*}

Finally, we consider the estimate of $\mathcal{L}^{10}_1[\omega ,\zeta]$.
For this term, we have either $N_1\sim N_3\sim N_{\max(2)}$ or $L_1\sim L_3\sim N_{\max(2)}$.
When the former occurs, we apply the previous argument for $\mathcal{L}^5$.
Noticing that the non-resonant condition $N_1K\ll N_2L_3$ implies $K\ll \frac{L_3^2}{N_1}$ and $N_1\ll L_3^2$, we take the absolute value, apply H\"older in $\xi$ and further Young/H\"older to obtain
\begin{align*}
\| \mathcal{L}^{10}_0[\omega]\|_{\hat{H}^{\frac12}}
&\lesssim M^{-\frac12+}\sum_{\begin{smallmatrix}N_1\sim N_3\gg N\sim N_2\\ L_1\sim L_3\gg N_2\gg L_2\\ N_1\gtrsim L_3,\,N_1K\ll N_2L_3\end{smallmatrix}}\frac{NN_3L_3}{(N_1K)^{\frac12+}N_2L_3}\Big\| \big( |\omega_{L_1}|*|\omega_{L_2}|*|\omega^*_{L_3}|\big) \\[-25pt]
&\hspace{220pt} * \Big( \mathbf{1}_{|\cdot |\sim K}\big( |\omega_{N_1}|*|\omega^*_{N_3}|\big) \Big) \Big\| _{L^\infty}\\
&\lesssim M^{-\frac12+}\sum_{\begin{smallmatrix}N_1\sim N_3\gg N\sim N_2\\ L_1\sim L_3\gg N_2\gg L_2\\ N_1\gtrsim L_3,\,N_1K\ll N_2L_3\end{smallmatrix}}\frac{N_3K^{\frac12-}}{N_1^{\frac12+}}\| |\omega_{L_1}|*|\omega_{L_2}|*|\omega^*_{L_3}|\|_{L^\infty} \| |\omega_{N_1}|*|\omega^*_{N_3}| \|_{L^\infty}\\
&\lesssim M^{-\frac12+}\sum_{\begin{smallmatrix}N_1\sim N_3\gg N\sim N_2\\ L_1\sim L_3\gg N_2\gg L_2\\ L_3\lesssim N_1\ll L_3^2\end{smallmatrix}}\frac{N_3L_3^{1-}}{N_1}\| \omega_{L_1}\|_{L^2}\| \omega_{L_2}\|_{L^1} \| \omega_{L_3}\|_{L^2} \| \omega_{N_1}\|_{L^2}\| \omega_{N_3}\|_{L^2}\\
&\lesssim M^{-\frac12+}\sum_{N,N_j,L_j}N_{\max(2)}^{-1-}\| \omega_{L_1}\|_{\hat{H}^{\frac12}}\| \omega_{L_2}\|_{\hat{H}^{\frac12}} \| \omega_{L_3}\|_{\hat{H}^{\frac12}} \| \omega_{N_1}\|_{\hat{H}^{\frac12}}\| \omega_{N_3}\|_{\hat{H}^{\frac12}}.
\end{align*}
It remains to consider the case $L_1\sim L_3\sim N_{\max(2)}\gg N_1\gtrsim$ (others), and we apply the previous argument for $\mathcal{L}^9$.
The same calculation gives the bound
\begin{align*}
&M^{-\frac12+}\sum_{\begin{smallmatrix}N_1\sim N_3\gg N\sim N_2\\ L_1\sim L_3\gg N_2\gg L_2\\ L_3\gg N_1,\,K\lesssim N\end{smallmatrix}}\frac{N_3K^{\frac12-}}{N_1^{\frac12+}}\| \omega_{L_1}\|_{L^2}\| \omega_{L_2}\|_{L^1} \| \omega_{L_3}\|_{L^2} \| \omega_{N_1}\|_{L^2}\| \omega_{N_3}\|_{L^2}\\
&\lesssim M^{-\frac12+}\sum_{\begin{smallmatrix}L_1\sim L_3\gg N_1\sim N_3\\ \gg N\sim N_2\gg L_2\end{smallmatrix}}\frac{1}{N_1^{0+}L_3}\| \omega_{L_1}\|_{\hat{H}^{\frac12}}\| \omega_{L_2}\|_{\hat{H}^{\frac12}} \| \omega_{L_3}\|_{\hat{H}^{\frac12}} \| \omega_{N_1}\|_{\hat{H}^{\frac12}}\| \omega_{N_3}\|_{\hat{H}^{\frac12}}.
\end{align*}
This gives: (i) when $\zeta$ replaces one of $\omega_{L_2},\omega_{N_1},\omega_{N_3}$, the bound with $\| z\|_{H^{-\frac12-}}$; and (ii) when $\zeta$ replaces either $\omega_{L_1}$ or $\omega_{L_3}$, the bound with $\| z\|_{H^{-\frac12}}$.
To obtain the bound with $\| z\|_{B^0_{1,\infty}}$ when $\omega_{L_1}$ is replaced by $\zeta$, for instance, we apply the decomposition \eqref{mult-id2} again.
To estimate the contribution from the second term, we bound the multiplier $|\phi (\vec{\eta})|^{-1}\big| \phi (\vec{\xi})+\phi (\vec{\eta})\big|^{-1}$ by $M^{-\frac12+}(N_2L_3)^{-\frac32-}$, undo the decomposition in $K$, and apply Young/H\"older as
\begin{align*}
&\Big\| \psi_N\Big[ |\omega_{N_1}|*|\omega^*_{N_3}|*\Big( \psi_{N_2}\big( |\zeta_{L_1}|*|\omega_{L_2}|*|\omega^*_{L_3}|\big) \Big) \Big] \Big\|_{\hat{H}^{\frac12}}\\
&\lesssim N\| \omega_{N_1}\|_{L^2}\| \omega_{N_3}\|_{L^2}\Big\| \psi_{N_2}\big( |\zeta_{L_1}|*|\omega_{L_2}|*|\omega^*_{L_3}|\big) \Big\|_{L^1} \\
&\lesssim N\| \omega_{N_1}\|_{L^2}\| \omega_{N_3}\|_{L^2}\cdot N_2\| \zeta_{L_1}\|_{L^2}\| \omega_{L_2}\|_{L^1}\| \omega_{L_3}\|_{L^2},
\end{align*}
which yields the bound
\begin{align*}
&M^{-\frac12+}\sum_{\begin{smallmatrix}L_1\sim L_3\gg N_1\sim N_3\\ \gg N\sim N_2\gg L_2\end{smallmatrix}}\frac{NN_2N_3L_3}{(N_2L_3)^{\frac32+}}\| \zeta_{L_1}\|_{L^2}\| \omega_{L_2}\|_{L^1}\| \omega_{L_3}\|_{L^2}\| \omega_{N_1}\|_{L^2}\| \omega_{N_3}\|_{L^2} \\
&\lesssim M^{-\frac12+}\sum_{\begin{smallmatrix}L_1\sim L_3\gg N_1\sim N_3\\ \gg N\sim N_2\gg L_2\end{smallmatrix}}\frac{1}{L_3^{1+}}\| \zeta_{L_1}\|_{\hat{H}^{\frac12}}\| \omega_{L_2}\|_{\hat{H}^{\frac12}} \| \omega_{L_3}\|_{\hat{H}^{\frac12}} \| \omega_{N_1}\|_{\hat{H}^{\frac12}}\| \omega_{N_3}\|_{\hat{H}^{\frac12}}\\
&\lesssim M^{-\frac12+}\| z\|_{H^{-\frac12-}}\| w\|_{H^{\frac12}}^4.
\end{align*}
The contribution from the first term is rewritten as
\begin{align*}
&\sum _{\begin{smallmatrix}L_1\sim L_3\gg N_1\sim N_3\\ \gg N\sim N_2\gg L_2,\\ M<N_1K\ll N_2L_3\end{smallmatrix}}\int _{\begin{smallmatrix}\xi=\xi_{123},\,\xi_2=\eta_{123}\\ K\leq |\xi_{13}|<2K\end{smallmatrix}}\frac{e^{it\phi(\vec{\xi})+it\phi(\vec{\eta})}}{\phi (\vec{\xi})\phi(\vec{\eta})}\psi_N(\xi)\psi_{N_2}(\xi_2)\xi_3\eta_3\\[-15pt]
&\hspace{160pt} \times \zeta_{L_1}(\eta_1)\omega_{L_2}(\eta_2)\omega^*_{L_3}(\eta_3)\omega _{N_1}(\xi_1)\omega^*_{N_3}(\xi_3)\\[5pt]
&\quad =\sum _{\begin{smallmatrix}L_1\sim L_3\gg N_1\sim N_3\\ \gg N\sim N_2\gg L_2,\\ M<N_1K\ll N_2L_3\end{smallmatrix}}\psi_N(\xi)\int _{\begin{smallmatrix}\xi=\xi_{123}\\ K\leq |\xi_{13}|<2K\end{smallmatrix}}\frac{e^{it\phi(\vec{\xi})}}{\phi (\vec{\xi})}\xi_3 \omega _{N_1}(\xi_1)\Omega (\xi_2)\omega^*_{N_3}(\xi_3),
\end{align*}
where 
\begin{align*}
\Omega (\xi_2)&=\Omega_{N_2,L_1,L_2,L_3}(\xi_2) :=\psi_{N_2}(\xi_2)\int _{\xi_2=\eta_{123}}\frac{e^{it\phi(\vec{\eta})}}{\phi(\vec{\eta})}\eta_3\zeta_{L_1}(\eta_1)\omega_{L_2}(\eta_2)\omega^*_{L_3}(\eta_3)\\
&=\mathcal{F}\Big[ S(-t)P_{N_2}\mathcal{F}^{-1}\Big[ \int _{\xi_2=\eta_{123}}\frac{\eta_3}{\phi(\vec{\eta})}\hat{z}_{L_1}(\eta_1)\hat{w}_{L_2}(\eta_2)\hat{\bar{w}}_{L_3}(\eta_3)\Big] \Big] .
\end{align*}
By Bernstein and Lemma~\ref{lem:CM2}, we have
\[ \| \Omega \|_{L^2}\lesssim N_2^{\frac12}\| S(t)\mathcal{F}^{-1}\Omega \|_{L^1}\lesssim N_2^{-\frac12}\| z_{L_1}\|_{L^1}\| w_{L_2}\|_{L^\infty}\| w_{L_3}\|_{L^\infty}.\]
Using this, we can derive the bound
\begin{align*}
&M^{-\frac12+}\sum _{\begin{smallmatrix}L_1\sim L_3\gg N_1\sim N_3\\ \gg N\sim N_2\gg L_2,\\ K\lesssim N\end{smallmatrix}}\frac{N^{\frac12}N_3K}{(N_1K)^{\frac12+}}\| \Omega\|_{L^2}\| \omega_{N_1}\|_{L^2}\| \omega_{N_3}\|_{L^2} \\
&\lesssim M^{-\frac12+}\sum _{\begin{smallmatrix}L_1\sim L_3\gg N_1\sim N_3\\ \gg N\sim N_2\gg L_2\end{smallmatrix}} \frac{N_3}{N_1^{0+}}\| z_{L_1}\|_{L^1}\| w_{L_2}\|_{L^\infty}\| w_{L_3}\|_{L^\infty} \| w_{N_1}\|_{L^2}\| w_{N_3}\|_{L^2}\\
&\lesssim M^{-\frac12+}\sum_{L_1\sim L_3}\| z_{L_1}\|_{L^1}\| w_{L_3}\|_{L^\infty}\sum _{N,N_j,L_2\lesssim N_1}N_1^{0-}\| w_{L_2}\|_{H^{\frac12}}\| w_{N_1}\|_{H^{\frac12}}\| w_{N_3}\|_{H^{\frac12}}\\
&\lesssim M^{-\frac12+}\| z\|_{B^0_{1,\infty}}\| w\|_{B^0_{\infty,1}}\| w\|_{H^{\frac12}}^3,
\end{align*}
as desired.
\end{proof}


\subsection{Justification of formal computations}
\label{subsec:just}

Here, we see how to justify each step in NFR argument.

We start with the integral equation
\begin{equation}\label{eq:omega-I}
\omega (t)=\omega (0) +\int_0^t \mathcal{R}[\omega(t')]\,dt'+\int _0^t\mathcal{N}[\omega (t')] \,dt'
\end{equation}
associated to \eqref{eq:omega}.
By Lemma~\ref{lem:R} and Proposition~\ref{prop:N-weak}, $\mathcal{R}[\omega]\in L^4_T\hat{H}^{\frac12}$ and $\mathcal{N}[\omega]\in L^4_T\hat{H}^{-\frac12-}$ for any $\omega (t)$ with $w(t)=S(t)\check{\omega}(t)\in L^\infty_TH^{\frac12}\cap L^4_TB^{0+}_{\infty,1}$, and hence the integrals in \eqref{eq:omega-I} belongs to $W^{1,4}_T\hat{H}^{-\frac12-}\subset C_T\hat{H}^{-\frac12-}$.
We also recall that $w(t)$ is a (distributional) solution to \eqref{eq:w2}, which implies $w\in C_TH^{\frac12-}(\mathcal{M})$ and thus $\omega \in C_T\hat{H}^{\frac12-}(\hat{\mathcal{M}})$.
Hence, the equation \eqref{eq:omega-I} holds pointwise in $t\in [-T,T]$ as $\hat{H}^{-\frac12-}$-valued functions.

It is worth noticing that the expression 
\begin{equation}\label{defn:N0}
\mathcal{N}[\omega(t)](\xi) = \frac{1}{2\pi}\int _{\begin{smallmatrix}\xi=\xi_{123}\\ |\xi_{13}|\wedge |\xi_{23}|\geq 1\end{smallmatrix}}e^{it\phi(\vec{\xi})}i\xi_3\omega (t,\xi_1)\omega (t,\xi_2)\omega ^*(t,\xi_3)
\end{equation}
is only \emph{formal}, because the integral in \eqref{defn:N0} does not in general converge absolutely for each $t,\xi$.
Indeed, the $L^4_TL^\infty_x$ control is available for $w=\mathcal{F}^{-1}[e^{-it\xi^2}\omega(t,\xi)]$ but not for $\mathcal{F}^{-1}[|\omega(t,\xi)|]$.
However, if we regard this term as the sum of localized terms as \eqref{defn:N1}--\eqref{defn:N2}, it is clear that the integral in \eqref{defn:N2} converges absolutely for each  $(t,\xi)\in [-T,T]\times \hat{\mathcal{M}}$.
We also have absolute convergence of the sum in \eqref{defn:N1} in the sense of $\hat{H}^{-\frac12-}$ for \emph{a.e.}~$t$, as in the next lemma.

\begin{lem}\label{lem:absconv}
Let $w\in L^\infty_TH^{\frac12}\cap L^4_TB^{0+}_{\infty,1}$ and $\omega (t):=\mathcal{F}S(-t)w(t)$.
Then, we have
\[ \int _{-T}^T\sum _{N,N_1,N_2,N_3}\| \mathcal{N}_{N,N_1,N_2,N_3}[\omega(t)]\|_{\hat{H}^{-\frac12-}}\,dt \ < \ \infty.\] 
\end{lem}

\begin{proof}
Similarly to Proposition~\ref{prop:N-weak}, we may estimate the whole cubic interaction
\begin{align*}
&\Big\| \mathcal{F}S(-t)\Big[ P_N\big[ w_{N_1}(t)w_{N_2}(t)\partial_x\bar{w}_{N_3}(t)\big] \Big] \Big\|_{\hat{H}^{-\frac12-}}\\
&\quad \lesssim N^{-\frac12-}\big\| P_N\big[ w_{N_1}(t)w_{N_2}(t)\partial_x\bar{w}_{N_3}(t)\big] \big\|_{L^2} \lesssim N^{0-}\| w_{N_1}(t)w_{N_2}(t)\partial_x\bar{w}_{N_3}(t) \|_{L^1}
\end{align*}
instead of $\mathcal{N}_{N,N_1,N_2,N_3}[\omega (t)]$.
Denote the maximum, median and minimum of $N_1,N_2,N_3$ by $N_{\max},N_{\mathrm{med}},N_{\min}$, respectively.
When the set $\{ (\xi ,\xi_l):\xi=\xi_{123},\,\xi\in I_N,\,\xi_l\in I_{N_l}\}$ is non-empty, it must hold either $N_{\max}\sim N_{\mathrm{med}}$ or $N_{\max}\sim N\gg N_{\mathrm{med}}$.
If $N_{\max}\sim N_{\mathrm{med}}$, we apply H\"older to the $L^1$ bound and obtain
\begin{align*}
&\lesssim N^{0-}N_3\| w_{N_{\max}}(t)\|_{L^2}\| w_{N_{\mathrm{med}}}(t)\|_{L^2}\| w_{N_{\min}}(t)\|_{L^\infty}\\
&\lesssim N^{0-}N_{\min}^{0-}\| w_{N_{\max}}(t)\|_{H^{\frac12}}\| w_{N_{\mathrm{med}}}(t)\|_{H^{\frac12}}\| w(t)\|_{B^{0+}_{\infty,1}},
\end{align*}
while if $N_{\max}\sim N\gg N_{\mathrm{med}}$ we estimate the $L^2$ bound with H\"older and Bernstein as
\begin{align*}
&\lesssim N^{-\frac12-}N_3\| w_{N_{\max}}(t)\|_{L^2}\| w_{N_{\mathrm{med}}}(t)\|_{L^\infty}\| w_{N_{\min}}(t)\|_{L^\infty}\\
&\lesssim N_{\max}^{0-}\| w(t)\|_{H^{\frac12}}^3.
\end{align*}
In each case, the resulting bound is summable in $\{ N,N_1,N_2,N_3\}$.
Given the above estimates for fixed $t$, we obtain the claim by applying the H\"older inequality.
\end{proof}

By Lemma~\ref{lem:absconv}, we can exchange the order of summation in $N,N_l$ and integration in $t$, to obtain the equation 
\[ \omega (t) =\omega (0) +\int_0^t\mathcal{R}[\omega]+\sum _{N,N_1,N_2,N_3}\int _0^t\mathcal{N}_{N,N_1,N_2,N_3}[\omega] ,\]
which holds pointwise in $t\in [-T,T]$ as $\hat{H}^{-\frac12-}_\xi$-valued functions.
Then, we fix $N,N_1,N_2,N_3$ and consider each integral:
\[ \int _0^t\mathcal{N}_{N,N_1,N_2,N_3}[\omega(t')](\xi)\,dt'=\int _0^t\psi_N(\xi)\int_{\begin{smallmatrix}\xi=\xi_{123}\\ |\xi_{13}|\wedge |\xi_{23}|\geq 1\end{smallmatrix}}e^{it'\phi(\vec{\xi})}\xi_3\omega_{N_1}(t',\xi_1)\omega_{N_2} (t',\xi_2)\omega_{N_3}^*(t',\xi_3)\,dt'.\]
After subtracting the resonant part, we have the trilinear term
\[ \int _0^t\psi_N(\xi)\int_{\xi=\xi_{123}}e^{it'\phi(\vec{\xi})}\mathbf{1}_{\text{NR},\,|\phi(\vec{\xi})|\gtrsim M}\xi_3\omega_{N_1}(t',\xi_1)\omega_{N_2} (t',\xi_2)\omega_{N_3}^*(t',\xi_3)\,dt',\]
where ``NR'' means the condition $|\xi_{13}|\wedge |\xi_{23}|\geq 1$, and the precise meaning of the restriction ``$|\phi(\vec{\xi})|\gtrsim M$'' depends on the relation between $N,N_1,N_2,N_3$.
Now, we need to justify the following NFR operation:
\begin{equation}\label{id:just}
\begin{aligned}
&\int _0^t\int_{\xi=\xi_{123}}e^{it'\phi}\mathbf{1}_{\text{NR},\,|\phi(\vec{\xi})|\gtrsim M}\xi_3\omega_{N_1}(t',\xi_1)\omega_{N_2} (t',\xi_2)\omega_{N_3}^*(t',\xi_3)\,dt'\\
&\quad=\int_{\xi=\xi_{123}}\frac{e^{it'\phi}}{i\phi}\mathbf{1}_{\text{NR},\,|\phi(\vec{\xi})|\gtrsim M}\xi_3\omega_{N_1}(t',\xi_1)\omega_{N_2} (t',\xi_2)\omega_{N_3}^*(t',\xi_3)\bigg|_{t'=0}^{t'=t}\\
&\qquad -\int_0^t\int_{\xi=\xi_{123}}\frac{e^{it'\phi}}{i\phi}\mathbf{1}_{\text{NR},\,|\phi(\vec{\xi})|\gtrsim M}\xi_3\Big\{ \zeta_{N_1}(t',\xi_1)\omega_{N_2} (t',\xi_2)\omega_{N_3}^*(t',\xi_3)\\
&\hspace*{50pt} +\omega_{N_1}(t',\xi_1)\zeta_{N_2} (t',\xi_2)\omega_{N_3}^*(t',\xi_3)+\omega_{N_1}(t',\xi_1)\omega_{N_2} (t',\xi_2)\zeta_{N_3}^*(t',\xi_3) \Big\} \,dt',
\end{aligned}
\end{equation}
where $\omega_{N_l}=\psi_{N_l}\omega \in W^{1,4}(-T,T;L^2(\hat{\mathcal{M}}))\cap C([-T,T];L^2(\hat{\mathcal{M}}))$ and 
\[ \zeta_{N_l}:=\partial_t\omega_{N_l} = \psi_{N_l}\big( \mathcal{N}[\omega ]+\mathcal{R}[\omega ]\big) \in L^4_TL^2(\hat{\mathcal{M}}) \]
for each $N_l$.
Since we do not know if $\omega _{N_l}(t,\xi)$ is continuously differentiable in $t$ (neither pointwise in $\xi$ nor in the $L^2_\xi$ sense), the above operation (i.e., integration by parts and application of the product rule) must be interpreted in a weak sense, which is done in the next lemma.

\begin{lem}\label{lem:just}
The equality \eqref{id:just} holds for any $t\in [-T,T]$ and $\xi\in \hat{\mathcal{M}}$.
\end{lem}

\begin{proof}
We write \eqref{id:just} as
\[ \int_0^t\int_{\xi=\xi_{123}}e^{it'\phi}f(t',\vec{\xi})\,dt'=\int_{\xi=\xi_{123}}\frac{e^{it'\phi}}{i\phi }f(t',\vec{\xi})\bigg|_0^t-\int_0^t\int_{\xi=\xi_{123}}\frac{e^{it'\phi}}{i\phi}g(t',\vec{\xi})\,dt', \]
where
\begin{align*}
f(t,\vec{\xi})&:=\mathbf{1}_{\text{NR},\,|\phi(\vec{\xi})|\gtrsim M}\xi_3\omega_{N_1}(t,\xi_1)\omega_{N_2}(t,\xi_2)\omega^*_{N_3}(t,\xi_3),\\
g(t,\vec{\xi})&:=\mathbf{1}_{\text{NR},\,|\phi(\vec{\xi})|\gtrsim M}\xi_3\Big\{ \zeta_{N_1}(t,\xi_1)\omega_{N_2}(t,\xi_2)\omega^*_{N_3}(t,\xi_3)\\
&\hspace{90pt}+\omega_{N_1}(t,\xi_1)\zeta_{N_2}(t,\xi_2)\omega^*_{N_3}(t,\xi_3)+\omega_{N_1}(t,\xi_1)\omega_{N_2}(t,\xi_2)\zeta^*_{N_3}(t,\xi_3)\Big\} .
\end{align*}
Note that every integral in $\xi_1,\xi_2,\xi_3$ in \eqref{id:just} converges absolutely for each $\xi$ due to the localization.
We show \eqref{id:just} by taking the limit $h\to 0$ of the following identity:
\begin{multline*}
\int_0^t \int_{\xi=\xi_{123}}\frac{e^{it'\phi}}{i\phi}\frac{f(t'+h,\vec{\xi})-f(t',\vec{\xi})}{h}\,dt'\ =\ \int _h^t\int_{\xi=\xi_{123}}\frac{e^{i(t'-h)\phi}-e^{it'\phi}}{ih\phi}f(t',\vec{\xi})\,dt' \\
+\frac{1}{h}\int_t^{t+h}\int _{\xi=\xi_{123}}\frac{e^{i(t'-h)\phi}}{i\phi}f(t',\vec{\xi})\,dt'-\frac{1}{h}\int_0^h\int _{\xi=\xi_{123}}\frac{e^{it'\phi}}{i\phi}f(t',\vec{\xi})\,dt'.
\end{multline*}
Therefore, it suffices to show the following convergences as $h\to 0$ (with $t+h\in [-T,T]$) for any $t\in [-T,T]$ and $\xi\in \hat{\mathcal{M}}$:
\begin{gather}
\int_h^t\int_{\xi=\xi_{123}}\frac{e^{i(t'-h)\phi}-e^{it'\phi}}{ih\phi}f(t',\vec{\xi})\,dt'\ \to \ -\int _0^t\int_{\xi=\xi_{123}}e^{it'\phi}f(t',\vec{\xi})\,dt', \label{conv1}\\
\frac{1}{h}\int_t^{t+h}\!\!\!\int _{\xi=\xi_{123}}\!\!\!\frac{e^{i(t'-h)\phi}}{i\phi}f(t',\vec{\xi})\,dt'-\frac{1}{h}\int_0^h\!\!\!\int _{\xi=\xi_{123}}\!\!\frac{e^{it'\phi}}{i\phi}f(t',\vec{\xi})\,dt' \to \int _{\xi=\xi_{123}}\!\!\frac{e^{it'\phi}}{i\phi}f(t',\vec{\xi})\bigg|_0^t, \label{conv2}\\
\int_0^t \int_{\xi=\xi_{123}}\frac{e^{it'\phi}}{i\phi}\frac{f(t'+h,\vec{\xi})-f(t',\vec{\xi})}{h}\,dt'\ \to \ \int_0^t \int_{\xi=\xi_{123}}\frac{e^{it'\phi}}{i\phi}g(t',\vec{\xi})\,dt'. \label{conv3}
\end{gather}

\medskip
{\bf Proof of \eqref{conv1}:} 
By the dominated convergence theorem, it suffices to prove that
\[ \int_{-T}^T\int_{\xi=\xi_{123}}|f(t',\vec{\xi})|\,dt'<\infty.\]
This is easy to see from the estimate
\begin{align*}
\sup _{\xi}\int_{\xi=\xi_{123}}|f(t',\vec{\xi})|&\lesssim N_3\Big\| |\omega_{N_1}(t')|*|\omega_{N_2}(t')|*|\omega^*_{N_3}(t')|\Big\|_{L^\infty} \lesssim N_3N_{\min}^{\frac12}\prod_{l=1}^3\| \omega_{N_l}(t')\|_{L^2}
\end{align*}
and the fact that $\omega_{N_l}\in C_TL^2$.

\medskip
{\bf Proof of \eqref{conv2}:} 
We first observe that
\[ \Big| \frac{1}{h}\int_t^{t+h}\int _{\xi=\xi_{123}}\frac{e^{i(t'-h)\phi}-e^{it'\phi}}{i\phi}f(t',\vec{\xi})\,dt'\Big| \lesssim \int_t^{t+h}\int _{\xi=\xi_{123}}|f(t',\vec{\xi})|\,dt' \to 0\qquad (h\to 0).\]
It then suffices to show that the map
\[ [-T,T]\ni ~t\quad \mapsto \quad \int_{\xi=\xi_{123}}\frac{e^{it\phi}}{i\phi}f(t,\vec{\xi})\quad \in L^\infty_\xi \]
is continuous.
On one hand, similarly to (i) we have
\begin{align*}
\Big\| \int_{\xi=\xi_{123}}\frac{e^{it'\phi}-e^{it\phi}}{i\phi}f(t',\vec{\xi})\Big\|_{L^\infty_\xi}&\lesssim |t'-t|N_3\Big\| |\omega_{N_1}(t')|*|\omega_{N_2}(t')|*|\omega^*_{N_3}(t')|\Big\|_{L^\infty} \\
&\lesssim |t'-t|N_3N_{\min}^{\frac12}\prod_{l=1}^3\| \omega_{N_l}(t')\|_{L^2}\to 0\qquad (t'\to t).
\end{align*}
On the other hand, using the fact $\omega_{N_l}\in C_TL^2$ and estimating similarly, we have
\begin{align*}
&\Big\| \int_{\xi=\xi_{123}}\frac{e^{it\phi}}{i\phi}\Big( f(t',\vec{\xi})-f(t,\vec{\xi})\Big)\Big\|_{L^\infty_\xi}\lesssim M^{-1}N_3N_{\min}^{\frac12}\Big\{ \| \omega_{N_1}(t')-\omega_{N_1}(t)\|_{L^2}\| \omega_{N_2}(t')\|_{L^2}\| \omega_{N_3}(t')\|_{L^2}\\
&\quad+\| \omega_{N_1}(t)\|_{L^2}\| \omega_{N_2}(t')-\omega_{N_2}(t)\|_{L^2}\| \omega_{N_3}(t')\|_{L^2}+\| \omega_{N_1}(t)\|_{L^2}\| \omega_{N_2}(t)\|_{L^2}\| \omega_{N_3}(t')-\omega_{N_3}(t)\|_{L^2}\Big\} \\
&\to 0\qquad (t'\to t).
\end{align*}
Hence, we obtain the continuity as desired.

\medskip
{\bf Proof of \eqref{conv3}:}
We begin with the identity
\begin{align*}
&\frac{f(t'+h)-f(t')}{h}-g(t')\\
&=\mathbf{1}_{|\phi |\gtrsim M}\xi_3 \Big\{ \begin{aligned}[t]
&\Big( \frac{\omega_{N_1}(t'+h)-\omega_{N_1}(t')}{h}-\zeta_{N_1}(t')\Big) \omega _{N_2}(t'+h)\omega_{N_3}^*(t'+h)\\
&+\omega_{N_1}(t')\Big( \frac{\omega_{N_2}(t'+h)-\omega_{N_2}(t')}{h}-\zeta_{N_2}(t')\Big) \omega_{N_3}^*(t'+h)\\
&+\omega _{N_1}(t')\omega_{N_2}(t')\Big( \frac{\omega^*_{N_3}(t'+h)-\omega^*_{N_3}(t')}{h}-\zeta^*_{N_3}(t')\Big) \\
&+\zeta_{N_1}(t')\Big( \omega_{N_2}(t'+h)\omega^*_{N_3}(t'+h)-\omega_{N_2}(t')\omega^*_{N_3}(t')\Big) \\
&+\omega_{N_1}(t')\zeta_{N_2}(t')\Big( \omega^*_{N_3}(t'+h)-\omega^*_{N_3}(t')\Big) \Big\} .
\end{aligned}
\end{align*}
Then, a similar argument implies
\begin{align*}
&\sup _\xi \Big| \int_0^t \int_{\xi=\xi_{123}}\frac{e^{it'\phi}}{i\phi}\Big( \frac{f(t'+h,\vec{\xi})-f(t',\vec{\xi})}{h}-g(t',\vec{\xi})\Big) \,dt'\Big| \\
&\quad \lesssim \int_{-|t|}^{|t|}\sup_\xi \int _{\xi=\xi_{123}}\frac{1}{|\phi|}\Big| \frac{f(t'+h,\vec{\xi})-f(t',\vec{\xi})}{h}-g(t',\vec{\xi})\Big| \,dt' \\
&\quad \lesssim T^{\frac34}M^{-1}N_3N_{\min}^{\frac12}\Big\{ \sum_{l=1}^3\Big\| \frac{\omega_{N_l}(\cdot +h)-\omega_{N_l}(\cdot )}{h}-\zeta_{N_l}(\cdot )\Big\|_{L^4_{|t|}L^2}\prod_{m\neq l}\| \omega_{N_m}\|_{L^\infty_TL^2} \\
&\qquad\qquad +\sum_{(k,l,m)\in \{ (1,2,3),(1,3,2),(2,1,3)\}}\| \zeta_{N_k}\|_{L^4_TL^2}\| \omega_{N_l}\|_{L^\infty_TL^2}\| \omega_{N_m}(\cdot +h)-\omega_{N_m}(\cdot )\|_{L^\infty_{|t|}L^2} \Big\} \\
&\quad \to 0\qquad (h\to 0),
\end{align*}
where in the last step we have used the convergences
\[ \| \omega _{N_m}(\cdot +h)-\omega _{N_m}(\cdot)\|_{L^\infty_{T'}L^2}\to 0,\qquad \Big\| \frac{\omega _{N_l}(\cdot +h)-\omega _{N_l}(\cdot)}{h}-\zeta_{N_l}(\cdot )\Big\|_{L^4_{T'}L^2}\to 0\qquad (h\to 0)\]
for $0<T'<T$.%
\footnote{%
The argument is also valid for $t=\pm T$ by considering a suitable half of the interval $[-T,T]$ and a one-sided limit in $h$.%
}
The first one is by the uniform continuity of $t\mapsto \omega_{N_m}(t)\in L^2_\xi$, while the second one is a general property of the Sobolev space $W^{1,4}_TL^2_\xi$ (see, e.g., \cite[Corollary~1.4.39]{CazHar}).
\end{proof}

After the first NFR, we have various terms such as
\begin{align*}
\sum_{N,N_1,N_2,N_3}\int_0^t\psi_N(\xi) \int_{\xi=\xi_{123}}&\frac{e^{it'\phi}}{i\phi}\mathbf{1}_{\text{NR},\,|\phi|\gtrsim M}\xi_3\\
&\quad \times \Big[ \sum _{L_1,L_2,L_3}\mathcal{N}_{N_1,L_2,L_2,L_3}[\omega (t')](\xi_1)\Big] \omega_{N_2} (t',\xi_2)\omega_{N_3}^*(t',\xi_3)\,dt'.
\end{align*}
When the dyadic frequencies $N,N_l,L_l$ are in a certain range, we need to apply NFR once more (see the last two subsections).
For such $N,N_l,L_l$, we need to exchange the order of summation in $L_l$ and integration in $t',\xi_l$.
This is verified from the fact
\[ \sum_{L_1,L_2,L_3}\int_{-T}^T\int_{\xi=\xi_{123}}\big| \mathcal{N}_{N_1,L_2,L_2,L_3}[\omega(t')](\xi_1)\big| |\omega_{N_2}(t',\xi_2)||\omega^*_{N_3}(t',\xi_3)|\,dt' <\infty \]
for each $\xi$ and $N,N_l$, which is a consequence of the Young inequality and Lemma~\ref{lem:absconv}.
Then, we fix all of $N,N_l,L_l$ and apply NFR to the localized integral
\begin{align*}
\int_0^t\psi_N(\xi) \int_{\begin{smallmatrix} \xi=\xi_{123} \\ \xi_1=\eta_{123}\end{smallmatrix}}&e^{it'(\phi(\vec{\xi})+\phi(\vec{\eta}))}\mathbf{1}_{\text{NR},\,|\phi(\vec{\xi})|\gtrsim M}\mathbf{1}_{\text{NR},\,|\phi(\vec{\xi})+\phi(\vec{\eta})|\gtrsim M}\frac{\xi_3\eta_3\psi_{N_1}(\xi_1)}{\phi(\vec{\xi})} \\
&\times \omega_{L_1}(t',\eta_1)\omega_{L_2} (t',\eta_2)\omega_{L_3}^*(t',\eta_3)\omega_{N_2} (t',\xi_2)\omega_{N_3}^*(t',\xi_3)\,dt',
\end{align*}
which is clearly absolutely convergent.
The second application of NFR is justified in exactly the same way as Lemma~\ref{lem:just}, so we omit the detail.

After the second NFR, we obtain the following equation (written in a formal way):
\begin{equation}\label{eq:NFR2}
\begin{aligned}
\omega (t)&= \omega (0)+ \int_0^t \mathcal{R}[\omega ]+\int_0^t \mathcal{N}^{\texttt{A}}[\omega] +\int_0^t \mathcal{N}^{\neg \texttt{A}}_{R}[\omega ] +\mathcal{N}^{\neg \texttt{A}}_0[\omega ] \Big|_0^t + \int_0^t \mathcal{N}^{\neg \texttt{A}}_1[\omega,\mathcal{R}[\omega]]\\
&\quad +\int_0^t\tilde{\mathcal{N}}_1[\omega ]+\sum_{k=1}^{10}\Big( \int _0^t\mathcal{L}^k_R[\omega ]+\mathcal{L}^k_0[\omega ]\Big|_0^t +\int_0^t\mathcal{L}^k_1[\omega ,\mathcal{R}[\omega]]+\int_0^t\mathcal{L}^k_1[\omega ,\mathcal{N}[\omega]]\Big) ,
\end{aligned}
\end{equation}
where
\begin{align*}
\tilde{\mathcal{N}}_1[\omega ]&:= \mathcal{N}^{\neg \texttt{A}}_1[\omega ,\mathcal{N}[\omega ]]-\sum_{k=1}^{10}\mathcal{L}^k[\omega] \\
&\;=\int_0^t \Big( \mathcal{N}^{\texttt{E}}_1+\mathcal{N}^{\texttt{C2}}_{1,1}+\mathcal{N}^{\texttt{C},\texttt{D}}_{1,2}+\mathcal{N}^{\texttt{D}}_{1,3}\Big) [\omega,\mathcal{N}[\omega]] +\int_0^t \Big( \mathcal{N}^{\texttt{B2}}_{1,1}+\mathcal{N}^{\texttt{B1}}_{1,2} \Big) [\omega ,\mathcal{N}^{\neg \texttt{C}}[\omega ]] \\
&\qquad +\int_0^t \Big( \mathcal{N}^{\texttt{D}}_{1,1}+\mathcal{N}^{\texttt{C2}}_{1,3} \Big) [\omega ,\mathcal{N}^{\neg \texttt{C1}}[\omega ]] +\int_0^t \Big( \mathcal{N}^{\texttt{B1},\texttt{C1}}_{1,1}+\mathcal{N}^{\texttt{B2}}_{1,2}+\mathcal{N}^{\texttt{B},\texttt{C1}}_{1,3}\Big) [\omega ,\mathcal{N}^{\neg \texttt{C2}}[\omega ]].
\end{align*}
Let us conclude this subsection by considering how to make sense of each term in \eqref{eq:NFR2}.
The first two integrals of $\mathcal{R}[\omega]$ and $\mathcal{N}^{\texttt{A}}[\omega ]$ have been considered in \eqref{eq:omega-I}. 
The remaining terms are also defined as the $t$-integral of the absolutely-convergent (in $\hat{H}^{\frac12-}$) sums of frequency-localized components (see Remark~\ref{rem:N_R} for the last three terms in the first line; absolute convergence for the other terms follows similarly from the proofs of Lemmas~\ref{lem:N_1E}--\ref{lem:L2}).
For instance, the last term $\int_0^t \mathcal{L}^{10}_1[\omega ,\mathcal{N}[\omega]]$ in \eqref{eq:NFR2} was initially obtained in the form
\[ \sum_{N,N_l,K}\bigg( \sum_{L_l}\int_0^t \big[ \text{$\mathcal{L}^{10}_1[\omega,\mathcal{N}[\omega]]$ localized according to $N,N_l,K,L_l$}\big] \bigg) ,\]
where absolute convergence of the $\xi_l,\eta_l$-integrals in localized components is always clear due to localization, but the summations in $N,N_l,K$ and in $L_l$ are still formal.
Now, the proof of the estimate \eqref{est:Lj1} in Lemma~\ref{lem:L2} shows%
\footnote{%
Although we have absolute convergence in $\hat{H}^{\frac12}$ in the case of $\mathcal{L}^{10}_1[\omega,\mathcal{N}[\omega]]$, for some terms (e.g., $\mathcal{L}^{5}_R[\omega]$) we use a slightly weaker $\hat{H}^{\frac12-}$ norm to ensure $\ell^1$ summability in $N$.}
\[ \sum _{N,N_l,K,L_l}\| \mathcal{L}^{10}_1[\omega,\mathcal{N}[\omega]] \|_{\hat{H}^{\frac12}}\lesssim M^{-\frac12+}\| w\|_{H^{\frac12}}^3\| w\|_{H^{\frac12}\cap B^{0+}_{\infty,1}}\| S(t)\mathcal{F}^{-1}\mathcal{N}[\omega]\|_{H^{-\frac12}+B^0_{1,\infty}},\]
and hence, by Proposition~\ref{prop:N-weak} and the H\"older inequality, 
\[ \int_{-T}^T\sum _{N,N_l,K,L_l}\| \mathcal{L}^{10}_1[\omega,\mathcal{N}[\omega]] \|_{\hat{H}^{\frac12}}<\infty \]
for any $\omega(t)$ such that $w(t)=S(t)\check{\omega}(t)\in L^\infty_TH^{\frac12}\cap L^4_TB^{0+}_{\infty,1}$.
Hence, we can exchange the order of summation in $N,N_l,K,L_l$ and integration in $t$, obtaining the expression $\int_0^t \mathcal{L}^{10}_1[\omega ,\mathcal{N}[\omega]]$ in \eqref{eq:NFR2}.
From Proposition~\ref{prop:N-weak} and Lemmas~\ref{lem:R}--\ref{lem:L2}, we see that all the integrated terms in \eqref{eq:NFR2} (except for $\int_0^t\mathcal{N}^{\texttt{A}}[\omega]$) are the $t$-integrals of functions in $L^2_T\hat{H}^{\frac12}$.


\subsection{Conclusion of the first stage}

We set the non-integrated part of the nonlinearity in \eqref{eq:NFR2} as
\[ \mathcal{W}_0[\omega ]:=\mathcal{N}^{\neg \texttt{A}}_0[\omega ]+\sum _{j=1}^{10}\mathcal{L}^j_0[\omega ] ,\]
and introduce the new unknown function
\[ \varpi := \omega -\mathcal{W}_0[\omega ]. \]
Then, the equation \eqref{eq:NFR2} can be rewritten as 
\begin{equation}\label{eq:varpi}
\varpi (t) = \varpi (0) +\int_0^t \mathcal{N}^{\texttt{A}}[\varpi ] +\int_0^t \mathcal{W}_1[\omega ],
\end{equation}
where 
\begin{align*}
\mathcal{W}_1[\omega ]&:=\mathcal{R}[\omega ]+\mathcal{N}^{\neg \texttt{A}}_{R}[\omega ]+\mathcal{N}^{\neg \texttt{A}}_1[\omega,\mathcal{R}[\omega]] +\tilde{\mathcal{N}}_1[\omega] \\
&\quad +\sum_{j=1}^{10}\Big( \mathcal{L}^j_R[\omega ]+\mathcal{L}^j_1[\omega ,\mathcal{R}[\omega]] +\mathcal{L}^j_1[\omega ,\mathcal{N}[\omega]]\Big) +\Big( \mathcal{N}^{\texttt{A}}[\omega ]-\mathcal{N}^{\texttt{A}}[\omega -\mathcal{W}_0[\omega]]\Big) .
\end{align*}
It is to the equation \eqref{eq:varpi} that we will apply infinite NFR scheme in Section~\ref{sec:NFR2}.

We show the regularity properties of $\mathcal{W}_0[\omega]$ and $\mathcal{W}_1[\omega]$ in the next two lemmas.
In particular, it turns out that $\mathcal{W}_0[\omega]$ enjoys an extra smoothing property, which ensures the $\hat{H}^{\frac12}$ control of the last difference term in $\mathcal{W}_1[\omega]$ (note that $\mathcal{N}^{\texttt{A}}[\omega]$ itself admits only the $\hat{H}^{-\frac12+}$ control; see the proof of Proposition~\ref{prop:N-weak}).
\begin{lem}\label{lem:W0}
Let $T\in (0,1]$.
For any $\omega$ with $w=S(t)\check{\omega}\in L^\infty_TH^{\frac12}\cap L^4_TB^{0+}_{\infty,1}$, the function $\mathcal{W}_0[\omega]$ belongs to $C_T\hat{H}^{\frac12}$ and satisfies
\begin{equation}\label{est:W00}
\|\mathcal{W}_0[\omega]\|_{L^\infty_T\hat{H}^{\frac12}}\lesssim \mathcal{P}_0\Big( M^{-\frac12+}\| w\|_{L^\infty_TH^{\frac12}}^2\Big) \| w\|_{L^\infty_TH^{\frac12}},
\end{equation}
where $\mathcal{P}_0(x):=x+x^2$.
Moreover, the following estimates hold:
\begin{gather}
\sup_NN\big\| P_NS(t)\mathcal{F}^{-1}\mathcal{W}_0[\omega ]\big\|_{L^4_TL^\infty}\lesssim \| w\|_{L^\infty_TH^{\frac12}}^2\Big( 1+M^{-\frac12+}\| w\|_{L^\infty_TH^{\frac12}}^2\Big) \| w\|_{L^\infty_TH^{\frac12}\cap L^4_TB^{0+}_{\infty,1}},\qquad \label{est:W0} \\
\begin{aligned}
&\Big\| \mathcal{N}^{\texttt{A}}[\omega ]-\mathcal{N}^{\texttt{A}}[\omega -\mathcal{W}_0[\omega]]\Big\|_{L^4_T\hat{H}^{\frac12}}\\
&\quad \lesssim \| w\|_{L^\infty_TH^{\frac12}}^3\Big( 1+M^{-\frac12+}\| w\|_{L^\infty_TH^{\frac12}}^2\Big) ^5\Big( 1+\| w\|_{L^\infty_TH^{\frac12}}\| w\|_{L^\infty_TH^{\frac12}\cap L^4_TB^{0+}_{\infty,1}}\Big). 
\end{aligned} \label{est:W0-N}
\end{gather}
\end{lem}

\begin{proof}
The first estimate \eqref{est:W00} is obtained from Lemma~\ref{lem:N_0} and the estimate \eqref{est:Lj0} in Lemmas~\ref{lem:L1}, \ref{lem:L2}.
To see that $\mathcal{N}^{\neg \texttt{A}}[\omega]\in C_T\hat{H}^{\frac12}$, we first notice that it is true for each localized components, which we denote by $\mathcal{N}_{0,N,N_1,N_2,N_3}[\omega]$ (see the proof of \eqref{conv2} in Lemma~\ref{lem:just}).
Then, (except for Case \texttt{A}) we see from the proof of Lemma~\ref{lem:N_0} that
\[ \big\| \mathcal{N}_{0,N,N_1,N_2,N_3}[\omega]\big\|_{L^\infty_T\hat{H}^{\frac12}}\lesssim M^{-\frac12+}\big( NN_1N_2N_3\big) ^{0-}\| \omega\|_{L^\infty_T\hat{H}^{\frac12}}^3.\]
Hence, the sum of these localized pieces is absolutely convergent in Banach space $C_T\hat{H}^{\frac12}$ and gives $\mathcal{N}_0^{\neg \texttt{A}}[\omega]\in C_T\hat{H}^{\frac12}$.
The same argument using the proof of \eqref{est:Lj0} in Lemmas~\ref{lem:L1}, \ref{lem:L2} shows that $\mathcal{L}_0^j[\omega]\in C_T\hat{H}^{\frac12}$ for $1\leq j\leq 10$.

We move on to proving \eqref{est:W0}.
The Bernstein inequality bounds the left-hand side by $\sup_N\| \psi_N\mathcal{W}_0[\omega]\|_{L^4_T\hat{H}^{\frac32}}$, and we start with this norm for most cases.
Concerning $\mathcal{L}^j_0$ with $1\leq j\leq 10$, we see from the proof of the estimate \eqref{est:Lj1} in Lemmas~\ref{lem:L1}, \ref{lem:L2} that
\[ \| \mathcal{L}^j_1[\omega ,\zeta ]\|_{\hat{H}^{\frac12}}\lesssim M^{-\frac12+}\| w\|_{H^{\frac12}}^4\| z\|_{H^{-\frac12}},\]
and hence%
\footnote{%
Recall that each of five terms in $\mathcal{L}^j_1[\omega ,\zeta]$ has been estimated separately.}%
\[ \| \mathcal{L}^j_0[\omega ]\|_{\hat{H}^{\frac32}}\leq \| \mathcal{L}^j_1[\omega ,\langle \xi\rangle \omega ]\|_{\hat{H}^{\frac12}}\lesssim M^{-\frac12+}\| w\|_{H^{\frac12}}^5.\]
Similarly, the proof of the estimates in Lemma~\ref{lem:N_1E} and in Lemma~\ref{lem:N_1CD} (see also Remark~\ref{rem:N_1CD}) actually show that
\[ \| \mathcal{N}^{\texttt{C2},\texttt{D},\texttt{E}}_1[\omega ,\zeta]\|_{\hat{H}^{\frac12}}+\| \mathcal{N}^{\texttt{C1}}_{1,2}[\omega ,\zeta]\|_{\hat{H}^{\frac12}}\lesssim \| w\|_{H^{\frac12}}\| w\|_{B^{0+}_{\infty,1}}\| z\|_{H^{-\frac12}}\]
and
\[ \sup_N\| \psi_N\mathcal{N}^{\texttt{C1}}_{1,1}[\omega ,\zeta]\|_{\hat{H}^{\frac12}}+\sup_N\| \psi_N\mathcal{N}^{\texttt{C1}}_{1,3}[\omega ,\zeta]\|_{\hat{H}^{\frac12}}\lesssim \| w\|_{H^{\frac12}}\| w\|_{B^{0}_{\infty,1}}\| z\|_{H^{-\frac12}},\]
from which we conclude
\[ \sup_N\| \psi_N\mathcal{N}^{\texttt{C},\texttt{D},\texttt{E}}_0[\omega]\|_{\hat{H}^{\frac32}}\lesssim \| w\|_{H^{\frac12}}^2\| w\|_{B^{0+}_{\infty,1}}.\]
It remains to treat $\mathcal{N}^{\texttt{B}}_0[\omega ]$; note that we cannot estimate this portion in $\hat{H}^{\frac32}$ due to logarithmic divergence in $K$ (see the footnote to the proof of Lemma~\ref{lem:N_1B}).
We use Lemma~\ref{lem:CM3} with $p=\infty$ and $q=2$, then for $\mathcal{N}^{\texttt{B1}}_0$ we have 
\begin{align*}
&\big\| S(t)\mathcal{F}^{-1}\big[ \mathcal{N}^{\texttt{B1}}_0[\omega ]\big] \big\| _{B^1_{\infty,1}} \\
&\lesssim \sum_{\begin{smallmatrix} N_1\sim N_3\gg N\sim N_2 \\ K\lesssim N_1\end{smallmatrix}}N\Big\| P_N\mathcal{F}^{-1}\Big[ \int_{\begin{smallmatrix} \xi=\xi_{123} \\ K\leq |\xi_{13}|<2K\end{smallmatrix}}\frac{\xi_3}{\xi_{13}\xi_{23}}\hat{w}_{N_1}(\xi_1)\hat{w}_{N_2}(\xi_2)\hat{\bar{w}}_{N_3}(\xi_3)\Big] \Big\| _{L^\infty}\\
&\lesssim \sum_{N_1\sim N_3}N_1^{1-}\| w_{N_1}\|_{L^2}\| w_{N_3}\|_{L^2}\sum _{K,N_2\lesssim N_1}N_2^{0+}\| w_{N_2}\|_{L^\infty} \lesssim \| w\|_{H^{\frac12}}^2\| w\|_{B^{0+}_{\infty,1}},
\end{align*}
and for $\mathcal{N}^{\texttt{B2}}_0$,
\begin{align*}
&\big\| S(t)\mathcal{F}^{-1}\big[ \mathcal{N}^{\texttt{B2}}_0[\omega ]\big] \big\| _{B^1_{\infty,1}} \\
&\lesssim \sum_{\begin{smallmatrix} N_1\sim N\gg N_2\sim N_3 \\ K\lesssim N_2\end{smallmatrix}}N\Big\| P_N\mathcal{F}^{-1}\Big[ \int_{\begin{smallmatrix} \xi=\xi_{123} \\ K\leq |\xi_{23}|<2K\end{smallmatrix}}\frac{\xi_3}{\xi_{13}\xi_{23}}\hat{w}_{N_1}(\xi_1)\hat{w}_{N_2}(\xi_2)\hat{\bar{w}}_{N_3}(\xi_3)\Big] \Big\| _{L^\infty}\\
&\lesssim \sum_{N_1\sim N}N_1\| w_{N_1}\|_{L^\infty}\sum _{K\lesssim N_2\sim N_3\ll N_1}\frac{N_3}{N_1}\| w_{N_2}\|_{L^2}\| w_{N_3}\|_{L^2} \lesssim \| w\|_{B^{0+}_{\infty,1}}\| w\|_{H^{\frac12}}^2.
\end{align*}
So far, we have proved
\begin{equation}\label{est:W0'}
\sup_NN\big\| P_NS(t)\mathcal{F}^{-1}\mathcal{W}_0[\omega (t)]\big\|_{L^\infty}\lesssim \| w(t)\|_{H^{\frac12}}^2\Big( 1+M^{-\frac12+}\| w(t)\|_{H^{\frac12}}^2\Big) \| w(t)\|_{H^{\frac12}\cap B^{0+}_{\infty,1}}
\end{equation}
for each $t$, and then \eqref{est:W0} follows from the H\"older inequality in $t$.

We next prove \eqref{est:W0-N}.
Recall that
\begin{align*}
S(t)\mathcal{F}^{-1}\mathcal{N}^{\texttt{A}}[\omega ]&=\sum_{N\sim N_1\sim N_2\sim N_3}P_N\Big\{ w_{N_1}w_{N_2}\partial_x\bar{w}_{N_3}\\
&\qquad -\mathcal{F}^{-1}\Big[ \frac{1}{2\pi}\int_{\begin{smallmatrix} \xi=\xi_{123} \\ |\xi_{13}|\wedge |\xi_{23}|<1 \end{smallmatrix}}i\xi_3\hat{w}_{N_1}(\xi_1)\hat{w}_{N_2}(\xi_2)\hat{\bar{w}}_{N_3}(\xi_3)\Big] \Big\}.
\end{align*}
The $H^{\frac12}$ norm of the second part is easily estimated by $\| w\|_{H^{\frac12}}^3$; see the argument in the proof of Lemma~\ref{lem:absconv}.
For the first part, it suffices to consider 
\[ \sum_{N\sim N_1\sim N_2\sim N_3}\Big\| P_N\Big[ w_{N_1}w_{N_2}\partial_x\bar{w}_{N_3}-P_{N_1}(w-W_0)P_{N_2}(w-W_0)P_{N_3}\partial_x(\overline{w-W_0})\Big] \Big\|_{H^{\frac12}}, \]
where $W_0:=S(t)\mathcal{F}^{-1}\mathcal{W}_0[\omega ]$.
By Bernstein and H\"older, the above is estimated by
\begin{align*}
&\sum_{N'_1\sim N'_2\sim N'_3}(N_1')^2\big( \| w_{N'_1}\|_{L^2}+\| P_{N_1'}W_0\|_{L^2}\big) \big( \| w_{N'_2}\|_{L^2}+\| P_{N_2'}W_0\|_{L^2}\big) \| P_{N'_3}W_0\|_{L^\infty} \\
&\quad \lesssim \big( \| w\|_{H^{\frac12}}+\| W_0\|_{H^{\frac12}}\big) ^2 \sup_{N}N\| P_{N}W_0\|_{L^\infty}.
\end{align*}
Hence, we obtain
\[ \Big\| \mathcal{N}^{\texttt{A}}[\omega ]-\mathcal{N}^{\texttt{A}}[\omega -\mathcal{W}_0[\omega]]\Big\|_{\hat{H}^{\frac12}}\lesssim \big( \| w\|_{H^{\frac12}}+\| W_0\|_{H^{\frac12}}\big) ^2 \Big( \| w\|_{H^{\frac12}}+\| W_0\|_{H^{\frac12}}+\sup_{N}N\| P_{N}W_0\|_{L^\infty}\Big) .\]
From Lemma~\ref{lem:N_0} and the estimate \eqref{est:Lj0} in Lemmas~\ref{lem:L1}, \ref{lem:L2}, we have
\[ \| W_0\|_{H^{\frac12}}\leq \| \mathcal{N}^{\neg \texttt{A}}_0[\omega ]\|_{\hat{H}^{\frac12}}+\sum _{j=1}^{10}\| \mathcal{L}^j_0[\omega ]\|_{\hat{H}^{\frac12}}\lesssim M^{-\frac12+}\| w\|_{H^{\frac12}}^3+M^{-1+}\| w\|_{H^{\frac12}}^5.\]
Using this estimate and \eqref{est:W0'}, we obtain
\begin{align*}
\Big\| \mathcal{N}^{\texttt{A}}[\omega ]-\mathcal{N}^{\texttt{A}}[\omega -\mathcal{W}_0[\omega]]\Big\|_{\hat{H}^{\frac12}}&\lesssim \| w\|_{H^{\frac12}}^3\Big( 1+M^{-\frac12+}\| w\|_{H^{\frac12}}^2\Big) ^5\Big( 1+\| w\|_{H^{\frac12}}\| w\|_{H^{\frac12}\cap B^{0+}_{\infty,1}}\Big) .
\end{align*}
The claim \eqref{est:W0-N} follows from H\"older in $t$.
\end{proof}

\begin{lem}\label{lem:W1}
Let $T\in (0,1]$.
For any $\omega$ with $w=S(t)\check{\omega}\in L^\infty_TH^{\frac12}\cap L^4_TB^{0+}_{\infty,1}$, we have $\mathcal{W}_1[\omega]\in L^2_T\hat{H}^{\frac12}$ and 
\[ \| \mathcal{W}_1[\omega] \|_{L^1_T\hat{H}^{\frac12}}\lesssim \mathcal{P}_1\Big( M^{-\frac12+}\| w\|_{S_T}^2,\,T^{\frac12}(1+\| w\|_{S_T}^2)\| w\|_{S_T}^2,\,TM\Big) \| w\|_{S_T} , \]
where $\mathcal{P}_1(x,y,z):=(1+x^5)y+xz$.
\end{lem}
\begin{proof}
We see $\mathcal{R}[\omega]\in L^4_T\hat{H}^{\frac12}$ from Lemma~\ref{lem:R}; $\mathcal{N}_R[\omega]\in L^\infty_T\hat{H}^{\frac12}$ from Lemma~\ref{lem:N_R}; $\mathcal{N}_1[\omega,\mathcal{R}[\omega]]\in L^4_T\hat{H}^{\frac12}$ from Lemmas~\ref{lem:N_0} and \ref{lem:R}; $\tilde{\mathcal{N}}_1[\omega]\in L^2_T\hat{H}^{\frac12}$ from Lemmas~\ref{lem:N_1E}, \ref{lem:N_1CD}, \ref{lem:N_1B} and Proposition~\ref{prop:N-weak}; $\mathcal{L}_R^j[\omega],\mathcal{L}_1^j[\omega,\mathcal{R}[\omega]]\in L^4_T\hat{H}^{\frac12}$ and $\mathcal{L}_1^j[\omega ,\mathcal{N}[\omega]]\in L^2_T\hat{H}^{\frac12}$ from Lemmas~\ref{lem:L1}, \ref{lem:L2}, \ref{lem:R} and Proposition~\ref{prop:N-weak}; $\mathcal{N}^{\texttt{A}}[\omega]-\mathcal{N}^{\texttt{A}}[\omega-\mathcal{W}_0[\omega]]\in L^4_T\hat{H}^{\frac12}$ from Lemma~\ref{lem:W0}.
The claimed estimate also follows from the estimates given in these lemmas.
\end{proof}

We conclude the first stage of our NFR argument with the following:
\begin{cor}\label{cor:eq-varpi}
Let $w\in L^\infty_TH^{\frac12}\cap L^4_TB^{0+}_{\infty,1}$ be a (distributional) solution of \eqref{eq:w2}, and let $\omega (t):=\mathcal{F}S(-t)w(t)$.
Then, $\varpi:=\omega -\mathcal{W}_0[\omega]$ belongs to $L^\infty_T\hat{H}^{\frac12}\cap C_T\hat{H}^{\frac12-}$ and satisfies the integral equation \eqref{eq:varpi} in the sense of $C_T\hat{H}^{-\frac12}$.
Moreover, $\partial_t\varpi$ belongs to $L^2_T\hat{H}^{-\frac12}$ and satisfies the differential equation
\begin{equation}\label{eq:varpi-d}
\partial_t\varpi=\mathcal{N}^{\texttt{A}}[\varpi] +\mathcal{W}_1[\omega]
\end{equation}
in the sense of $L^2_T\hat{H}^{-\frac12}$.
\end{cor}
\begin{proof}
As we have seen in Subsection~\ref{subsec:just}, $\omega$ belongs to $L^\infty_T\hat{H}^{\frac12}\cap C_T\hat{H}^{\frac12-}$, and thus $\varpi$ belongs to the same space by Lemma~\ref{lem:W0}.
It is not hard to show the estimate
\[ \| \mathcal{N}^{\texttt{A}}[\varpi]\|_{\hat{H}^{-\frac12}}\lesssim \| \varpi\|_{\hat{H}^{\frac12}}^3 \]
(this is the special case $J=1$ of Lemma~\ref{lem:NFR-weak} below), which implies $\mathcal{N}^{\texttt{A}}[\varpi]\in L^\infty_T\hat{H}^{-\frac12}$ for any $\varpi\in L^\infty_T\hat{H}^{\frac12}$, and thus $\int_0^t\mathcal{N}^{\texttt{A}}[\varpi] \in W^{1,\infty}_T\hat{H}^{-\frac12}\subset C_T\hat{H}^{-\frac12}$.
From Lemma~\ref{lem:W1}, we have $\int_0^t \mathcal{W}_1[\omega]\in W^{1,2}_T\hat{H}^{\frac12}\subset C_T\hat{H}^{\frac12}$.
Hence, the equation \eqref{eq:varpi} holds pointwise in $t\in [-T,T]$ as $\hat{H}^{-\frac12}$-valued functions, while the equation \eqref{eq:varpi-d} holds in $L^2_T\hat{H}^{-\frac12}$.
\end{proof}


\section{Normal form reduction: II. Infinite reductions}
\label{sec:NFR2}

In the second stage, we implement infinite NFR iteration scheme on the equation \eqref{eq:varpi} for $\varpi$ to derive a limit equation which does not have derivative losses.
Using it, we shall prove Theorem~\ref{thm:main} in Subsection~\ref{subsec:proof}.


\subsection{Description of general steps via the tree notation}

We recall some notation for the infinite NFR argument from \cite{K-all}.
\begin{defn}[Trees]
For $J\in \mathbb{N}$, a tree $\mathcal{T}$ of the $J$-th generation means a partially ordered set (with a partial order $\succeq$) of $3J+1$ elements satisfying all the following conditions.
\begin{enumerate}
\item[$\mathrm{(i)}$] There is the unique element $r\in \mathcal{T}$ called the ``root'' such that $r\succeq a$ for all $a\in \mathcal{T}$.
\item[$\mathrm{(ii)}$] We say $a\in \mathcal{T}$ is a ``parent'' of $b\in \mathcal{T}$ and $b$ is a ``child'' of $a$ if $a\neq b$, $a\succeq b$, and also $a\succeq c\succeq b$ implies $c=a$ or $c=b$.
(We write $\mathcal{T}^0$ to denote the subset of $\mathcal{T}$ consisting of all elements that have a child, and $\mathcal{T}^\infty :=\mathcal{T}\setminus \mathcal{T}^0$.)
Then, each element of $\mathcal{T}\setminus \{ \text{root}\}$ has exactly one parent, whereas each element of $\mathcal{T}^0$ has exactly three children.
(This condition implies that $\# \mathcal{T}^0=J$ and $\# \mathcal{T}^\infty =2J+1$.)
\item[$\mathrm{(iii)}$] The elements of $\mathcal{T}^0$ are numbered from $1$ to $J$ so that $a^{j_1}\succeq a^{j_2}$ implies $j_1\le j_2$, where we write the $j$-th parent (the $j$-th element of $\mathcal{T}^0$) as $a^j$. 
We also distinguish the three children of $a^j$ and write them as $a^j_1,a^j_2,a^j_3$.
\end{enumerate}
We write $\mathfrak{T}(J)$ to denote the set of all trees of the $J$-th generation.
\end{defn} 

We note that the above set of trees $\mathfrak{T}(J)$ also has an inductive definition:
\begin{itemize}
\item Let $\mathfrak{T}(1)$ be the set of one tree consisting of the root and its three children.
\item With $\mathfrak{T}(J-1)$ given, define $\mathfrak{T}(J)$ as the set of all trees that can be obtained from some $\mathcal{T}\in \mathfrak{T}(J-1)$ by choosing an element $a\in \mathcal{T}^\infty$ as the $J$-th parent and adding its three children.
\end{itemize}
From this definition, we easily see that $\# \mathfrak{T}(J)=\prod _{j=1}^J(2j-1)$. 
\begin{defn}[Index functions]
Given $J\in \mathbb{N}$ and $\mathcal{T}\in \mathfrak{T}(J)$, an index function $\bm{\xi}=(\xi_a)_{a\in \mathcal{T}}$ on $\mathcal{T}$ is a map from $\mathcal{T}$ to $\hat{\mathcal{M}}$ ($=\mathbb{R}$ or $\mathbb{Z}$ for $\mathcal{M}=\mathbb{R}$ or $\mathbb{T}$, respectively) such that
\[ \xi_{a^j}=\xi_{a^j_1}+\xi_{a^j_2}+\xi_{a^j_3} \qquad (j=1,\dots ,J). \]
We write $\Xi (\mathcal{T})$ to denote the set of all index functions on $\mathcal{T}$, and for $\xi \in \hat{\mathcal{M}}$, denote by $\Xi_\xi(\mathcal{T})$ all index functions with the value $\xi$ at the root.
We divide $\mathcal{T}$ into two disjoint subsets $\mathcal{T}_1$ and $\mathcal{T}_2$ by the following rule:
\[ a^1\in \mathcal{T}_1,\qquad \left\{ \begin{aligned} a^j\in \mathcal{T}_1\quad &\Rightarrow \quad a^j_1,a^j_2\in \mathcal{T}_1~~\text{and}~~a^j_3\in \mathcal{T}_2\\
a^j\in \mathcal{T}_2\quad &\Rightarrow \quad a^j_1,a^j_2\in \mathcal{T}_2~~\text{and}~~a^j_3\in \mathcal{T}_1\end{aligned} \right. \qquad (j=1,\dots ,J). \]
Write $\mathcal{T}^0_{m}:=\mathcal{T}^0\cap \mathcal{T}_m$ and $\mathcal{T}^\infty_m:=\mathcal{T}^\infty\cap \mathcal{T}_m$ for $m=1,2$.
Given $\bm{\xi}\in \Xi (\mathcal{T})$ and for $j=1,\dots ,J$, we define
\[ \phi _j:=\left\{ \begin{alignedat}{2} &\xi_{a^j}^2-\xi_{a^j_1}^2-\xi_{a^j_2}^2+\xi_{a^j_3}^2 & &\text{if $a^j\in \mathcal{T}_1^0$,} \\
-\big( &\xi_{a^j}^2-\xi_{a^j_1}^2-\xi_{a^j_2}^2+\xi_{a^j_3}^2\big) &\quad &\text{if $a^j\in \mathcal{T}_2^0$,} \end{alignedat}\right. \quad\qquad \tilde{\phi}_j:=\sum _{k=1}^j\phi_k ,\]
and 
\[ \Psi_j:=\mathbf{1}_{|\xi_{a^j_1a^j_3}|\wedge |\xi_{a^j_2a^j_3}|\geq 1}\sum_{\begin{smallmatrix} N,N_1,N_2,N_3\geq 1 \\ N\sim N_1\sim N_2\sim N_3\end{smallmatrix}}\psi_N(\xi_{a^j})\psi_{N_1}(\xi_{a^j_1})\psi_{N_2}(\xi_{a^j_2})\psi_{N_3}(\xi_{a^j_3}). \]
\end{defn}
By the above definition, an index function $\bm{\xi}\in \Xi(\mathcal{T})$ is determined uniquely by its values on $\mathcal{T}^\infty$.
Moreover, we see that $\# \mathcal{T}^\infty_1=J+1$, $\# \mathcal{T}^\infty_2=J$, and 
\[ \xi =\sum_{a\in \mathcal{T}^\infty}\xi_a,\qquad \tilde{\phi}_J=\xi^2-\sum _{a\in \mathcal{T}^\infty_1}\xi_a^2+\sum _{a'\in \mathcal{T}^\infty_2}\xi_{a'}^2\]
for any $\mathcal{T}\in \mathfrak{T}(J)$, $\bm{\xi}\in \Xi _\xi (\mathcal{T})$.
We also note that 
\[ |\Psi_j|\leq 1,\qquad {\rm supp}\,\Psi_j\subset \{ \bm{\xi}\in \Xi(\mathcal{T}): \langle \xi_{a^j}\rangle \sim \langle \xi_{a^j_1}\rangle \sim \langle \xi_{a^j_2}\rangle \sim \langle \xi_{a^j_3}\rangle \} .\]

Using the tree notation, we can explicitly write down the multilinear terms arising during the NFR argument (up to a harmless multiplicative constant), as we describe below.
In the following we \emph{formally} perform the computations, and the justification will be addressed in Subsection~\ref{subsec:just-infinite}.

The initial cubic term 
\[ \mathcal{N}^{\texttt{A}}[\varpi](\xi)=\sum_{N\sim N_1\sim N_2\sim N_3}\int_{\begin{smallmatrix} \xi=\xi_{123} \\ |\xi_{13}|\wedge |\xi_{23}|\geq 1\end{smallmatrix}}e^{it(\xi^2-\xi_1^2-\xi_2^2+\xi_3^2)}\psi_N(\xi)\xi_3\varpi_{N_1}(\xi_1)\varpi_{N_2}(\xi_2)\varpi_{N_3}^*(\xi_3) \]
in the equation \eqref{eq:varpi} can be written as
\begin{align*}
\mathcal{N}^{\texttt{A}}[\varpi]&=\sum_{\mathcal{T}\in \mathfrak{T}(1)}\mathcal{N}^{(1;\mathcal{T})}[\varpi],\\
\mathcal{N}^{(1;\mathcal{T})}[\varpi](\xi)&:=\int_{\bm{\xi}\in \Xi_\xi(\mathcal{T})}e^{it\tilde{\phi}_1}\Psi_1\xi_{a^1_3}\prod_{a\in \mathcal{T}^\infty_1}\varpi(\xi_a)\prod_{a'\in \mathcal{T}^\infty_2}\varpi^*(\xi_{a'}).
\end{align*}
By NFR, as we did for the other type of cubic terms, we decompose it as $\mathcal{N}^{\texttt{A}}=\mathcal{N}^{\texttt{A}}_R+\mathcal{N}^{\texttt{A}}_{N\!R}$ and apply an integration by parts as $\mathcal{N}^{\texttt{A}}_{N\!R}[\varpi]=\partial_t\mathcal{N}^{\texttt{A}}_0[\varpi]+\mathcal{N}^{\texttt{A}}_1[\varpi,\partial_t\varpi]$.
These terms can be written as follows:
\begin{gather*}
\mathcal{N}^{\texttt{A}}_R[\varpi]=\sum_{\mathcal{T}\in \mathfrak{T}(1)}\mathcal{N}^{(1;\mathcal{T})}_R[\varpi],\qquad \mathcal{N}^{\texttt{A}}_0[\varpi]=\sum_{\mathcal{T}\in \mathfrak{T}(1)}\mathcal{N}^{(1;\mathcal{T})}_0[\varpi],\\
\mathcal{N}^{\texttt{A}}_1[\varpi,\zeta ]=\sum_{\mathcal{T}\in \mathfrak{T}(1)}\sum_{a^*\in \mathcal{T}^\infty}\mathcal{N}^{(1;\mathcal{T},a^*)}_1[\varpi,\zeta],
\end{gather*}
\begin{align*}
\mathcal{N}^{(1;\mathcal{T})}_R[\varpi](\xi)&:=\int_{\bm{\xi}\in \Xi_\xi(\mathcal{T})}e^{it\tilde{\phi}_1}\Psi_1\xi_{a^1_3}\mathbf{1}_{|\tilde{\phi}_1|\leq M}\prod_{a\in \mathcal{T}^\infty_1}\varpi(\xi_a)\prod_{a'\in \mathcal{T}^\infty_2}\varpi^*(\xi_{a'}),\\
\mathcal{N}^{(1;\mathcal{T})}_0[\varpi](\xi)&:=\int_{\bm{\xi}\in \Xi_\xi(\mathcal{T})}e^{it\tilde{\phi}_1}\Psi_1\xi_{a^1_3}\frac{\mathbf{1}_{|\tilde{\phi}_1|>M}}{\tilde{\phi}_1}\prod_{a\in \mathcal{T}^\infty_1}\varpi(\xi_a)\prod_{a'\in \mathcal{T}^\infty_2}\varpi^*(\xi_{a'}),\\
\mathcal{N}^{(1;\mathcal{T},a^*)}_1[\varpi,\zeta ](\xi)&:=\int_{\bm{\xi}\in \Xi_\xi(\mathcal{T})}e^{it\tilde{\phi}_1}\Psi_1\xi_{a^1_3}\frac{\mathbf{1}_{|\tilde{\phi}_1|>M}}{\tilde{\phi}_1}\\
&\qquad \times \left\{ \begin{alignedat}{2}
\zeta(\xi_{a^*})\prod_{a\in \mathcal{T}^\infty_1\setminus \{ a^*\}}\varpi(\xi_a)\prod_{a'\in \mathcal{T}^\infty_2}\varpi^*(\xi_{a'}) &\qquad &(\text{if $a^*\in \mathcal{T}^\infty_1$}),\\
\zeta^*(\xi_{a^*})\prod_{a\in \mathcal{T}^\infty_1}\varpi(\xi_a)\prod_{a'\in \mathcal{T}^\infty_2\setminus \{ a^*\}}\varpi^*(\xi_{a'}) &\qquad &(\text{if $a^*\in \mathcal{T}^\infty_2$}).
\end{alignedat}\right.
\end{align*}
Inserting $\partial_t\varpi =\mathcal{W}_1[\omega] +\mathcal{N}^{\texttt{A}}[\varpi]$, the quintic term $\mathcal{N}_1^{\texttt{A}}[\varpi,\mathcal{N}^{\texttt{A}}[\varpi]]$ appears.
Recall that $\mathfrak{T}(2)=\{ \tilde{\mathcal{T}}(a^*):\mathcal{T}\in \mathfrak{T}(1),a^*\in \mathcal{T}^\infty\}$, where $\tilde{\mathcal{T}}(a^*)$ denotes the tree obtained by assigning $a^*\in \mathcal{T}^\infty$ to the second parent and adding its three children.
Then, we see that
\begin{align*}
\mathcal{N}_1^{\texttt{A}}[\varpi,\mathcal{N}^{\texttt{A}}[\varpi]](\xi) 
&=\sum_{\mathcal{T}\in \mathfrak{T}(1)} \sum_{a^*\in \mathcal{T}^\infty}\mathcal{N}_1^{(1;\mathcal{T},a^*)}\bigg[ \varpi, \sum_{\mathcal{T}'\in \mathfrak{T}(1)}\mathcal{N}^{(1;\mathcal{T}')}[\varpi]\bigg] \\
&=\sum_{\tilde{\mathcal{T}}\in \mathfrak{T}(2)}\int_{\bm{\xi}\in \Xi_\xi(\tilde{\mathcal{T}})}e^{it\tilde{\phi}_2}\Psi_1\xi_{a^1_3}\frac{\mathbf{1}_{|\tilde{\phi}_1|>M}}{\tilde{\phi}_1}\Psi_2\xi_{a^2_3}\prod_{a\in \tilde{\mathcal{T}}^\infty_1}\varpi(\xi_a)\prod_{a'\in \tilde{\mathcal{T}}^\infty_2}\varpi^*(\xi_{a'}).
\end{align*}
We write the right-hand side as $\sum_{\tilde{\mathcal{T}}\in \mathfrak{T}(2)}\mathcal{N}^{(2;\tilde{\mathcal{T}})}[\varpi]$ and obtain the equation of the second generation:
\begin{align*}
\varpi (t)&=\varpi(0)+\int_0^t\mathcal{W}_1[\omega ] \\
&\quad +\sum_{\mathcal{T}\in \mathfrak{T}(1)}\Big\{ \int_0^t\mathcal{N}^{(1;\mathcal{T})}_R[\varpi ]+\mathcal{N}^{(1;\mathcal{T})}_0[\varpi] \Big|_0^t +\sum _{a^*\in \mathcal{T}^\infty}\int_0^t \mathcal{N}^{(1;\mathcal{T},a^*)}_1[\varpi,\mathcal{W}_1[\omega ]] \Big\} \\
&\quad +\sum_{\tilde{\mathcal{T}}\in \mathfrak{T}(2)}\int_0^t \mathcal{N}^{(2;\tilde{\mathcal{T}})}[\varpi].
\end{align*}
For notational simplicity, from now on we do not distinguish $\varpi$ and $\varpi^*$; in fact, the difference will play no role in the following argument.

In a general step (i.e., in the equation of the $J$-th generation obtained after $(J-1)$-times applications of NFR), we have the last term 
\[ \sum_{\mathcal{T}\in \mathfrak{T}(J)}\int_0^t\mathcal{N}^{(J;\mathcal{T})}[\varpi],\]
where for $\mathcal{T}\in \mathfrak{T}(J)$
\begin{align*}
\mathcal{N}^{(J;\mathcal{T})}[\varpi ](\xi ):=\int_{\bm{\xi}\in \Xi_\xi(\mathcal{T})} e^{it\tilde{\phi}_J}\prod_{j=1}^J\Psi_j\xi_{a^j_3}\cdot \prod_{j=1}^{J-1}\frac{\mathbf{1}_{|\tilde{\phi}_j|>M_j}}{\tilde{\phi}_j}\prod_{a\in \mathcal{T}^\infty}\varpi(\xi_a).
\end{align*}
The $\bm{\xi}$-integral in the above expression is an integral (or sum) on a $2J$-dimensional hyperplane in $\hat{\mathcal{M}}^{3J+1}$.
We choose a different threshold $M_j>1$ for the ``resonant'' part at each step; in fact, we will take $M_j:=10^{j-1}M$.
We then carry out the $J$-th NFR as 
\[ \mathcal{N}^{(J;\mathcal{T})}[\varpi]=\mathcal{N}^{(J;\mathcal{T})}_R[\varpi] +\partial_t\mathcal{N}^{(J;\mathcal{T})}_0[\varpi]+\sum _{a^*\in \mathcal{T}^\infty}\mathcal{N}^{(J;\mathcal{T},a^*)}_1[\varpi,\partial_t\varpi]\qquad (\mathcal{T}\in \mathfrak{T}(J)),\]
where
\begin{align*}
\mathcal{N}^{(J;\mathcal{T})}_R[\varpi ](\xi )&:=\int_{\bm{\xi}\in \Xi_\xi(\mathcal{T})} e^{it\tilde{\phi}_J}\prod_{j=1}^J\Psi_j\xi_{a^j_3}\cdot \prod_{j=1}^{J-1}\frac{\mathbf{1}_{|\tilde{\phi}_j|>M_j}}{\tilde{\phi}_j}\textbf{1}_{|\tilde{\phi}_{J}|\leq M_{J}}\prod_{a\in \mathcal{T}^\infty}\varpi(\xi_a),\\
\mathcal{N}^{(J;\mathcal{T})}_0[\varpi ](\xi )&:=\int_{\bm{\xi}\in \Xi_\xi(\mathcal{T})} e^{it\tilde{\phi}_J}\prod_{j=1}^J\frac{\Psi_j\xi_{a^j_3}\mathbf{1}_{|\tilde{\phi}_j|>M_j}}{\tilde{\phi}_j}\prod_{a\in \mathcal{T}^\infty}\varpi(\xi_a),\\
\mathcal{N}^{(J;\mathcal{T},a^*)}_1[\varpi ,\zeta ](\xi )&:=\int_{\bm{\xi}\in \Xi_\xi(\mathcal{T})} e^{it\tilde{\phi}_J}\prod_{j=1}^J\frac{\Psi_j\xi_{a^j_3}\mathbf{1}_{|\tilde{\phi}_j|>M_j}}{\tilde{\phi}_j}\zeta(\xi_{a^*})\prod_{a\in \mathcal{T}^\infty\setminus \{ a^*\}}\varpi(\xi_a).
\end{align*}
Inserting the equation $\partial_t\varpi =\mathcal{W}_1[\omega] +\mathcal{N}^{\texttt{A}}[\varpi]$, we have 
\[ \mathcal{N}^{(J;\mathcal{T},a^*)}_1[\varpi,\partial_t\varpi] =\mathcal{N}^{(J;\mathcal{T},a^*)}_1[\varpi,\mathcal{W}_1[\omega ]]+\mathcal{N}^{(J+1;\tilde{\mathcal{T}}(a^*))}[\varpi],\]
where $\tilde{\mathcal{T}}(a^*)\in \mathfrak{T}(J+1)$ is the tree obtained by assigning $a^*\in \mathcal{T}^\infty$ to the $(J+1)$-th parent and adding its three children.
Since $\mathfrak{T}(J+1)=\{ \tilde{\mathcal{T}}(a^*):\mathcal{T}\in \mathfrak{T}(J),\,a^*\in \mathcal{T}^\infty\}$, the equation of the $(J+1)$-th generation is given by
\begin{equation}\label{eq:omega-J}
\begin{aligned}
\varpi (t)&=\varpi(0)+\int_0^t\mathcal{W}_1[\omega ]\\
&\quad +\sum_{j=1}^J\sum_{\mathcal{T}\in \mathfrak{T}(j)}\Big\{ \int_0^t\mathcal{N}^{(j;\mathcal{T})}_R[\varpi ]+\mathcal{N}^{(j;\mathcal{T})}_0[\varpi] \Big|_0^t +\sum _{a^*\in \mathcal{T}^\infty}\int_0^t \mathcal{N}^{(j;\mathcal{T},a^*)}_1[\varpi,\mathcal{W}_1[\omega ]] \Big\} \\
&\quad +\sum_{\tilde{\mathcal{T}}\in \mathfrak{T}(J+1)}\int_0^t \mathcal{N}^{(J+1;\tilde{\mathcal{T}})}[\varpi].
\end{aligned}
\end{equation}
We will see that the terms in the second line of \eqref{eq:omega-J} are estimated in $\hat{H}^{\frac12}$, while the last term $\mathcal{N}^{(J+1;\tilde{\mathcal{T}})}[\varpi]$ is still out of the $\hat{H}^{\frac12}$ control.
However, as $J\to \infty$ we \emph{formally} obtain the limit equation
\begin{equation}\label{eq:limit}
\begin{aligned}
\varpi(t)&=\varpi(0)+\int_0^t\mathcal{W}_1[\omega]\\
&\quad +\sum_{j=1}^\infty \sum_{\mathcal{T}\in \mathfrak{T}(j)}\Big\{ \int_0^t\mathcal{N}^{(j;\mathcal{T})}_R[\varpi]+\mathcal{N}^{(j;\mathcal{T})}_0[\varpi] \Big|_0^t +\sum _{a^*\in \mathcal{T}^\infty}\int_0^t \mathcal{N}^{(j;\mathcal{T},a^*)}_1[\varpi,\mathcal{W}_1[\omega]] \Big\} ,
\end{aligned}
\end{equation}
in which the last term disappears.
In particular, the limit equation \eqref{eq:limit} has no derivative losses.


\subsection{Multilinear estimates}

We give multilinear estimates of the terms appearing in \eqref{eq:omega-J}.
\begin{lem}\label{lem:NFR-NR}
There exists $C>0$ independent of $J,\mathcal{T},M$ such that for any $T\in (0,1]$, $J\in \mathbb{N}$, and $\mathcal{T}\in \mathfrak{T}
(J)$, we have
\begin{align*}
\Big( \prod_{j=1}^J3^{j-1}\Big) \| \mathcal{N}^{(J;\mathcal{T})}_R[\varpi ]\|_{L^\infty_T\hat{H}^{\frac12}}&\leq M(CM^{-\frac12+}\| \varpi \|_{L^\infty_T\hat{H}^{\frac12}}^{2})^J\| \varpi \|_{L^\infty_T\hat{H}^{\frac12}},\\
\Big( \prod_{j=1}^J3^{j-1}\Big) \| \mathcal{N}^{(J;\mathcal{T})}_0[\varpi ]\|_{L^\infty_T\hat{H}^{\frac12}}&\leq (CM^{-\frac12+}\| \varpi \|_{L^\infty_T\hat{H}^{\frac12}}^{2})^J\| \varpi \|_{L^\infty_T\hat{H}^{\frac12}}.
\end{align*}
\end{lem}
\begin{rem}
The factor $\prod\limits_{j=1}^J3^{j-1}$ will be used to sum up in $\mathcal{T}\in \mathfrak{T}(J)$; in fact, we see $\# \mathfrak{T}(J)=\prod\limits_{j=1}^{J}(2j-1)\leq \prod\limits_{j=1}^J3^{j-1}$.
Note that $\# \mathfrak{T}(J)$ is of $O(C^{J\log J})$ but not of $O(C^J)$.
\end{rem}

\begin{proof}[Proof of Lemma~\ref{lem:NFR-NR}]
The proof is a generalization of the argument for $J=1$ given in Lemmas~\ref{lem:N_R}, \ref{lem:N_0}.
Applying the Cauchy-Schwarz inequality to the $\bm{\xi}$-integral, we bound $\| \mathcal{N}^{(J;\mathcal{T})}_R[\varpi]\|_{\hat{H}^{\frac12}}$ and $\| \mathcal{N}^{(J;\mathcal{T})}_0[\varpi]\|_{\hat{H}^{\frac12}}$ by
\[ \| \varpi \|_{\hat{H}^{\frac12}}^{2J+1}\times \left\{ \begin{aligned}
&\sup _{\xi\in \hat{\mathcal{M}}}\Big( \int_{\bm{\xi}\in \Xi_\xi(\mathcal{T})}\frac{\langle \xi_{a^1}\rangle}{\prod\limits_{a\in \mathcal{T}^\infty}\langle \xi_a\rangle}\prod_{j=1}^{J}\Psi_j^2|\xi_{a^j_3}|^2\cdot \prod_{j=1}^{J-1}\frac{\mathbf{1}_{|\tilde{\phi}_j|>M_j}}{|\tilde{\phi}_j|^2}\cdot \mathbf{1}_{|\tilde{\phi}_{J}|\leq M_{J}}\Big) ^{\frac12},\\
&\sup _{\xi\in \hat{\mathcal{M}}}\Big( \int_{\bm{\xi}\in \Xi_\xi(\mathcal{T})}\frac{\langle \xi_{a^1}\rangle}{\prod\limits_{a\in \mathcal{T}^\infty}\langle \xi_a\rangle}\prod_{j=1}^{J}\frac{\Psi_j^2|\xi_{a^j_3}|^2\mathbf{1}_{|\tilde{\phi}_j|>M_j}}{|\tilde{\phi}_j|^2}\Big) ^{\frac12},
\end{aligned}\right. \]
respectively.
To estimate the inside of $\sup_\xi$, we first notice that 
\[ \frac{\langle \xi_{a^j}\rangle |\xi_{a^j_3}|^2}{\langle \xi_{a^j_1}\rangle \langle \xi_{a^j_2}\rangle \langle \xi_{a^j_3}\rangle}\Psi_j\leq C,\qquad j=1,\dots ,J,\]
which shows
\begin{equation}\label{est:weight-J}
\frac{\langle \xi_{a^1}\rangle}{\prod\limits_{a\in \mathcal{T}^\infty}\langle \xi_a\rangle}\prod_{j=1}^{J}\Psi_j|\xi_{a^j_3}|^2
=\prod_{j=1}^J \frac{\langle \xi_{a^j}\rangle |\xi_{a^j_3}|^2}{\langle \xi_{a^j_1}\rangle \langle \xi_{a^j_2}\rangle \langle \xi_{a^j_3}\rangle}\Psi_j \leq C^{J}.
\end{equation}
Next, we prepare the following lemma:
\begin{lem}\label{lem:combi}
Let $\mathcal{M}=\mathbb{R}$ or $\mathbb{T}$.
For any $\varepsilon,\varepsilon'>0$, there exists $C>0$ such that
\[ \int_{\begin{smallmatrix} \xi=\xi_{123}\\ |\xi_{13}|\wedge |\xi_{23}|\geq 1\end{smallmatrix}}\frac{1}{\langle \phi_*+2\xi_{13}\xi_{23}\rangle^{1+\varepsilon}}\leq C\langle \phi_*\rangle ^{\varepsilon'}\]
for any $\xi,\phi_*\in \hat{\mathcal{M}}$.
\end{lem}
\begin{proof}
We may assume $|\phi_*|\gg 1$ and restrict the domain of the integral to $\{ (\xi_1,\xi_2):|\phi_*+2(\xi-\xi_1)(\xi-\xi_2)|\ll |(\xi-\xi_1)(\xi-\xi_2)|\sim |\phi_*|\}$.
In the non-periodic case, we can estimate the integral as
\[ \iint_{\begin{smallmatrix} |\eta_1|, |\eta_2|\geq 1\\ |\phi_*+2\eta_1\eta_2|\ll |\eta_1\eta_2|\sim |\phi_*|\end{smallmatrix}}\frac{d\eta_1\,d\eta_2}{\langle \phi_*+2\eta_1\eta_2\rangle^{1+}}\leq \int_{1\leq |\eta_1|\lesssim |\phi_*|}\frac{1}{2|\eta_1|}\int_{\mathbb{R}}\frac{d\eta_2'}{\langle \eta_2'\rangle^{1+}}\,d\eta_1 \lesssim |\phi_*|^{0+},\]
as claimed.
In the periodic case, the integral (sum) is bounded by 
\[ \sum_{\nu \in [-\frac{|\phi_*|}{2},\frac{|\phi_*|}{2}]\cap \mathbb{Z}}\frac{\# \{ (\xi_1,\xi_2):2(\xi-\xi_1)(\xi-\xi_2)=-\phi_*+\nu\}}{\langle \nu \rangle ^{1+}} ,\]
and the estimate follows from the divisor bound
\[ \# \{ (\xi_1,\xi_2):2(\xi-\xi_1)(\xi-\xi_2)=-\phi_*+\nu\} \lesssim |{-}\phi_*+\nu |^{0+}\sim |\phi_*|^{0+}. \qedhere \]
\end{proof}

By Lemma~\ref{lem:combi}, we have
\begin{align*}
\int_{\begin{smallmatrix} \xi=\xi_{123}\\ |\xi_{13}|\wedge |\xi_{23}|\geq 1\end{smallmatrix}}\frac{\mathbf{1}_{|\phi_*+2\xi_{13}\xi_{23}|>M}}{|\phi_*+2\xi_{13}\xi_{23}|^{1+\alpha}}&\lesssim _\alpha M^{-\alpha+}\langle \phi_*\rangle ^{0+}\qquad (\alpha >0)
\intertext{and}
\int_{\begin{smallmatrix} \xi=\xi_{123}\\ |\xi_{13}|\wedge |\xi_{23}|\geq 1\end{smallmatrix}}\mathbf{1}_{|\phi_*+2\xi_{13}\xi_{23}|\leq M}&\lesssim M^{1+}\langle \phi_*\rangle ^{0+},
\end{align*}
where the implicit constants can be chosen uniformly in $\xi,\phi_*\in \hat{\mathcal{M}}$ and $M>1$.
Using these estimates (with $\alpha={1-}$ or $\alpha =1$) repeatedly, we obtain that
\[ \left\{ \begin{aligned}
\sup _{\xi\in \hat{\mathcal{M}}}\Big( \int_{\bm{\xi}\in \Xi_\xi(\mathcal{T})}\prod_{j=1}^{J-1}\frac{\mathbf{1}^{N\!R}_{|\tilde{\phi}_j|>M_j}}{|\tilde{\phi}_j|^2}\cdot \mathbf{1}^{N\!R}_{|\tilde{\phi}_{J}|\leq M_{J}}\Big) ^{\frac12} &\leq C^{J}\prod_{j=1}^{J-1}M_j^{-\frac12+}\cdot M_{J}^{\frac12+}=C^{J}M_J\prod_{j=1}^{J}M_j^{-\frac12+},\\
\sup _{\xi\in \hat{\mathcal{M}}}\Big( \int_{\bm{\xi}\in \Xi_\xi(\mathcal{T})}\prod_{j=1}^{J}\frac{\mathbf{1}^{N\!R}_{|\tilde{\phi}_j|>M_j}}{|\tilde{\phi}_j|^2}\Big) ^{\frac12} &\leq C^{J}\prod_{j=1}^{J}M_j^{-\frac12+},
\end{aligned}\right. \]
where we write $\mathbf{1}^{N\!R}$ to indicate the restriction to $\{ |\xi_{a^j_1a^j_3}|\wedge |\xi_{a^j_2a^j_3}|\geq 1\}$.
So far, we have obtained the bounds
\begin{align*}
\| \mathcal{N}^{(J;\mathcal{T})}_R[\varpi]\|_{\hat{H}^{\frac12}}&\leq C^{J}M_J\Big( \prod_{j=1}^{J}M_j^{-\frac12+}\Big) \| \varpi\|_{\hat{H}^{\frac12}}^{2J+1},\\
\| \mathcal{N}^{(J;\mathcal{T})}_0[\varpi]\|_{\hat{H}^{\frac12}}&\leq C^{J}\Big( \prod_{j=1}^{J}M_j^{-\frac12+}\Big) \| \varpi\|_{\hat{H}^{\frac12}}^{2J+1}. 
\end{align*}
Now, we insert $M_j=10^{j-1}M$ and use $10^{-\frac12+}\leq 3^{-1}$ to get the bounds
\begin{align*}
\| \mathcal{N}^{(J;\mathcal{T})}_R[\varpi]\|_{\hat{H}^{\frac12}}&\leq M\big( 10CM^{-\frac12+}\| \varpi\|_{\hat{H}^{\frac12}}^2\big)^J \Big( \prod_{j=1}^{J}3^{j-1}\Big)^{-1}\| \varpi\|_{\hat{H}^{\frac12}},\\
\| \mathcal{N}^{(J;\mathcal{T})}_0[\varpi]\|_{\hat{H}^{\frac12}}&\leq \big( CM^{-\frac12+}\| \varpi\|_{\hat{H}^{\frac12}}^2\big)^J \Big( \prod_{j=1}^{J}3^{j-1}\Big)^{-1}\| \varpi\|_{\hat{H}^{\frac12}}. 
\end{align*}
This completes the proof of Lemma~\ref{lem:NFR-NR}.
\end{proof}

The proof of the estimate of $\mathcal{N}^{(J;\mathcal{T})}_0[\varpi]$ in Lemma~\ref{lem:NFR-NR} immediately implies the following:
\begin{lem}\label{lem:NFR-N1R}
We have
\begin{align*}
\Big( \prod _{j=1}^J3^{j-1}\Big) \sum_{a^*\in \mathcal{T}^\infty}\| \mathcal{N}^{(J;\mathcal{T},a^*)}_1[\varpi,\zeta]\|_{L^1_T\hat{H}^{\frac12}}&\leq (CM^{-\frac12+}\| \varpi \|_{L^\infty_T\hat{H}^{\frac12}}^2)^J\| \zeta \|_{L^1_T\hat{H}^{\frac12}}.
\end{align*}
\end{lem}

We conclude this subsection with the following weak ($\hat{H}^{-\frac12}$) estimate of $\mathcal{N}^{(J;\mathcal{T})}[\varpi]$:
\begin{lem}\label{lem:NFR-weak}
We have
\[ \Big( \prod_{j=1}^J3^{j-1}\Big) \| \mathcal{N}^{(J;\mathcal{T})}[\varpi ]\|_{L^\infty_T\hat{H}^{-\frac12}}\leq C(CM^{-\frac12+}\| \varpi \|_{L^\infty_T\hat{H}^{\frac12}}^2)^{J-1}\| \varpi \|_{L^\infty_T\hat{H}^{\frac12}}^3.\]
\end{lem}

\begin{proof}
The proof is similar to the estimate of $\mathcal{N}^{(J;\mathcal{T})}_R[\varpi]$ in Lemma~\ref{lem:NFR-NR}.
By Cauchy-Schwarz and \eqref{est:weight-J}, we have
\begin{align*}
\| \mathcal{N}^{(J;\mathcal{T})}[\varpi ]\|_{\hat{H}^{-\frac12}}
&\leq \| \varpi\|_{\hat{H}^{\frac12}}^{2J+1}\sup_{\xi\in \hat{\mathcal{M}}}\Big( \int_{\bm{\xi}\in \Xi_\xi(\mathcal{T})}\frac{1}{\langle \xi_{a^1}\rangle \prod\limits_{a\in \mathcal{T}^\infty}\langle \xi_a\rangle}\prod_{j=1}^{J}\Psi_j^2|\xi_{a^j_3}|^2\cdot \prod_{j=1}^{J-1}\frac{\mathbf{1}_{|\tilde{\phi}_j|>M_j}}{|\tilde{\phi}_j|^2}\Big) ^{\frac12} \\
&\leq C^J\| \varpi\|_{\hat{H}^{\frac12}}^{2J+1}\sup_{\xi\in \hat{\mathcal{M}}}\Big( \int_{\bm{\xi}\in \Xi_\xi(\mathcal{T})}\frac{1}{\langle \xi_{a^1}\rangle^2}\prod_{j=1}^{J}\Psi_j\cdot \prod_{j=1}^{J-1}\frac{\mathbf{1}^{N\!R}_{|\tilde{\phi}_j|>M_j}}{|\tilde{\phi}_j|^2}\Big) ^{\frac12} .
\end{align*}
Now, we have $\langle \xi_{a^J_l}\rangle \leq C^J\langle \xi_{a^1}\rangle$ for $l=1,2$ in the support of $\prod_{j=1}^J\Psi_j$, and hence
\[ \frac{1}{\langle \xi_{a^1}\rangle^2}\int_{\xi_{a^J}=\xi_{a^J_1a^J_2a^J_3}}\prod_{j=1}^{J}\Psi_j\leq 4C^{2J}.\]
Using this inequality instead of 
\[ \int_{\xi_{a^J}=\xi_{a^J_1a^J_2a^J_3}} \mathbf{1}^{N\!R}_{|\tilde{\phi}_{J-1}+\phi_J|\leq M_J}\lesssim M_J^{1+}\langle \tilde{\phi}_{J-1}\rangle ^{0+} \]
for the previous estimate of $\mathcal{N}^{(J;\mathcal{T})}_R[\varpi]$, we see that
\[ \sup_{\xi\in \hat{\mathcal{M}}}\Big( \int_{\bm{\xi}\in \Xi_\xi(\mathcal{T})}\frac{1}{\langle \xi_{a^1}\rangle^2}\prod_{j=1}^{J}\Psi_j\cdot \prod_{j=1}^{J-1}\frac{\mathbf{1}^{N\!R}_{|\tilde{\phi}_j|>M_j}}{|\tilde{\phi}_j|^2}\Big) ^{\frac12} \leq C^J\prod_{j=1}^{J-1}M_j^{-\frac12+}\leq C^J\big( M^{-\frac12+}\big)^{J-1}\Big( \prod_{j=1}^{J-1}3^{j-1}\Big) ^{-1}.\]
This implies the desired estimate.
\end{proof}


\subsection{Justification of formal computations}
\label{subsec:just-infinite}

Here, we give a justification to the derivation of each generation \eqref{eq:omega-J} and the limit equation \eqref{eq:limit}.
The basic idea is the same as in Subsection~\ref{subsec:just}.
According to Corollary~\ref{cor:eq-varpi}, let $\varpi\in L^\infty_T\hat{H}^{\frac12}\cap C_T\hat{H}^{\frac12-}\cap W^{1,2}_T\hat{H}^{-\frac12}$ satisfy the equation (of the first generation) \eqref{eq:varpi} and its differential form \eqref{eq:varpi-d}.
We regard $\mathcal{W}_1[\omega]$ as a given forcing term in $L^2_T\hat{H}^{\frac12}$.

Assuming that the $J$-th generation equation is given, let us see how to justify the transformation of the last term $\int_0^t\mathcal{N}^{(J;\mathcal{T})}[\varpi(t')]$ for each fixed $\mathcal{T}\in \mathfrak{T}(J)$.
Since all the frequencies are comparable in this term, it is quite easy to see that the integral in $(t',\bm{\xi})$ is absolutely convergent for any $(t,\xi)\in [-T,T]\times \hat{\mathcal{M}}$:
\begin{equation}\label{est:integrable}
\begin{aligned}
&\int_0^t\int_{\bm{\xi}\in \Xi_\xi(\mathcal{T})}\bigg| \prod_{j=1}^J\Psi_j\xi_{a^j_3}\cdot \prod_{j=1}^{J-1}\frac{\mathbf{1}_{|\tilde{\phi}_j|>M_j}}{\tilde{\phi}_j}\cdot \prod_{a\in \mathcal{T}^\infty}\varpi(t',\xi_a)\bigg| \,dt' \\
&\lesssim \langle \xi\rangle^{J+} \int_{-T}^T\int_{\bm{\xi}\in \Xi_\xi(\mathcal{T})}\prod_{a\in \mathcal{T}^\infty}\langle \xi_a\rangle^{0-}|\varpi(t',\xi_a)|\,dt' \\
&\lesssim \langle \xi\rangle^{J+} \int_{-T}^T\big\| \langle \xi \rangle^{0-}\varpi(t')\big\|_{L^1}^{2J-1} \big\| \langle \xi \rangle^{0-}\varpi(t')\big\|_{L^2}^2\,dt' \\
&\lesssim \langle \xi\rangle^{J+}T\| \varpi\|_{L^\infty_T\hat{H}^{\frac12}}^{2J+1}<\infty .
\end{aligned}
\end{equation}
We divide the $\bm{\xi}$-integral, and for the non-resonant part
\begin{gather*}
\int_0^t\int_{\bm{\xi}\in\Xi_\xi(\mathcal{T})}e^{it'\tilde{\phi}_J}f(t',\bm{\xi})\,dt',\qquad
f(t,\bm{\xi}):=\prod_{j=1}^J\Psi_j\xi_{a^j_3}\cdot \prod_{j=1}^{J-1}\frac{\mathbf{1}_{|\tilde{\phi}_j|>M_j}}{\tilde{\phi}_j}\cdot \mathbf{1}_{|\tilde{\phi}_J|>M_J}\prod_{a\in \mathcal{T}^\infty}\varpi(t,\xi_a),
\end{gather*}
we justify the integration by parts operation in the following lemma:
\begin{lem}\label{lem:just2}
We have
\[ \int_0^t\int_{\bm{\xi}\in \Xi_\xi(\mathcal{T})}e^{it'\tilde{\phi}_J}f(t',\bm{\xi})\,dt'=\int_{\bm{\xi}\in \Xi_\xi(\mathcal{T})}\frac{e^{it'\tilde{\phi}_J}}{i\tilde{\phi}_J}f(t',\bm{\xi})\bigg|_0^t-\int_0^t\int_{\bm{\xi}\in \Xi_\xi(\mathcal{T})}\frac{e^{it'\tilde{\phi}_J}}{i\tilde{\phi}_J}g(t',\bm{\xi})\,dt' \]
for any $(t,\xi)\in [-T,T]\times \hat{\mathcal{M}}$, where
\[ g(t,\bm{\xi}):=\prod_{j=1}^J\Psi_j\xi_{a^j_3}\cdot \prod_{j=1}^{J-1}\frac{\mathbf{1}_{|\tilde{\phi}_j|>M_j}}{\tilde{\phi}_j}\cdot \mathbf{1}_{|\tilde{\phi}_J|>M_J}\sum _{a^*\in\mathcal{T}^\infty} (\partial_t\varpi)(t,\xi_{a^*})\prod_{a\in \mathcal{T}^\infty\setminus\{ a^*\}}\varpi(t,\xi_a). \]
\end{lem}

\begin{proof}
The proof is parallel to that of Lemma~\ref{lem:just}, so we omit the detailed estimates.
We shall prove the convergences corresponding to \eqref{conv1}--\eqref{conv3}.
For \eqref{conv1}, it suffices to show
\[ \int_{-T}^T\int_{\bm{\xi}\in \Xi_\xi(\mathcal{T})}|f(t',\bm{\xi})|\,dt' <\infty ,\]
which we have already seen in \eqref{est:integrable}.
For \eqref{conv2}, it suffices to show the continuity of the map
\[ [-T,T]\ni \ t\quad \mapsto \quad \langle \xi\rangle^{-J-} \int_{\bm{\xi}\in \Xi_\xi(\mathcal{T})}\frac{e^{it\tilde{\phi}_J}}{i\tilde{\phi}_J}f(t,\bm{\xi}) \ \in L^\infty_\xi .\]
This is also verified by estimating as \eqref{est:integrable}; we slightly modify it to have the $\hat{H}^{\frac12-}$ norm for $\varpi(t)$'s, so that we can use the continuity of $\varpi(t)$ in $\hat{H}^{\frac12-}$.
The convergence corresponding to \eqref{conv3} can be shown in a similar manner; this time we multiply the integral by $\langle \xi\rangle^{-J-1-}$ and argue as \eqref{est:integrable} to measure $\frac{1}{h}(\varpi(t'+h)-\varpi(t'))-(\partial_t\varpi)(t')$ and $(\partial_t\varpi)(t')$ in $L^2_T\hat{H}^{-\frac12}$ while $\varpi(t'+h)-\varpi(t')$ in $L^\infty_T\hat{H}^{\frac12-}$.
\end{proof}

Inserting \eqref{eq:varpi-d}, we have so far justified the equality
\begin{align*}
\int_0^t\mathcal{N}^{(J;\mathcal{T})}[\varpi(t')]\,dt'
&=\int_0^t\mathcal{N}_R^{(J;\mathcal{T})}[\varpi(t')]\,dt'+\mathcal{N}_0^{(J;\mathcal{T})}[\varpi(t')]\bigg|_0^t\\
&\quad +\sum_{a^*\in\mathcal{T}^\infty}\Big\{ \begin{aligned}[t]
&\int_0^t\mathcal{N}_1^{(J;\mathcal{T},a^*)}[\varpi(t'),\mathcal{W}_1(t')]\,dt'\\
&+\int_0^t\mathcal{N}_1^{(J;\mathcal{T},a^*)}[\varpi(t'),\mathcal{N}^{\texttt{A}}[\varpi(t')]]\,dt'\Big\} 
\end{aligned}
\end{align*}
pointwise for $(t,\xi)\in [-T,T]\times \hat{\mathcal{M}}$.
Each term on the right-hand side is well-defined for any $(t,\xi)$ as an absolutely convergent integral in $(t',\bm{\xi})$ or in $\bm{\xi}$.
The last term can be rewritten as $\int_0^t\mathcal{N}^{(J+1;\tilde{\mathcal{T}}(a^*))}[\varpi(t')]\,dt'$, which is an absolutely convergent integral over the index functions on the extended tree $\tilde{\mathcal{T}}(a^*)\in \mathfrak{T}(J+1)$, by the estimate \eqref{est:integrable} for $\tilde{\mathcal{T}}(a^*)$ and Fubini's theorem.
Repeating this procedure, we justify the equation \eqref{eq:omega-J} for any generation $J\geq 1$.

Finally, we assume that $M>1$ is large and $T\in (0,1]$ is small so that
\begin{equation}\label{ass:M}
M^{-\frac12+}\| \varpi\|_{L^\infty_T\hat{H}^{\frac12}}^2\ll 1,\qquad TM\leq 1.
\end{equation}
Then, by Lemmas~\ref{lem:NFR-NR}, \ref{lem:NFR-N1R}, \ref{lem:NFR-weak}, as $J\to \infty$ the second line of \eqref{eq:omega-J} converges in $L^\infty_T\hat{H}^{\frac12}$ and the third line vanishes in $L^\infty_T\hat{H}^{-\frac12}$.
This verifies that the limit equation \eqref{eq:limit} holds pointwise in $t\in [-T,T]$ as $\hat{H}^{\frac12}$-valued functions.


\subsection{Proof of the main theorem}
\label{subsec:proof}

We are now in a position to prove Theorem~\ref{thm:main}.
\begin{proof}[Proof of Theorem~\ref{thm:main}]
As we mentioned before, it suffices to prove uniqueness of the (distributional) solution $w$ to \eqref{eq:w2} in the class $L^\infty_TH^{\frac12}\cap L^4_TB^{0+}_{\infty,1}$.
Moreover, it is enough to show the uniqueness on the interval $[-T',T']$ for some $T'\in (0,T]$ (we write $T'$ as $T$ below).
Let $w_1,w_2$ be two such solutions with $w_1(0)=w_2(0)$, and define $\omega_m:=\mathcal{F}S(-t)w_m$ and $\varpi_m:=\omega_m-\mathcal{W}_0[\omega_m]$, $m=1,2$.
Then, each of $\varpi_m$ is a solution of \eqref{eq:varpi} as stated in Corollary~\ref{cor:eq-varpi}.
Choose $M\gg 1$ and $0<T\ll 1$ according to $\| w_m\|_{S_T}$, $m=1,2$ so that
\begin{equation}\label{ass:M2} 
TM\leq 1,\qquad \rho_0:=C\sum_{m=1,2}\big( M^{-\frac12+}\| w_m\|_{S_T}^2+T^{\frac12}(1+\| w_m\|_{S_T}^2)\| w_m\|_{S_T}^2\big) \ll 1.
\end{equation}
Then, from the estimate \eqref{est:W00} in Lemma~\ref{lem:W0} (and its difference version) we have
\begin{gather*}
\frac12\| w_m\|_{L^\infty_TH^{\frac12}}\leq \| \varpi_m\|_{L^\infty_T\hat{H}^{\frac12}}\leq 2\| w_m\|_{L^\infty_TH^{\frac12}},\\
\frac12\| w_1-w_2\|_{L^\infty_TH^{\frac12}}\leq \| \varpi_1-\varpi_2\|_{L^\infty_T\hat{H}^{\frac12}}\leq 2\| w_1-w_2\|_{L^\infty_TH^{\frac12}}.
\end{gather*}
In particular, the assumption \eqref{ass:M2} implies \eqref{ass:M} for $\varpi_m$, $m=1,2$, and thus $\varpi_m$ satisfies the limit equation \eqref{eq:limit}.
We now take the difference of the limit equation for $\varpi_1$, $\varpi_2$ and apply the difference versions of Lemmas~\ref{lem:W1}, \ref{lem:NFR-NR} and \ref{lem:NFR-N1R}, to see that
\[ \| \varpi_1-\varpi_2\|_{L^\infty_T\hat{H}^{\frac12}}\lesssim \rho _0\big( \| w_1-w_2\|_{S_T}+\| \varpi_1-\varpi_2\|_{L^\infty_T\hat{H}^{\frac12}}\big) ,\]
and hence, 
\[ \| w_1-w_2\|_{L^\infty_TH^{\frac12}}\leq 2\| \varpi_1-\varpi_2\|_{L^\infty_T\hat{H}^{\frac12}}\lesssim \rho _0\| w_1-w_2\|_{S_T}.\]
On the other hand, from Proposition~\ref{prop:RS-w} we have
\[ \| w_1-w_2\|_{L^4_TB^{0+}_{\infty,1}}\leq CT^{\frac14-}\big( 1+\| w_1\|_{S_T}^4+\| w_2\|_{S_T}^4\big) \| w_1-w_2\|_{S_T} \leq \frac14 \| w_1-w_2\|_{S_T} \]
if $T$ is modified appropriately.
From these estimates, we have $\| w_1- w_2\|_{S_T}\leq \frac12\| w_1-w_2\|_{S_T}$, which establishes the uniqueness for \eqref{eq:w2} in $L^\infty_TH^{\frac12}\cap L^4_TB^{0+}_{\infty,1}$.
\end{proof}


\section*{Acknowledgements}
\addcontentsline{toc}{section}{Acknowledgements}
The author would like to thank the referees for their valuable comments.
This work was supported by JSPS KAKENHI Grant Numbers JP16K17626, JP20K03678.


\appendix
\section{Weak solutions}
\label{appendix:weak}

Let us consider the inhomogeneous linear Schr\"odinger equation
\begin{equation}\label{eq:app1}
\partial_tu(t,x)=i\partial_x^2u(t,x)+\partial_xF(t,x),\qquad (t,x)\in (-T,T)\times \mathcal{M},
\end{equation}
where $\mathcal{M}$ is either $\mathbb{R}$ or $\mathbb{T}$.

\begin{lem}\label{lem:weak-mild}
Let $s,\sigma\in \mathbb{R}$ satisfy $\sigma <s$.
Suppose that $u\in L^\infty_TH^s(\mathcal{M})\cap L^1_{\mathrm{loc}}((-T,T)\times \mathcal{M})$, $F\in L^1_{\mathrm{loc}}((-T,T)\times \mathcal{M})$, $\partial_xF\in L^\infty_TH^\sigma (\mathcal{M})$, and $u$, $F$ satisfy \eqref{eq:app1} in the sense of distributions; i.e., 
\begin{equation}\label{eq:app1weak}
-\int_{-T}^T\int_{\mathcal{M}}u(t,x)(\partial_t\psi)(t,x)\,dx\,dt =\int_{-T}^T\int_{\mathcal{M}}\Big\{ iu(t,x)(\partial_x^2\psi)(t,x)-F(t,x)(\partial_x\psi)(t,x)\Big\} \,dx\,dt 
\end{equation}
for any $\psi \in \mathcal{D}((-T,T)\times \mathcal{M})$.%
\footnote{%
Although we will use the same notation for $\mathcal{M}=\mathbb{R}$ and $\mathbb{T}$, when $\mathcal{M}=\mathbb{R}$ we actually consider the extended class of test functions $\mathcal{D}(-T,T;\mathcal{S}(\mathbb{R}))$, where $\mathcal{S}$ is the Schwartz space.
Then, the $L^1_{\mathrm{loc}}$ assumption on $u,F$ should be regarded as $L^1_{\mathrm{loc}}(-T,T;L^1_{\mathrm{poly}}(\mathbb{R}))$, where $L^1_{\mathrm{poly}}$ is the set of $L^1_{\mathrm{loc}}$ functions of at most polynomial growth.%
}
Then, there exists a (unique) function
\[ \tilde{u}\in C([-T,T];H^{s-}(\mathcal{M}))\cap C_w([-T,T];H^s(\mathcal{M}))\cap W^{1,\infty}(-T,T;H^{\sigma'}(\mathcal{M}))\]
(where $\sigma':=\min \{ s-2,\sigma \}$) such that
\begin{enumerate}
\item $\tilde{u}(t)=u(t)$ (in $H^s(\mathcal{M})$) for \textit{a.e.}~$t\in (-T,T)$.
\item $\tilde{u}$ and $F$ satisfy \eqref{eq:app1} in $H^{\sigma'}(\mathcal{M})$ for \textit{a.e.} $t\in (-T,T)$.
\item $\tilde{u}(t)=S(t-t_0)\tilde{u}(t_0)+\displaystyle\int_{t_0}^tS(t-t')\partial_xF(t')\,dt'$ holds in $H^{\sigma}(\mathcal{M})$ for any $t,t_0\in [-T,T]$.
\end{enumerate}
\end{lem}

\begin{proof}
First, we show the identity 
\[ -\int_{-T}^Tu(t)\eta'(t)\,dt =\int_{-T}^T\Big\{ i\partial_x^2u(t)+\partial_xF(t)\Big\} \eta(t)\,dt\]
for any $\eta \in \mathcal{D} (-T,T)$, where the integrals are considered as those of  $H^{\sigma'}(\mathcal{M})$-valued functions.
To see this, let $\phi \in \mathcal{D}(\mathcal{M})$, and observe that
\[ \Big\langle {-}\int_{-T}^Tu(t)\eta'(t)\,dt,\,\phi \Big\rangle =-\int_{-T}^T \big\langle u(t)\eta'(t),\,\phi \big\rangle \,dt=-\int_{-T}^T\int_{\mathcal{M}}u(t,x)(\partial_t\psi)(t,x)\,dx\,dt,\]
where $\psi(t,x):=\eta (t)\phi(x)\in \mathcal{D}((-T,T)\times \mathcal{M})$.
At the first equality, we have used the fact that $u(t)\eta'(t)\in L^1(-T,T;H^{\sigma'}(\mathcal{M}))$ and that the evaluation map $u\mapsto \langle u,\phi\rangle$ is a bounded linear functional on $H^{\sigma'}(\mathcal{M})$ and hence commutes with the $t$-integral. 
Similarly, we have
\begin{align*}
&\Big\langle \int_{-T}^T\big\{ i\partial_x^2u(t)+\partial_xF(t)\big\} \eta(t)\,dt,\,\phi \Big\rangle =\int_{-T}^T\big\langle i\partial_x^2u(t)+\partial_xF(t),\,\phi \big\rangle \eta(t)\,dt \\
&=\int_{-T}^T\Big\{ i\big\langle u(t),\phi''\big\rangle -\big\langle F(t),\,\phi'\big\rangle \Big\} \eta(t)\,dt\\
&=\int_{-T}^T\int_{\mathcal{M}}\Big\{ iu(t,x)(\partial_x^2\psi)(t,x)-F(t,x)(\partial_x\psi)(t,x)\Big\} \,dx\,dt,
\end{align*}
and hence, by \eqref{eq:app1weak} we have
\[ \Big\langle {-}\int_{-T}^Tu(t)\eta'(t)\,dt,\,\phi \Big\rangle =\Big\langle \int_{-T}^T\big\{ i\partial_x^2u(t)+\partial_xF(t)\big\} \eta(t)\,dt,\,\phi \Big\rangle \]
for any $\phi\in \mathcal{D}(\mathcal{M})$, showing the claim.
In particular, $u$ satisfies \eqref{eq:app1} as an $H^{\sigma'}(\mathcal{M})$-valued distribution on $(-T,T)$.
Since the right-hand side of \eqref{eq:app1} belongs to $L^\infty_TH^{\sigma'}(\mathcal{M})$, we see that $u\in W^{1,\infty}_TH^{\sigma'}(\mathcal{M})$ and \eqref{eq:app1} holds in $H^{\sigma'}(\mathcal{M})$ for \textit{a.e.} $t\in (-T,T)$.
Then, since $v(t):=u(t)-\int_0^t(\partial_tu)(t')\,dt'$ satisfies $\partial_tv=0$ in $L^\infty_TH^{\sigma'}(\mathcal{M})$, there exists $u_0\in H^{\sigma'}(\mathcal{M})$ such that $v(t)=u_0$ for \textit{a.e.} $t\in (-T,T)$.
Setting $\tilde{u}(t):=u_0+\int_0^t (\partial_tu)(t')\,dt'\in C_TH^{\sigma'}(\mathcal{M})$, we have $u(t)=\tilde{u}(t)$ for \textit{a.e.} $t\in (-T,T)$.

From now on, we write $\tilde{u}$ as $u$.
Recalling $u\in L^\infty_TH^s(\mathcal{M})$, let $J\subset (-T,T)$ denote the set of full measure in which $\| u(t)\|_{H^s}\leq \| u\|_{L^\infty_TH^s}$.
Given any $t\in [-T,T]$, we approximate it by $\{ t_n\}_n\subset J$.
Then, for any $\phi\in \mathcal{D}(\mathcal{M})$, $\{ (u(t_n),\phi)_{H^s}\}_n$ is a Cauchy sequence, by continuity of $u(t)$ in $H^{\sigma'}(\mathcal{M})$, and thus it converges to $(u(t),\phi)_{H^s}$.
This means that $|(u(t),\phi)_{H^s}|\leq \| u\|_{L^\infty_TH^s}\| \phi\|_{H^s}$ for any $\phi\in \mathcal{D}(\mathcal{M})$, and therefore $u(t)\in H^s(\mathcal{M})$, $\| u(t)\|_{H^s}\leq \| u\|_{L^\infty_TH^s}$ for all $t\in [-T,T]$ and $u\in C_{w,T}H^s(\mathcal{M})$.
By interpolation between $H^s$ and $H^{\sigma'}$, we also show that $u\in C_TH^{s-}(\mathcal{M})$.

Finally, set $w(t'):=S(-t')u(t')$ for $t'\in [-T,T]$.
For any $\eta \in \mathcal{D}(-T,T)$ and $\phi\in \mathcal{D}(\mathcal{M})$, we see that
\begin{align*}
&\Big\langle {-}\int _{-T}^Tw(t')\eta'(t')\,dt',\,\phi\Big\rangle =-\int_{-T}^T\big\langle u(t'),\,S(-t')\phi \big\rangle \eta'(t')\,dt' \\\
&=-\int_{-T}^T\int_{\mathcal{M}}u(t',x)\partial_{t'}\big[ \eta(t')S(-t')\phi\big] (x)\,dx\,dt +\int_{-T}^T \big\langle u(t'),\,-i\partial_x^2S(-t')\phi\big\rangle \eta(t')\,dt'\\
&=-\int_{-T}^T\int_{\mathcal{M}}F(t',x)\eta(t')\big[ \partial_xS(-t')\phi \big](x)\,dx\,dt =\Big\langle \int_{-T}^TS(-t')\partial_xF(t')\eta(t')\,dt',\,\phi\Big\rangle ,
\end{align*}
where we have used \eqref{eq:app1weak} with $\psi (t',x)=\eta(t')S(-t')\phi(x)$ at the third equality.
Since $S(-\cdot )\partial_xF(\cdot)\in L^\infty_TH^{\sigma}(\mathcal{M})$, this shows $w\in W^{1,\infty}_TH^\sigma (\mathcal{M})$ and $\partial_{t'}w(t')=S(-t')\partial_xF(t')$ in $L^\infty_TH^{\sigma}(\mathcal{M})$.
Similarly to the previous argument, for any $t_0\in [-T,T]$ we see that $z_{t_0}(t):=w(t)-\int_{t_0}^{t}(\partial_{t'}w)(t')\,dt'$ coincides with some $w_{t_0}\in H^\sigma(\mathcal{M})$ for \textit{a.e.} $t$. 
Since $w(t')=S(-t')u(t')\in C_TH^\sigma(\mathcal{M})$, we have $z_{t_0}(t)\equiv w_{t_0}=z_{t_0}(t_0)=w(t_0)$ (in $H^\sigma(\mathcal{M})$) on $[-T,T]$.
This can be rewritten as
\[ S(-t)u(t)-\int_{t_0}^{t}S(-t')\partial_xF(t')\,dt'=S(-t_0)u(t_0)\]
in $H^\sigma (\mathcal{M})$, for all $t,t_0\in [-T,T]$.
We obtain the desired integral equation for $u(t)$ by applying $S(t)$.
\end{proof}


\bibliographystyle{abbrv}
\bibliography{kishimoto}
\addcontentsline{toc}{section}{References}

\end{document}